\documentclass[11pt]{amsart}
\usepackage{amssymb}
\usepackage{verbatim}
\usepackage{hyperref}
\usepackage{color}

\begin{document}
\setcounter{tocdepth}{1}

\newtheorem{theorem}{Theorem}    
\newtheorem{proposition}[theorem]{Proposition}
\newtheorem{conjecture}[theorem]{Conjecture}
\def\theconjecture{\unskip}
\newtheorem{corollary}[theorem]{Corollary}
\newtheorem{lemma}[theorem]{Lemma}
\newtheorem{sublemma}[theorem]{Sublemma}
\newtheorem{fact}[theorem]{Fact}
\newtheorem{observation}[theorem]{Observation}
\theoremstyle{definition}
\newtheorem{definition}{Definition}
\newtheorem{notation}[definition]{Notation}
\newtheorem{remark}[definition]{Remark}
\newtheorem{question}[definition]{Question}
\newtheorem{questions}[definition]{Questions}

\newtheorem{example}[definition]{Example}
\newtheorem{problem}[definition]{Problem}
\newtheorem{exercise}[definition]{Exercise}

\numberwithin{theorem}{section}
\numberwithin{definition}{section}
\numberwithin{equation}{section}

\def\reals{{\mathbb R}}
\def\torus{{\mathbb T}}
\def\heis{{\mathbb H}}
\def\integers{{\mathbb Z}}
\def\rationals{{\mathbb Q}}
\def\naturals{{\mathbb N}}
\def\complex{{\mathbb C}\/}
\def\distance{\operatorname{distance}\,}
\def\sym{\operatorname{Symm}\,}
\def\support{\operatorname{support}\,}
\def\dist{\operatorname{dist}}
\def\Span{\operatorname{span}\,}
\def\degree{\operatorname{degree}\,}
\def\kernel{\operatorname{kernel}\,}
\def\dim{\operatorname{dim}\,}
\def\codim{\operatorname{codim}}
\def\trace{\operatorname{trace\,}}
\def\Span{\operatorname{span}\,}
\def\dimension{\operatorname{dimension}\,}
\def\codimension{\operatorname{codimension}\,}
\def\Gl{\operatorname{Gl}}
\def\nullspace{\scriptk}
\def\kernel{\operatorname{Ker}}
\def\Re{\operatorname{Re} }
\def\Im{\operatorname{Im} }
\def\eps{\varepsilon}
\def\lt{L^2}
\def\diver{\operatorname{div}}
\def\curl{\operatorname{curl}}
\newcommand{\norm}[1]{ \|  #1 \|}
\def\expect{\mathbb E}
\def\bull{$\bullet$\ }
\def\det{\operatorname{det}}
\def\Det{\operatorname{Det}}
\def\bestA{\mathbf A}
\def\bestC{\mathbf C}
\def\bestAqd{\mathbf A_{q,d}}
\def\bestBqd{\mathbf B_{q,d}}
\def\bestB{\mathbf B}
\def\bestC{\mathbf C}
\def\Apq{\mathbf A_{p,q}}
\def\Apqr{\mathbf A_{p,q,r}}
\def\rank{\operatorname{rank}}
\def\diameter{\operatorname{diameter}}
\def\essinf{\operatorname{ess\,inf}}
\def\esssup{\operatorname{ess\,sup}}

\newcommand{\abr}[1]{ \langle  #1 \rangle}
\def\unitQ{{\mathbf Q}}
\def\mbfp{{\mathbf P}}

\def\aff{\operatorname{Aff}}
\def\T{{\mathcal T}}

\newcommand{\Norm}[1]{ \Big\|  #1 \Big\| }
\newcommand{\set}[1]{ \left\{ #1 \right\} }
\newcommand{\sset}[1]{ \{ #1 \} }

\def\one{{\mathbf 1}}
\def\bb{{{\mathbb B}}}
\def\onei{{\mathbf 1}_I}
\def\onee{{\mathbf 1}_E}
\def\onea{{\mathbf 1}_A}
\def\oneb{{\mathbf 1}_B}
\def\onebb{{\mathbf 1}_{\bb}}
\def\wonee{\widehat{\mathbf 1}_E}
\newcommand{\modulo}[2]{[#1]_{#2}}

\def\repair{\medskip\hrule\hrule\medskip}

\def\scriptf{{\mathcal F}}
\def\scripts{{\mathcal S}}
\def\scriptq{{\mathcal Q}}
\def\scriptg{{\mathcal G}}
\def\scriptm{{\mathcal M}}
\def\scriptb{{\mathcal B}}
\def\scriptc{{\mathcal C}}
\def\scriptt{{\mathcal T}}
\def\scripti{{\mathcal I}}
\def\scripte{{\mathcal E}}
\def\scriptv{{\mathcal V}}
\def\scriptw{{\mathcal W}}
\def\scriptu{{\mathcal U}}
\def\scripta{{\mathcal A}}
\def\scriptr{{\mathcal R}}
\def\scripto{{\mathcal O}}
\def\scripth{{\mathcal H}}
\def\scriptd{{\mathcal D}}
\def\scriptl{{\mathcal L}}
\def\scriptn{{\mathcal N}}
\def\scriptp{{\mathcal P}}
\def\scriptk{{\mathcal K}}
\def\scriptP{{\mathcal P}}
\def\scriptj{{\mathcal J}}
\def\scriptz{{\mathcal Z}}
\def\frakv{{\mathfrak V}}
\def\frakE{{\mathfrak E}}
\def\frakG{{\mathfrak G}}
\def\frakA{{\mathfrak A}}
\def\frakB{{\mathfrak B}}
\def\frakC{{\mathfrak C}}
\def\frakf{{\mathfrak F}}
\def\fraki{{\mathfrak I}}
\def\fcross{{\mathfrak F^{\times}}}

\def\symdif{\,\Delta\,}
\def\defe{\dist(E,\frakE)}
\def\defb{|E\symdif \bb|}
\def\Edagger{E^\dagger}

\def\barq{\bar q}

\def\br{{\mathbf r}}
\def\bx{{\mathbf x}}
\def\by{{\mathbf y}}
\def\bv{{\mathbf v}}
\def\bG{{\mathbf G}}
\def\bE{{\mathbf E}}
\def\bEstar{{\mathbf E^\star}}
\def\Psharp{P^\sharp}
\def\Gje{G_{\text{e}}}
\def\Gjo{G_{\text{o}}}

\author{Michael Christ}
\address{
        Michael Christ\\
        Department of Mathematics\\
        University of California \\
        Berkeley, CA 94720-3840, USA}
\email{mchrist@berkeley.edu}

\date{April 10, 2015. Revised March 24, 2017.}
\thanks{Research supported by NSF grant DMS-1363324.}

\title[On an extremization problem]
{On an extremization problem  \\ concerning Fourier coefficients}

\begin{abstract} 
Among subsets of Euclidean space with prescribed measure, for which sets
is the $L^q$ norm of the Fourier transform of the indicator function maximized? 
Various partial results concerning this question are established, including the existence
of maximizers and the identification of maximizers as ellipsoids
for exponents sufficiently close to even integers.
\end{abstract}

\maketitle
\tableofcontents

\section{Introduction}

\subsection{An extremization problem}
Consider the Fourier transform
\[ \widehat{f}(\xi) = \int_{\reals^d} e^{-2\pi i x\cdot\xi}f(x)\,dx\]
of a function $f:\reals^d\to\complex$.
Let $q\in(2,\infty)$
and let $p=q'=q/(q-1)\in(1,2)$ be the exponent conjugate to $q$.
The sharp Hausdorff-Young inequality of Beckner \cite{beckner} states that
\begin{equation} \norm{\widehat{f}}_q\le \bestC_q^d\norm{f}_{q'} \end{equation}
for all $f\in L^{q'}(\reals^d)$, where $\bestC_q = p^{1/2p}q^{-1/2q}$. 
The constant $\bestC_q^d$ is optimal. Lieb \cite{liebgaussian} has shown that
$f$ maximizes the ratio $\norm{\widehat{f}}_q/\norm{f}_{q'}$ if and only if
$f$ is a Gaussian function $c\exp(-Q(x,x)+v\cdot x)$,
where $Q$ is a positive definite real quadratic form, $v\in\complex^d$,
and $c\in\complex$.

In this paper we pose a variant extremization problem 
by introducing a constraint.
Denote by $\one_E$ the indicator function of a set $E$. 
Which sets
$E\subset\reals^d$, if any,
maximize the ratio $\norm{\widehat{\onee}}_q/\norm{\onee}_{q'}$?
That the supremum of this ratio is strictly less than $\bestC_q^d$
is a consequence of \cite{christHY}.

By a subset of $\reals^d$ or of $S^{d-1}$ we will always
mean a Lebesgue measurable subset. 
Two Lebesgue measurable sets 
are considered to be equivalent if their symmetric difference is a Lebesgue null set.
The ratio $\norm{\widehat{\onee}}_q/\norm{\onee}_{q'}$ respects this equivalence relation.
Throughout the discussion, we do not distinguish sets from equivalence classes of sets
under this relation.

The exponents
$q=2$ and $q=\infty$ are exceptional, in that all sets of given measure
yield a common value for the ratio. 
For $q=2$ this is Plancherel's Theorem, while for $q=1$ there is the relation
$\norm{\widehat{\one_E}}_\infty = \widehat{\one_E}(0)=|E|$.

Exponents $q\in\set{4,6,8,\dots}$ are known to be exceptional in other respects.
For instance, the optimal constant in the Hausdorff-Young inequality was found
first by Babenko \cite{babenko} for these exponents.
More directly relevant to our topic is the 
link between the Fourier transform and convolution, whereby
the Fourier transform can be eliminated from the discussion 
for $q\in\{4,6,8,\dots\}$ via the identity
\begin{equation} \label{eq:removeFT}
\norm{\widehat{\onee}}_q^q = \norm{\onee*\onee*\cdots*\onee}_2^2, \end{equation} 
where $*$ denotes convolution of functions and there are $q/2$ factors $\one_E$ in the convolution. 
No similar elimination of the Fourier transform is available for any other exponent in $(2,\infty)$.

The Riesz-Sobolev inequality implies that 
among all sets $E$ of specified measure, the quantity $\norm{\onee*\onee*\cdots*\onee}_2^2$
is maximized by balls, and hence, by the affine invariance discussed below,
by arbitrary ellipsoids.  Therefore ellipsoids are among the maximizers $\norm{\widehat{\one_E}}_q/|E|^{1/q'}$
for these particular exponents.  This reduction leads to a complete answer to our question
for these same exponents, through work of Burchard \cite{burchard}, who has shown that 
$\norm{\onee*\onee*\cdots*\onee}_2^2$ is maximized, among all sets $E$ of specified Lebesgue measure, 
only by ellipsoids.

It is natural to ask:
Might ellipsoids be extremizers for other exponents, perhaps for all exponents in $(2,\infty)$? 
There are grounds for caution.
Indeed, for typical exponents 
the mapping $f\mapsto \norm{\widehat{f}}_q$ lacks at least two
important monotonicity properties that hold for $q\in\{4,6,8,\dots\}$.
Let $q\ge 2$, and for the sake of
simplicity of the statements consider only functions in $L^1\cap L^2$.  Then:
\newline
(i) The inequality $\norm{\widehat{f}}_q\le \norm{\widehat{\,|f|\,}}_q$ holds
for all functions $f$, if and only if $q\in\{2,4,6,\dots\}$.
The failure of this inequality for $q=3$
was shown by Hardy and Littlewood, and subsequently was proved by Boas \cite{boas}
for all other exponents not in $\{2,4,6,\dots\}$. 
Green and Ruzsa \cite{greenruzsa}, 
and independently
Mockenhaupt and Schlag \cite{schlagmockenhaupt}, 
have established more quantitative results in this direction.
\newline
(ii)
Let $f\ge 0$, let $p=q'$, and
let $u(t,x)^p$ satisfy the heat equation
$\partial_t u^p = \Delta_x u^p$ with $u(0,\cdot)=f$. 
Then $\norm{u(t,\cdot)}_p\equiv \norm{f}_p$ for all $t\ge 0$.
For $q\in\set{2,4,6,\dots}$,
the norm $\norm{\widehat{u(t,\cdot)}}_q$ is a nondecreasing function of $t$ under
this flow for arbitrary initial data $f$. 
Babenko's inequality is a simple corollary of this monotonicity.
Bennett, Bez, and Carbery \cite{bennettbezcarbery}
have shown that for every exponent
$q\notin\set{2,4,6,\dots}$ there exist initial data for which this monotonicity fails.

In this paper we explore this extremization problem. 
We show that for many exponents, ellipsoids are indeed (global) maximizers, 
and moreover are the only maximizers.
We establish various related results. 
However, all of our global results rely in part on perturbative considerations.
We do not bring to light any general principle which would explain
such a result for arbitrary exponents, and the question remains open in general. 

If ellipsoids are indeed extremizers, then the optimal 
constant in the inequality satisfies
$\bestAqd^q = |\bb|^{-(q-1)} \int_{\reals^d} |\widehat{\one_{\bb}}|^q$
where $\bb$ is the unit ball in $\reals^d$. 
This author is not aware of an expression for the right-hand side in more elementary
terms for general dimensions and exponents.
It is the identity of the extremizing sets, rather than an elementary expression
for the values of these constants, that is of interest.

This paper is one of a series in which inverse theorems of additive combinatorics
are applied to affine-invariant inequalities. 
All of our global results rely on a compactness theorem,
whose proof relies in turn on such an inverse theorem.

The author is grateful to Dominique Maldague for a critical reading of the manuscript
and for useful advice on the exposition.

\section{Results}

\subsection{Some notation}

Two Lebesgue measurable sets are considered to be equivalent if the Lebesgue measure
of their symmetric difference vanishes. Likewise, two functions are equivalent if they
agree almost everywhere.
By equality of sets or functions we will always mean equivalence in this sense.
Thus two Lebesgue measurable functions are said to be disjointly supported if their product
vanishes almost everywhere. 

For $E\subset\reals^d$, we denote by $E^\star$ the closed ball in $\reals^d$
that is centered at $0$ and satisfies $|E^\star|=|E|$. 
$\abr{x} = (1+|x|^2)^{1/2}$ for $x\in\reals^d$.
Denote by $\bb$ the unit ball $\bb=\{x\in\reals^d: |x|\le 1\}$
in $\reals^d$, and by $\omega_d=|\bb|$ its Lebesgue measure.
$\unitQ^d$ denotes the unit cube \[\unitQ^d=\big\{x=(x_1,\dots,x_d)\in\reals^d: \text{ $0\le x_j\le 1$ 
for every $1\le j\le d$}\big\}.\]
$|E|$ denotes the Lebesgue measure of a set $E$. $A\bigtriangleup B$
denotes the symmetric difference $(A\setminus B)\cup (B\setminus A)$ of two sets.

Reflections of functions are denoted by
\begin{equation} \tilde g(x) = g(-x)\end{equation} 
for any function $g:\reals^d\to\complex$. 
We write $\langle f,g\rangle = \int_{\reals^d} f\,\overline{g}$;
in the analysis below, often both functions are real-valued.
$q'=\frac{q}{q-1}$ denotes the exponent conjugate to $q$,
throughout the paper.

\begin{definition}
For $d\ge 1$ and $q\in(2,\infty)$,
\begin{equation} \Phi_q(E) = |E|^{-1/q'}\norm{\widehat{\onee}}_q\end{equation}
and
\begin{equation}
\bestAqd = \sup_E \Phi_q(E).
\end{equation}
\end{definition}
\noindent 
Division by $|E|^{1/q'}$ normalizes $\Phi_q$,
making $\Phi_q(\tilde E)=\Phi_q(E)$ whenever $\tilde E$ is a dilate of $E$.
The supremum defining $\bestAqd$ is taken over all Lebesgue measurable sets $E\subset\reals^d$
with positive, finite Lebesgue measures.
The classical Hausdorff-Young inequality guarantees that $\bestAqd\le 1<\infty$;
the sharp Hausdorff-Young inequality of Beckner guarantees
that $\bestAqd\le \bestC_q^d=(p^{1/2p}/q^{1/2q})^d<1$;
the compactness theorem of 
\cite{christHY} for the Hausdorff-Young inequality
ensures strict inequality $\bestAqd < \bestC_q^d$.

$\aff(d)$ denotes the group of all affine automorphisms of $\reals^d$.
The functional $\Phi_q$ is affine-invariant in the sense that
\begin{equation} \Phi_q(T(E))=\Phi_q(E) \ \text{ for all $T\in\aff(d)$}\end{equation}
for all Lebesgue measurable sets $E\subset\reals^d$ with $|E|\in\reals^+$.
Consequently
\begin{equation} \bestAqd = |\bb|^{-1/q'} \sup_{|E|=\bb} \norm{\widehat{\one_E}}_q,  
\end{equation}
and $\norm{\widehat{\one_\scripte}}_q=\norm{\widehat{\one_\bb}}_q$
for any ellipsoid $\scripte$ satisfying $|\scripte|=|\bb|$.

Throughout the paper, $\varrho$ denotes a strictly positive quantity
which depends on the dimension $d$ and on the exponent $q$,
and which can be taken to be independent of $q$ so long as
$q$ is restricted to a compact subset of the indicated domain.

Let $\frakE$ denote the set of all ellipsoids $\scripte\subset \reals^d$.
For any Lebesgue measurable subset $E\subset\reals^d$
with $|E|\in(0,\infty)$
define
\begin{equation} \defe = \inf_{\scripte\in\frakE} \frac{|\scripte\symdif E|}{|E|} \end{equation}
where the infimum is taken over all ellipsoids satisfying $|\scripte|=|E|$.
Again, the normalizing factor makes $\defe$ invariant under the action of $\aff(d)$.
We emphasize that this notion of distance is not closely related to the Minkowski
distance between compact sets. 

The elementary proof of the following lemma is omitted.
\begin{lemma}
Let $d\ge 1$.
For any Lebesgue measurable set $E\subset\reals^d$,
$\defe=0$ if and only if there exists an ellipsoid $\scripte$
satisfying $|E\symdif\scripte|=0$.
\end{lemma}

The following two functions occur naturally in analysis of $\widehat{\one_E}$ for
sets $E$ having small symmetric difference with the unit ball $\bb$.
\begin{definition}
For $q\in(2,\infty)$ and $d\ge 1$,
functions $K_{q}$ and $L_{q}$ with domain $\reals^d$ are defined by
\begin{equation} \label{KLdefn}
\left\{ \begin{aligned}
&\widehat{K_{q}} = \widehat{\one_\bb}\,\,|\widehat{\one_\bb}|^{q-2}
\\& \widehat{L_{q}} = |\widehat{\one_\bb}|^{q-2}.
\end{aligned} \right. \end{equation}
\end{definition}

For any $d\ge 1$ and $q\in(2,\infty)$,
$K_{q}$ and $L_q$ are both locally integrable, real-valued, radially symmetric functions.
This is justified below.

Define
\begin{equation} \label{eq:qddefn} q_d= 4-2(d+1)^{-1}.  \end{equation}
This exponent satisfies $3\le q_d<4$ for all dimensions $d$. 
A restriction $q>q_d$ arises in parts of our analysis. 
It might perhaps be relaxed at some points, 
but certain aspects of the local analysis do break down in an
essential way for all $q\in(2,3)$ when $d=1$.

\subsection{Existence and precompactness}
The following precompactness statement, valid for arbitrary exponents $q\in(2,\infty)$, 
is a foundational result. 
$\bestAqd$ continues to denote the supremum over all sets $E$ 
of $\norm{\widehat{\one_E}}_q\,/\,|E|^{1/q'}$.
\begin{theorem}\label{thm:compactness}
Let $d\ge 1$ and $q\in(2,\infty)$.
Let $(E_\nu)$ be a sequence of Lebesgue measurable subsets of $\reals^d$ with $|E_\nu|\in\reals^+$.
Suppose that $\lim_{\nu\to\infty} \Phi_q(E_\nu)=\bestAqd$.
There exist a sequence $\nu_k$ of indices tending to infinity,
a Lebesgue measurable set $E\subset\reals^d$ satisfying $0<|E|<\infty$,
and a sequence $(T_k)$ of affine automorphisms of $\reals^d$
such that the sets $E_k^\sharp = T_k(E_{\nu_k})$ satisfy 
\begin{equation} \lim_{j\to\infty} |E_j^\sharp\symdif E|=0.\end{equation}
\end{theorem}

The existence of maximizers is a direct consequence.
\begin{corollary} \label{cor:maximizersexist}
For any $d\ge 1$ and any $q\in(2,\infty)$,
there exists a Lebesgue measurable set $E\subset\reals^d$ with $|E|\in\reals^+$
such that $\Phi_q(E)=\bestAqd$.
\end{corollary}

An additive combinatorial inverse theorem of 
Fre{\u\i}man
is central to the proof of Theorem~\ref{thm:compactness}. 
Its presence is somewhat obscured by the organization of the proof;
the central Lemmas~\ref{lemma:quasi1} and \ref{lemma:quasi2} of the present paper
are direct consequences of Proposition~6.4 of \cite{christHY},
whose proof in turn relies on the inverse theorem.

\subsection{First variation: a necessary condition}

We say that ellipsoids are local maximizers of $\Phi_q$ if
there exists $\eta>0$ such that 
$\Phi_q(\bb)\ge\Phi_q(E)$ for all Lebesgue measurable
sets $E \subset\reals^d$ satisfying $\defe<\eta$. 

\begin{proposition} \label{prop:firstvariation}
Let $d\ge 1$ and $q\in(2,\infty)$.
If ellipsoids are local maximizers of the functional $\Phi_q$ then
\begin{equation}  \label{Knecessarycondition}
\essinf_{|x|\le 1} K_{q}(x) \ \ge\  \esssup_{|x|\ge 1} K_{q}(x).  \end{equation}
\end{proposition}
Here $\essinf$ and $\esssup$ are the essential infimum and 
supremum, respectively, in the usual measure-theoretic sense.

For $d=1$ there is a partial converse; see Proposition~\ref{prop:locald1q>3} below.
For $d>1$ a second variation comes into play,
and the analysis is more intricate. 

We are not in possession of any formula for $K_{q}$, for general exponents,
which would make it possible to determine whether \eqref{Knecessarycondition} holds. 
Various partial results concerning its validity 
are established below and in the companion paper \cite{mcjw}.
Condition \eqref{Knecessarycondition} holds for all sufficiently large $q$; it holds for all
exponents sufficiently close to $\{4,6,8,\dots\}$; in dimension $d=1$
it holds for all exponents sufficiently close to $2$ 
and also for $q=3$ \cite{mcjw};
numerical evidence \cite{mcjw} suggests that it holds for all $q\in(2,\infty)$ for $d=1$; 
we are not aware of any case in which it fails.

The condition \eqref{Knecessarycondition} would hold if 
the restriction of the radially symmetric function 
$K_q$ to rays emanating from $0$ were strictly decreasing. 
This monotonicity holds for $q\in\{4,6,8,\dots\}$,
but numerical calculations \cite{mcjw} indicate that it is
false for general exponents, and in particular, for $q=3$ when $d=1$.

\subsection{A form of stability}
The next result asserts that if ellipsoids maximize the functional
$\Phi_q$ in a strong sense for some exponent $q$, then they also 
maximize $\Phi_r$ for all exponents $r$ sufficiently close to $q$.  
This principle will be used to analyze exponents close to $\{4,6,8,\dots\}$.

Recall that for $E\subset\reals^d$, 
$E^\star\subset\reals^d$ denotes the closed ball centered at the origin
satisfying $|E^\star|=|E|$.

\begin{proposition}\label{prop:taylor}
Let $d\ge 1$ and $q>q_d$.
Suppose that there exist $c,\delta>0$ such that
\[ \norm{\widehat{\one_E}}_q^q \le
\norm{\widehat{\one_{E^\star}}}_q^q -c\defe^2|E|^{q-1}\]
for all Lebesgue measurable sets $E\subset\reals^d$
satisfying $|E|\in(0,\infty)$ and $\defe\le\delta$.
Then for any exponent sufficiently close to $q$, 
the same inequality holds for some constants $c,\delta>0$.
\end{proposition}

Scaling considerations dictate the exponent $q-1$ attached to $|E|$ in a straightforward manner.
More worthy of note is the exponent $2$ attached to $\defe$. 

\subsection{Exponents close to even integers}

\begin{theorem} \label{thm:neareven}
Let $d\ge 1$.
For each even integer $m\in\{4,6,8,\dots\}$
there exists $\delta(m)>0$ such that the following three conclusions
 hold for all exponents satisfying $|q-m|\le\delta(m)$.
Firstly, ellipsoids are extremizers of $\Phi_q$, that is,
\begin{equation} \bestAqd = \Phi_q(\bb).  \end{equation}
Secondly, ellipsoids are the only extremizers; 
for any Lebesgue measurable set $E\subset\reals^d$ with $0<|E|<\infty$,
$\Phi_q(E) = \bestAqd$ if and only if $E$ is an ellipsoid.
Thirdly, there exists $c_{q,d}>0$ such that for every set $E\subset\reals^d$
with positive, finite Lebesgue measure,
\begin{equation}\label{eq:definitehessian2} \Phi_q(E) \le \bestAqd - c_{q,d} \defe^2.  \end{equation}
\end{theorem}

\subsection{Large exponents}

Ellipsoids are local maximizers of $\Phi_q$, in a strong sense, 
for all sufficiently large $q$.
\begin{theorem}\label{thm:local}
Let $d\ge 1$. There exists $Q<\infty$ such that for every exponent $q\in[Q,\infty)$,
there exist $c,\eta>0$ such that 
\begin{equation} \Phi_q(E) \le  \Phi_q(\bb) - c\defe^2 \end{equation}
for every set $E\subset\reals^d$ satisfying $0 \le \defe<\eta$.
\end{theorem}

\subsection{Dimension one}

Our partial results are most comprehensive for dimension $d=1$.  
\begin{proposition} \label{prop:locald1q>3}
Let $d=1$.  If $q>3$ and if $K_{q}$ satisfies 
\begin{equation} \label{eq:Knecessarystronger}
\min_{|x|\le 1-\delta} K_{q}(x) > \max_{|x|\ge 1+\delta} K_{q}(x)
\ \text{for all $\delta>0$} \end{equation}
then intervals are local maximizers of $\Phi_q$, in the strong sense that 
there exists $c_q>0$ such that
\begin{equation} \Phi(E)\le \Phi(\bb) -c_q|\defe|^2 \end{equation}
whenever $\defe$ is sufficiently small.
\end{proposition}
This result will be extended to $q=3$, by an argument relying in part on numerical
calculations, in \cite{mcjw}.

The restriction to large exponents in Theorem~\ref{thm:local} is superfluous for $d=1$:
\begin{theorem}\label{thm:1D}
Let $d=1$.
For every $m\in\{4,6,8,\dots\}$
there exists $\delta(m)>0$ such that for all exponents satisfying $|q-m|\le\delta(m)$,
the following hold.
Firstly,
\begin{equation} {\mathbf A}_{q,1}= \Phi_q(\bb)
= 2^{-1/q'}\pi^{-1}\big(\int_{-\infty}^\infty |\xi^{-1}\sin(2\pi\xi)|^q\,d\xi\big)^{1/q}.  
\end{equation}
Secondly,
there exists $c_q>0$ such that for every such set $E$,
\begin{equation}\label{eq:definitehessian} \Phi_q(E) \le {\mathbf A}_{q,1} 
- c_q\defe^2.  \end{equation}
In particular,
for any Lebesgue measurable set $E\subset\reals^d$ with $0<|E|<\infty$,
$\Phi_q(E) = {\mathbf A}_{q,1}$ if and only if $E$ is an interval.
\end{theorem}

\subsection{The special case $q=4$ and $d=2$}
More detailed calculations are possible for $(q,d)=(4,2)$.
Recall that ${\mathbf A}_{m,d} = \Phi_m(\bb)$ for all dimensions $d$
and all $m\in\{4,6,8,\dots\}$, and in particular for $m=4$,
a direct consequence of the Riesz-Sobolev inequality.

The next two results concern the cases $q=4$, 
and $q\approx 4$, respectively, in dimension $d=2$.
\begin{theorem} \label{thm:q4d2}
Let $d=2$.
There exist $c,C\in(0,\infty)$ such that
for any Lebesgue measurable subset $E\subset\reals^2$ 
with $|E|\in\reals^+$,
\begin{equation} 
\norm{\widehat{\one_E}}_{L^4(\reals^2)}^4
\le {\mathbf A}_{4,2}^4|E|^3 - c \dist(E,\frakE)^2|E|^3. 
\end{equation} 
If $\dist(E,\frakE)$ is sufficiently small then
\begin{equation}
\norm{\widehat{\one_E}}_{L^4(\reals^2)}^4
\le {\mathbf A}_{4,2}^4|E|^3
- \gamma_{4,2} \dist(E,\frakE)^2|E|^3 + C\dist(E,\frakE)^{2+\varrho}|E|^3
\end{equation}
where $\gamma_{4,2} = \tfrac{8}{5}\pi^{-1}$ and $\varrho>0$.
\end{theorem}

\begin{theorem} \label{thm:2d}
Let $d=2$.
For all exponents sufficiently close to $4$,
\begin{equation} \bestA_{q,2} = \Phi_q(\bb).  \end{equation}
For all Lebesgue measurable sets $E\subset\reals^2$ with $0<|E|<\infty$,
$\Phi_q(E)=\bestA_{q,2}$ if and only if $E$ is an ellipsoid.
Moreover, 
\begin{equation}
\Phi_q(E) \le \bestA_{q,2} - c_0\dist(E,\frakE)^2
\end{equation}
where $c_0$ is an absolute constant.
Finally,
if $\dist(E,\frakE)$ is sufficiently small and if $q$ is sufficiently close to $4$ then
\begin{equation}
\norm{\widehat{\one_E}}_{L^q(\reals^2)}^q
\le {\mathbf A}_{q,2}^q |E|^{q-1}
- \gamma_{q,2} \dist(E,\frakE)^2 |E|^{q-1}
+ C\dist(E,\frakE)^{2+\varrho} |E|^{q-1}
\end{equation}
where $\gamma_{q,2} \to \tfrac85 \pi^{-1}$ as  $q\to 4$, and $\varrho>0$.
\end{theorem}

The proofs of Theorems~\ref{thm:q4d2} and \ref{thm:2d}
rely on an explicit calculation of the eigenvalues of a certain rotation-invariant
operator on the unit sphere $S^1\subset\reals^2$. Each
exponent and dimension give rise in the same way to an operator of this
type, on $S^{d-1}$, depending on both parameters. 
We are not able to carry out the corresponding calculations in general,
but nonetheless have established the required inequalities for the corresponding
eigenvalues by other means to prove Theorem~\ref{thm:neareven}.
Theorem~\ref{thm:local} is based on an asymptotic analysis of these eigenvalues as $q\to\infty$.

\subsection{Some variants} 
{\bf (i)}\  For $q\in(1,2)$, the relation between norms in the Hausdorff-Young inequality is reversed:
$\norm{\widehat{f}}_q \ge c_q^d\norm{f}_{q'}$.
This reversed inequality is equivalent, by duality, 
to the Hausdorff-Young inequality in the regime $q\ge 2$.
We assert that all of the above theorems have reversed versions for
exponents $q$ in the range $(1,2)$. However, these reversed versions are equivalent by duality
not to the theorems stated above, but to their analogues concerning the functional
\begin{equation}
\Phi^\sharp_q(f) = 
\frac{\norm{\widehat{f}}_{L^{q,\infty}}}  {\norm{f}_{q'}},
\end{equation}
where $L^{q,\infty}$ is the usual weak $L^q$ space
with the norm
\begin{equation}\label{eq:lorentznorm} 
\norm{g}_{L^{q,\infty}} = \sup_E |E|^{-1/q'} \big| \int_E g\big|.\end{equation}
This norm is equivalent, though not identical, 
to other norms commonly used for $L^{q,\infty}$ \cite{steinweiss}; duality
considerations impose the particular norm \eqref{eq:lorentznorm} on us.

These analogues follow from the methods developed here. 
Certain steps are carried out in this paper as parts of the proofs of the main theorems,
but these analogues are not fully proved here.

{\bf (ii)} Another quite natural variant is obtained by replacing indicator functions
of sets by bounded multiples, as follows. 
\begin{definition}
Let $f:\reals^d\to\complex$ and $E\subset\reals^d$
be a Lebesgue measurable function and a Lebesgue measurable set, respectively.
Then $f\prec E$ if $f=0$ almost everywhere on $\reals^d\setminus E$,
and $|f|\le 1$ almost everywhere  on $E$.
\end{definition}

\begin{definition}
For $d\ge 1$ and $q\in(2,\infty)$,
\begin{align} \Psi_q(E) &= 
\sup_{f\prec E}\  \frac{\norm{\widehat{f}}_q} {|E|^{1/q'}}
\\ \bestBqd &= \sup_E \Psi_q(E)
\end{align}
where the supremum is taken over all Lebesgue measurable sets $E\subset\reals^d$
with positive, finite Lebesgue measures.
\end{definition}

A direct consequence of their definitions is that $\bestAqd\le \bestBqd$.
Adapted from the periodic setting to $\reals$,
the example of Hardy and Littlewood demonstrates that $\Psi_3(E)$ is strictly
greater than $\Phi_3(E)$ for some sets $E$.
According to the results of subsequent authors, the same holds for all exponents $q\in(2,\infty)
\setminus\{4,6,8,\dots\}$.

In work in progress \cite{maldague}, D.~Maldague is developing
results for the functionals $\Psi_q$ parallel to those for $\Phi_q$.
For many exponents, $\bestAqd=\bestBqd$, and all 
extremizers take the form $e^{ix\cdot\xi}\one_\scripte$
where $\scripte$ is an ellipsoid and $\xi\in\reals^d$ is arbitrary.

{\bf (iii)}
For $q=2$, all sets are extremizers. Because the derivative with respect to $q$
of $\norm{\widehat{\one_E}}_q^q$, evaluated at $q=2$,
is $\Psi(E)=2\int |\widehat{\one_E}|^2\ln|\widehat{\one_E}|$,
it is natural to pose the same questions concerning $\Psi$ 
as for $\int|\widehat{\one_E}|^q$.
These questions are not addressed in this paper.

\section{Preliminaries}

\subsection{Continuous dependence of $\bestAqd$ on $q$}
Part of our analysis is perturbative with respect to the exponent $q$.
It will be important that $\bestAqd$ and $\norm{\widehat{\one_E}_q}$
depend continuously on $q$.

\begin{lemma}
Let $d\ge 1$ and $r\in(2,\infty)$.
As $E$ varies over all subsets of $\reals^d$ satisfying $|E|=1$,
the functions $q\mapsto  \norm{\widehat{\one_E}}_q$
form an equicontinuous family of functions of $q$ on any compact subset of $(2,\infty)$.
\end{lemma}

\begin{proof}
For all $y\in(0,1]$ and all $\theta\in(0,1]$,
\[y^\theta\le y + \ln(1/y)(1-\theta)\] 
because equality holds for $\theta=1$,
and the derivative of $y^\theta$ with respect to
$\theta$ has absolute value $\ln(1/y)y^\theta\le \ln(1/y)$.
If $\theta\in[\tfrac12,1]$ then also $y^\theta\le y^{1/2}\le y+y^{1/2}$,
so
\begin{align*} y^\theta &\le y + \min\big(\ln(1/y)(1-\theta),y^{1/2}\,\big) 
\\ &\le y + (1-\theta)^{1/2} \big( \ln(1/y) y^{1/2}\big)^{1/2}
\\& \le y + C(1-\theta)^{1/2}.\end{align*}

Let $r\in(2,\infty)$ be given and assume $|E|=1$.
Consider any $q\in[r,2r]$ and set $\theta = r/q\in[\tfrac12,1]$. Then
since $\norm{\widehat{\one_E}}_s\le 1$ for all $s\in[2,\infty]$,
\[ \norm{\wonee}_q\le \norm{\wonee}_r^\theta \norm{\wonee}_\infty^{1-\theta} \le \norm{\wonee}_r^\theta
\le \norm{\widehat{\one_E}}_r + C (1-\theta)^{1/2}
\le \norm{\widehat{\one_E}}_r + C_r |q-r|^{1/2}.\]
This same inequality holds for $q\in(2,r)$ sufficiently close to $r$.
Indeed, define $\theta$ by $q^{-1} =  \theta r^{-1} + (1-\theta) 2^{-1}$. 
Provided that $q$ is sufficiently
close to $r$ to ensure that $\theta\ge\tfrac12$, 
\begin{multline*} \norm{\wonee}_q\le \norm{\wonee}_r^\theta \norm{\wonee}_2^{1-\theta} 
= \norm{\wonee}_r^\theta |E|^{(1-\theta)/2}
\\ = \norm{\wonee}_r^\theta 
\le \norm{\widehat{\one_E}}_r + C|1-\theta|^{1/2}
\le \norm{\widehat{\one_E}}_r + C_r|q-r|^{1/2} .\end{multline*}

The constant $C_r$ is uniformly bounded for $r$ in any compact subinterval of $(2,\infty)$,
so these upper bounds can be reversed to yield lower bounds.
We have shown that
\[ \big| \norm{\wonee}_q-\norm{\wonee}_r \big|
= O(|q-r|^{1/2})\] uniformly for all sets satisfying $|E|=1$,
so long as $q,r$ vary over some compact subset of $(2,\infty)$.
\end{proof}

This equicontinuity has the following immediate consequence.
\begin{corollary}
For each dimension $d\ge 1$,
the mapping $ (2,\infty)\owns q\mapsto \bestAqd\in\reals^+$
is continuous.
\end{corollary}

The following variant will make possible a perturbation analysis with
respect to the exponent $q$.
\begin{corollary}\label{cor:away}
Let $d\ge 1$ and $q_0\in(2,\infty)$.
Let $S$ be a collection of Lebesgue measurable subsets $E\subset\reals^d$,
satisfying $|E|\in\reals^+$.
Let $\eta>0$, and suppose that 
$\Phi_{q_0}(E) \le {\mathbf A}_{q_0,d}-\eta$
for all $E\in S$.
Then there exists $\delta>0$ such that
\begin{equation} \Phi_{q}(E) \le \bestAqd - \tfrac12 \eta \end{equation}
for all $E\in S$ and all exponents $q$ satisfying $|q-q_0|<\delta$.
\end{corollary}

\begin{proof}
Because $\Phi_q$ is invariant under dilations, it suffices to prove this
under the hypothesis that $|E|=1$ for all $E\in S$. Then the conclusion 
follows from the preceding corollary and lemma.
\end{proof}

If $q_0$ is an exponent for which one knows that $\Phi_{q_0}$
attains its maximum value on ellipsoids and on no other sets,
then for any $\tau>0$,
one can  apply this corollary with $S$ equal to the collection of all sets
satisfying $\dist(E,\frakE)\ge \tau$ 
to conclude that only sets satisfying $\defe\le\tau$ are candidates
to be extremizers of $\Phi_q$ for $q$ close to $q_0$.


\subsection{Taylor expansion with respect to $E$}


Let $d\ge 1$ and $q\in(2,\infty)$.
Let $K_q,L_q$ be the functions defined in \eqref{KLdefn}, whose 
Fourier transforms are essentially powers of $|\widehat{\one_\bb}|$.

Much of our analysis will involve comparison of $\widehat{\one_E}$
with $\widehat{\one_\bb}$ for sets $E\subset\reals^d$ that are close to $\bb$.
Two senses of closeness arise naturally in this analysis;
first, the measure $|E\symdif\bb|$ of their symmetric difference may be small,
and second, $E\symdif\bb$ may be contained in a small neighborhood of the boundary of $\bb$.
To facilitate the comparison we use the function
\begin{equation} f=\one_E-\one_\bb.\end{equation} 
This notation $f$ will be employed throughout the discussion.
Recall the notations $\tilde g(x)=g(-x)$ and $\langle f,g\rangle = \int_{\reals^d} f\overline{g}$.
Convolution in $\reals^d$ will be denoted by $*$.

Recall that
$q_d= 4-2(d+1)^{-1}$, and that $3\le q_d<4$.
\begin{lemma} \label{lemma:expansion}
Let $d\ge 1$ and $q\in(3,\infty)$.
There exist $\varrho>0$ and $C<\infty$ with the following property.
Let $E\subset\reals^d$ and set $f=\one_E-\one_\bb$.
If $|E\symdif\bb|$ is sufficiently small then
\begin{multline} \label{eq:expansion}
\norm{\widehat{\one_E}}_q^q
= \norm{\widehat{\one_\bb}}_q^q
+q\langle K_{q},f\rangle
\\
+\tfrac14 q^2 \langle f*L_{q},f\rangle
+\tfrac14 q(q-2)  \langle f*L_{q},\tilde f\rangle
+O(|E\symdif\bb|^{2+\varrho}).
\end{multline}
\end{lemma}

So long as $q$ belongs to any compact subset of $(q_d,\infty)$,
the exponent $\varrho$ 
and the constant implicit in the notation $O(\cdot)$  in the term 
$O(|E\symdif\bb|^{2+\varrho})$ 
may be taken to be independent of $q$. 

Because we are dealing with a constrained optimization problem,
it will not the case that the first perturbation term
$q\langle K_{q},\,f\rangle$ vanishes when $\bb$ is an extremizer.
Nor is this term strictly of second order in $\defb$. 
It has a mixture of first order and second order aspects, and on the whole is rather negative.
On the other hand, the formally quadratic term
$\tfrac14 q^2 \langle f*L_{q},f\rangle
+\tfrac14 q(q-2)  \langle f*L_{q},\tilde f\rangle$
on the right-hand side of \eqref{eq:expansion}
will be nonnegative, to leading order. Thus in order to determine whether $\bb$ is a local
maximizer, we must determine whether the favorable term $q\langle K_{q},f\rangle$
is sufficiently negative to outweigh the unfavorable terms involving $L_{q}$.

\begin{proof}[Proof of Lemma~\ref{lemma:expansion}]
Let $q\in(3,\infty)$.
For any $t\in\complex$,
\begin{align*} |1+t|^q 
&= 1 + q\Re(t) + \tfrac12 q(q-1)\Re(t)^2 
+ \tfrac12 q \Im(t)^2 + O(|t|^3 + |t|^q).
\end{align*}
Indeed, if  $|t|\le\tfrac12$ this holds by the binomial theorem.
If $|t|\ge\tfrac12$ then 
\begin{equation*} |1+t|^q = O(|t|^q) = 1 + q\Re(t) + \tfrac12 q(q-1)\Re(t)^2 
+ \tfrac12 q \Im(t)^2 + O(|t|^q).  \end{equation*}
For $q\ge 2$ and $|t|\ge\tfrac12$, the quantities $1,|t|,|t|^2$ are all $O(|t|^q)$.

Since $\widehat{\one_\bb}$ is real-valued,
\begin{equation} \label{tobeintegrated} \begin{aligned}
|\widehat{\one_\bb}+\widehat{f}|^q
& =
|\widehat{\one_\bb}|^q
+ q\Re(\widehat{f}) {\widehat{\one_\bb}} |\widehat{\one_\bb}|^{q-2}
\\
& + \tfrac12 q (q-1) (\Re(\widehat{f}))^2  |\widehat{\one_\bb}|^{q-2}
+ \tfrac12 q  (\Im (\widehat{f}))^2 |\widehat{\one_\bb}|^{q-2}
\\& \qquad
+O(|\widehat{f}|^3 |\widehat{\one_{\bb}}|^{q-3}) 
+O(|\widehat{f}|^q). 
\end{aligned} \end{equation}


Since $f$, $K_{q}$, and $L_q$ are real-valued,
integrating the pointwise inequality \eqref{tobeintegrated} over $\reals^d$ 
and invoking Plancherel's theorem gives
\begin{align*}
\norm{\widehat{\one_E}}_q^q
& =
\norm{\widehat{\one_\bb}}_q^q
+q\langle K_{q},f\rangle
\\& \qquad 
+\tfrac12 q(q-1) \int (\Re\widehat{f})^2 |\widehat{\one_\bb}|^{q-2}
+ \tfrac12 q \int (\Im\widehat{f})^2 |\widehat{\one_\bb}|^{q-2}
\\&\qquad\qquad
+O(\norm{f}_{q'}^3\norm{\one_\bb}_{q'}^{q-3})
+O(\norm{f}_{q'}^q).
\end{align*}
Now $\norm{f}_{q'} = \defb^{1/q'}$.
For $q>3$, both $3/q'=3(q-1)/q = 3-3q^{-1}$ and $q/q'= q-1$ are strictly greater than $2$.
Therefore
\begin{align*}
\norm{\widehat{\one_E}}_q^q
& =
\norm{\widehat{\one_\bb}}_q^q
+q\langle K_{q},f\rangle
+\tfrac12 q(q-1) \int (\Re\widehat{f})^2 |\widehat{\one_\bb}|^{q-2}
\\& \qquad \qquad 
+ \tfrac12 q \int (\Im\widehat{f})^2 |\widehat{\one_\bb}|^{q-2}
+O(\defb^{2+\varrho})
\end{align*}
where $\varrho=\varrho(q)>0$.

Since $f$ is real-valued, $\widehat{\tilde f}$ is the complex conjugate of $\widehat{f}$ and therefore
\begin{align*}
\tfrac12 q(q-1) (\Re\widehat{f})^2 
+ \tfrac12 q (\Im\widehat{f})^2 
&=
\tfrac18 q(q-1) (\widehat{f}+\overline{\widehat{f}})^2 
- \tfrac18 q (\widehat{f}-\overline{\widehat{f}})^2 
\\&=
\tfrac18 q(q-2) \widehat{f}\,^2
+ \tfrac18 q(q-2) (\overline{\widehat{f}})^2
+\tfrac14 q^2 \widehat{f}\overline{\widehat{f}}
\\&=
\tfrac18 q(q-2) \widehat{f}\cdot\overline{\widehat{\tilde f}}
+ \tfrac18 q(q-2) \widehat{\tilde f}\cdot \overline{\widehat{f}}
+\tfrac14 q^2 \widehat{f}\cdot\overline{\widehat{f}}.
\end{align*}

Using this identity together with Plancherel's Theorem gives
\begin{multline*}
\tfrac12 q(q-1) \int (\Re\widehat{f})^2 |\widehat{\one_\bb}|^{q-2}
+ \tfrac12 q \int (\Im\widehat{f})^2 |\widehat{\one_\bb}|^{q-2}
=
\tfrac14 q^2 \langle f*L_q,f\rangle
+\tfrac14 q(q-2) \langle f*L_q,\tilde f\rangle
\end{multline*}
since $L_q$ is real-valued and even.
Likewise,
\[\Re\big(\int \widehat{f}\,\,\widehat{\one_{\bb}}\,|\widehat{\one_\bb}|^{q-2} \big)
= \Re\big( \int \widehat{f}\,\,\widehat{K_{q}}\big)
= \Re\langle f,K_{q}\rangle = \langle f,\,K_{q}\rangle \]
since $f,K_{q}$ are real-valued.  
\end{proof}

A variant of Lemma~\ref{lemma:expansion} holds for $q=3$,
and follows from the same proof.
\begin{lemma} \label{lemma:expansionq=3}
Let $q=3$.
For each $d\ge 1$ there exists $C<\infty$ with the following property.
Let $E\subset\reals^d$.  If $|E\symdif\bb|$ is sufficiently small then
\begin{equation} 
\norm{\widehat{\one_E}}_3^3
= \norm{\widehat{\one_\bb}}_3^3
+3\langle K_{3},f\rangle
+\tfrac94  \langle f*L_{3},f\rangle
+\tfrac34   \langle f*L_{3},\tilde f\rangle
+O(|E\symdif\bb|^{2})
\end{equation}
where $f = \one_E-\one_\bb$.
\end{lemma}


The distinction between the two lemmas is that the
remainder term becomes merely $O(\defb^2)$ for $q=3$, rather than $O(\defb^{2+\varrho})$.
This lemma will be exploited in \cite{mcjw}. 

For $q\in(2,3)$ the situation changes. For simplicity we consider only the case $d=1$.
\begin{lemma} \label{lemma:firstorderexpansion}
Let $d=1$ and $q\in(2,3)$.
For any set $E\subset\reals^1$ satisfying $|E|=|\bb|$ and $\defb\ll 1$, 
\begin{equation} \norm{\widehat{\one_E}}_q^q = \norm{\widehat{\one_\bb}}_q^q
+q\langle K_{q},f\rangle + O(\defb^{q-1}).\end{equation}
\end{lemma}

What is different in this regime is that
the leading term $q\langle K_q,f\rangle$ can have the same order of magnitude 
as occurs when $E$ is a translate of $\bb$. 

\begin{proof}
Consider any dimension $d$.  As in the proof of Lemma~\ref{lemma:expansion},
\begin{align*}
\norm{\widehat{\one_E}}_q^q
&\le
\norm{\widehat{\one_\bb}}_q^q
+ q\langle K_{q},f\rangle
+ O\big(\int |\widehat{\one_\bb}(\xi)|^{q-2}|\widehat{f}(\xi)|^2\,d\xi \big)
+ O\big(\norm{\widehat{f}}_q^q)
\\&
= \norm{\widehat{\one_\bb}}_q^q
+ q\langle K_{q},f\rangle
+ O\big(\norm{\widehat{\one_\bb}}_q^{q-2} \norm{\widehat{f}}_q^2)
+ O\big(\norm{\widehat{f}}_q^q)
\\&
= \norm{\widehat{\one_\bb}}_q^q
+ q\langle K_{q},f\rangle
+ O\big(\defb^{2(q-1)/q})
+ O\big(\defb^{q-1}).
\\&
= \norm{\widehat{\one_\bb}}_q^q
+ q\langle K_{q},f\rangle
+ O\big(\defb^{1+\varrho})
\end{align*}
where $1+\varrho=2(q-1)q^{-1}\in(1,q-1)$ since $q>2$.

To do better for $d=1$, 
use the majorization $\widehat{\one_\bb}(\xi) = O(|\xi|^{-(d+1)/2})$ and specialize
to $d=1$ to obtain
\begin{align*}
\int |\widehat{\one_\bb}(\xi)|^{q-2}|\widehat{f}(\xi)|^2\,d\xi 
&\le C
\int |\xi|^{-(q-2)}|\widehat{f}(\xi)|^2\,d\xi 
\\&
=C\iint f(x)f(y)|x-y|^{q-3}\,dx\,dy
\\&
\le C\iint \one_{E\symdif\bb}(x)\,\one_{E\symdif\bb}(y)\,|x-y|^{q-3}\,dx\,dy
\\&
\le  C|E\symdif\bb|^{q-1}.
\end{align*}
\end{proof}

\subsection{A necessary condition}

Here we prove Proposition~\ref{prop:firstvariation}.
Recall that $K_{q}$ is real-valued and even, and is continuous 
for all $q > 3-\frac2{d+1}$. In particular, for $d=1$ it is continuous for all $q>2$.
$L_{q}$ is continuous for all $q > 4-\frac2{d+1} = q_d$.

Because $\widehat{K_{q}}(\xi)=O(\abr{\xi}^{-(q-1)(d+1)/2})$, 
$\widehat{K_{q}}\in L^s$ for all $s > 2d(d+1)^{-1}(q-1)^{-1}$.
If $2d(d+1)^{-1}(q-1)^{-1}<1$ it follows that $\widehat{K_{q}}\in L^1$
and hence $K_{q}$ is a continuous function. 
For any $q>2$, $2d(d+1)^{-1}(q-1)^{-1} \le 2d(d+1)^{-1}<2$
and therefore $\widehat{K_{q}}\in L^s$ for some $s<2$.
By the Hausdorff-Young inequality, $K_{q}\in L^{s'}$.
In particular, $K_{q}$ is well-defined as a locally integrable function.

\begin{proof}[Proof of Proposition~\ref{prop:firstvariation}]
Let $x',x''\in\reals^d$ be any Lebesgue points of $K_{q}$ that satisfy $|x'|>1$ and $|x''|<1$.
Let $B',B''$ be the balls of measures $\delta$ with centers $x',x''$ respectively. 
Consider any $\delta>0$ that is sufficiently small to ensure that these balls 
are contained in $\reals^d\setminus\bb$ and $\bb$, respectively. 
Let $E = (\bb\setminus B'')\cup B'$.
Then $|E|=|\bb|$  and $f = \one_{B'} - \one_{B''}$.

Apply Lemma~\ref{lemma:firstorderexpansion}.
The leading term is
\[ \langle K_{q},f\rangle = \big(K_{q}(x')-K_{q}(x'')\big)\delta + o_{x',x''}(\delta).\]
Therefore
\[ \norm{\widehat{\one_E}}_q^q = \norm{\widehat{\one_\bb}}_q^q
+q \big(K_{q}(x')-K_{q}(x'')\big)\delta + o_{x',x''}(\delta).\]
By letting $\delta\to 0$ while $x',x''$ remain fixed we conclude that if $K_{q}(x')-K_{q}(x'')> 0$
then $\bb$ is not a local maximizer.
\end{proof}


\subsection{The kernels $K_{q}$, $L_q$} \label{subsection:thekernels}

$\widehat{\one_\bb}$ is a radially symmetric real-valued real analytic function which satisfies
\begin{equation} \label{eq:besseldecay}
|\widehat{\one_\bb}(\xi)| 
+ |\nabla \widehat{\one_\bb}(\xi)| 
\le C_d (1+|\xi|)^{-(d+1)/2}.
\end{equation}
The following lemmas are direct consequences of these properties.
Recall that $q_d = 4-2(d+1)^{-1}<4$.

\begin{lemma} \label{lemma:thekernels1}
Let $d\ge 1$ and $q\in(q_d,\infty)$.
The functions $K_{q}$, $L_q$ are 
real-valued, radially symmetric,
bounded, and H\"older continuous of some positive order.
Moreover, $K_{q}(x)\to 0$ as $|x|\to\infty$ and likewise for $L_q(x)$.
The function $K_{q}$ is continuously differentiable,  
and $x\cdot\nabla K_{q}$ is likewise real-valued, radially symmetric, and H\"older continuous
of some positive order.
These conclusions hold uniformly for $q$ in any compact subset of $(q_d,\infty)$.
\end{lemma}

The H\"older continuity  of $L_q$ is a direct consequence of the fact that $(1+|\xi|)^\rho |\widehat{L_q}|\in L^1$
for some $\rho>0$, which holds by virtue of \eqref{eq:besseldecay} since $(q_d-2)(d+1)/2 =d$ and $q>q_d$.


\begin{lemma} \label{lemma:thekernels2}
For each $d\ge 1$, 
$K_{q}$, $L_{q}$, and $x\cdot\nabla_x K_{q}$ depend continuously on $q\in(q_d,\infty)$.
This holds in the sense that for each compact subset $\Lambda\subset(q_d,\infty)$,
the mappings $q\mapsto K_{q}$  and $q \mapsto L_q$ are continuous from $\Lambda$ to the space
of continuous functions on $\reals^d$ that tend to zero at infinity.
Moreover,
there exists $\rho>0$ such that this mapping
from $\Lambda$ to the space of bounded H\"older continuous functions
of order $\rho$ on any bounded subset of $\reals^d$ is continuous.
The two conclusions also hold for $q\mapsto x\cdot\nabla K_{q}$. 
\end{lemma}


\begin{lemma}
Let $d\ge 1$ and $q\in(q_d,\infty)$.
Let $E\subset\reals^d$ be a Lebesgue measurable set with finite measure,
and let $f=\one_E-\one_\bb$.
Then
\begin{equation} \langle f*L_{q},f\rangle = O(|E\symdif\bb|^2)
\text{ and } \langle f*L_{q},\tilde f\rangle = O(|E\symdif\bb|^2). \end{equation}
\end{lemma}

This is an immediate consequence of the boundedness of $L_{q}$
together with the relation $|f|= \one_{E\symdif\bb}$. \qed

\begin{lemma} \label{lemma:strongnecessary}
For each $d\ge 1$ and each even integer $m\ge 4$ there exists $\eta=\eta(d,m)>0$ such that
whenever $|q-m|<\eta$,
there exists $c>0$ such that whenever $|y|\le 1\le |x|\le 2$,
\begin{equation} \label{eq:strongnecessary}
K_q(y)\ge K_q(x)+c\big|\,|x|-|y|\,\big|.
\end{equation}
\end{lemma}

\begin{proof}
Recall that $4>q_d$, so exponents satisfying the hypothesis satisfy $q>q_d$
provided that $\eta$ is chosen to be sufficiently small.

The conclusion holds when $q$ is an even integer $m\ge 4$. 
Indeed, $K_{m}$ is the convolution product
of $m-1$ factors of $\one_\bb$. This is a nonnegative radially symmetric function,
which is supported in the closed ball of radius $m-1$ centered at $0$,
and is strictly positive in the corresponding open ball.
The convolution product of two factors of $\one_\bb$ is a radial function $f_2$
whose support equals the ball of radius $2$ centered at $0$. For $|y|=1$, the function
$[0,2]\owns t\mapsto f_2(ty)$ is nonincreasing, with strictly negative derivative for all $t\in(0,2)$.
By induction on the number of factors in such a convolution product 
we conclude that $t\mapsto K_{m}(ty)$ has strictly negative derivative for all $t\in(0,m-1)$,
a stronger result than \eqref{eq:strongnecessary}.

For $q$ close to some even integer $m$,
\eqref{eq:strongnecessary} for $q$ follows from \eqref{eq:strongnecessary} for $m$,
because $x\cdot\nabla_x K_{q}(x)$ depends continuously on $q$.
\end{proof}


\section{Localization of perturbation near the boundary} \label{section:localization}
The next lemma asserts that if $|E\symdif\bb|$ is small, 
then in the expansion \eqref{eq:expansion} of $\norm{\widehat{\one_E}}_q^q$ about 
$\norm{\widehat{\one_\bb}}_q^q$, the first-order perturbation term $q \langle K_{q},
\one_E-\one_\bb \rangle$ is rather negative unless most of the symmetric
difference $E\symdif\bb$ lies close to the boundary of $\bb$.
To facilitate the discussion, for $E\subset\reals^d$ and $\eta\in(0,1]$ define 
\begin{equation} \label{eq:nearnotation}
E_\eta =\{x\in E\symdif\bb: \big|\,|x|-1\,\big|\ge\eta\}.\end{equation}

This notation is employed at several points of the analysis below.
The reader should beware that although $E_\eta$ is always constructed from $E$ as in \eqref{eq:nearnotation},
$A_\eta,B_\eta$ are not constructed from $A,B$ respectively in exactly this same way
in the proof of Lemma~\ref{lemma:nearboundary}.

\begin{lemma} \label{lemma:nearboundary}
Let $d\ge 1$. Let $q\in(q_d,\infty)$.
Suppose that $q>q_d$ and that
$K_q$ satisfies both \eqref{eq:Knecessarystronger} and \eqref{eq:strongnecessary}. 
Then there exist $c,C,\varrho\in\reals^+$ with the following property.
Let $\eta\in(0,1]$ and let $E\subset\reals^d$ be a Lebesgue measurable set satisfying $|E|=|\bb|$.
Then 
\begin{equation}
\norm{\widehat{\one_E}}_q^q
\le \norm{\widehat{\one_{E\setminus E_\eta}}}_q^q
-c\eta |E_\eta| 
+C \defb\cdot|E_\eta| + C \defb^{2+\varrho}.
\end{equation}
If $q$ is sufficiently large then this holds with $\varrho=1$.
\end{lemma}

The key point is that the negative term $-c\eta|E_\eta|$
is proportional to the first power of $|E_\eta|$, albeit with a factor
that tends to zero with $\eta$.

\begin{proof}
Continue to express $E = (\bb\cup A)\setminus B$
where $B\subset\bb$ and $A\subset\reals^d\setminus\bb$ satisfy $|A|=|B|=\tfrac12|E\symdif\bb|$.
Set 
$A_\eta = A\cap E_\eta$ and (although this is inconsistent with the general notation \eqref{eq:nearnotation})
$B_\eta = B\cap E_\eta$. 
Consider first the case in which $|B_\eta|\ge|A_\eta|$.
Choose a measurable subset $A'_\eta$ satisfying $A_\eta\subset A'_\eta\subset A$
that satisfies $|A'_\eta|=|B_\eta|$.  Such a set exists since $|A|=|B|\ge |B_\eta|$.

Set $f=\one_E-\one_\bb=\one_A-\one_B$, $f_\eta = \one_{A'_\eta}-\one_{B_\eta}$, and 
\begin{equation}\label{fdaggerdefn}
f^\dagger = f-f_\eta = \one_{A\setminus A'_\eta} - \one_{B\setminus B_\eta}. \end{equation}

Because $L_{q}$ is bounded and $|f_\eta|= \one_{A'_\eta\cup B_\eta}$
satisfies \[\norm{f_\eta}_{L^1}\le |A'_\eta|+|B_\eta| = 2|B_\eta|\le 2|E_\eta|,\]
one has
\begin{align*} \langle L_{q}*f,f\rangle =
\langle L_{q}*f^\dagger,f^\dagger \rangle +O(\defb \cdot |E_\eta|). 
\end{align*}
Likewise
$\langle L_{q}*f,\tilde f\rangle
=\langle L_{q}*f^\dagger,\tilde f^\dagger\rangle
+O(\defb \cdot |E_\eta|)$.

By the hypothesis \eqref{eq:strongnecessary},
\[ K_{q}(x) \ge K_{q}(y) + c\eta \ \text{ whenever $|x|\le 1-\eta$ and $|y|\le 1+\eta$.}\]
Consequently
\begin{align*} \langle K_{q},f_\eta\rangle \le -c\eta \norm{f_\eta}_1
= -c\eta |B_\eta|-c\eta |A'_\eta| \le  -c\eta|E_\eta| \end{align*}
and in total, 
\begin{equation} \langle K_{q},f\rangle \le -2c\eta|E_\eta| + \langle K_q,f^\dagger\rangle.\end{equation}

Therefore by Lemma~\ref{lemma:expansion},
\begin{align*}
\norm{\widehat{\one_E}}_q^q
&= \norm{\widehat{\one_\bb}}_q^q +q\langle K_{q},f\rangle 
+\tfrac14 q^2\langle L_{q}*f,f\rangle 
+\tfrac14 q(q-2)\langle L_{q}*f,\tilde f\rangle 
\\ &\qquad
+ O(|E\symdif\bb|^{2+\varrho})
\\ & \le \norm{\widehat{\one_\bb}}_q^q -c\eta|E_\eta| + q\langle K_q,f^\dagger\rangle
+\tfrac14 q^2\langle L_{q}*f^\dagger,f^\dagger\rangle 
+\tfrac14 q(q-2)\langle L_{q}*f^\dagger,\tilde f^\dagger\rangle 
\\ &\qquad +O(\defb\cdot|E_\eta|) + O(|E\symdif\bb|^{2+\varrho})
\\ & = \norm{\widehat{\one_{E\setminus E_\eta}}}_q^q -c\eta|E_\eta|
+O(\defb\cdot|E_\eta|) + O(|E\symdif\bb|^{2+\varrho})
\end{align*}
since Lemma~\ref{lemma:expansion} also gives
\begin{align*} 
\norm{\widehat{\one_{E\setminus E_\eta}}}_q^q 
= 
q\langle K_q,f^\dagger\rangle
+\tfrac14 q^2\langle L_{q}*f^\dagger,f^\dagger\rangle 
+\tfrac14 q(q-2)\langle L_{q}*f^\dagger,\tilde f^\dagger\rangle 
+ O(|(E\setminus E_\eta)\symdif\bb|)^{2+\varrho}
\end{align*}
and $(E\setminus E_\eta)\symdif\bb\subset E\symdif\bb$.
This completes the analysis of the case in which $|B_\eta|\ge |A_\eta|$.

The same reasoning can be applied in the alternative situation when $|A_\eta|\ge |B_\eta|$,
by replacing $B_\eta$ by a superset $B'_\eta$ satisfying $B_\eta\subset B'_\eta\subset B$
and $|B'_\eta|=|A_\eta|$.
\end{proof}

A natural choice of the parameter $\eta$ in the preceding lemma leads to a bound 
which will be combined with a more detailed analysis of
$\norm{\widehat{\one_{E\setminus E_\eta}}}_q^q$ below.

\begin{corollary} \label{cor:corona}
Under the hypotheses of Lemma~\ref{lemma:nearboundary},
if $\lambda$ is a sufficiently large constant depending only on $d,q$ then for
any set $E\subset\reals^d$ that satisfies
$|E|=|\bb|$ and $\lambda |E\symdif\bb|\le 1$, if $\eta = \lambda|E\symdif\bb|$ then
\begin{equation}
\norm{\widehat{\one_E}}_q^q
\le \norm{\widehat{\one_{E\setminus E_\eta}}}_q^q
-c \lambda |E_\eta| \cdot \defb + C_\lambda \defb^{2+\varrho}
\end{equation}
where $c,\varrho\in\reals^+$ depend only on $d,q$
while $C_\lambda$ depends only on $d,q,\lambda$.
If $q$ is sufficiently large then the conclusion holds with $\varrho=1$.
\end{corollary}
These statements hold uniformly for all $q$ in any fixed compact subset of $(q_d,\infty)$.

\section{Connection with a spectral problem on $S^{d-1}$}

Let $d\ge 1$.
Let $E\subset\reals^d$ be a Lebesgue measurable set with $|E|=|\bb|$.
For any set whose symmetric difference $E\symdif\bb$ is contained in a small neighborhood
of the unit sphere, we next show how $\norm{\widehat{\one_E}}_q^q$
can be expressed in terms of associated functions on the unit sphere and rotation-invariant
quadratic forms involving these functions, modulo a small error.

Let $\sigma$ denote surface measure on the unit sphere $S^{d-1}\subset\reals^d$.
Continue to express
$E = (\bb\cup A)\setminus B$
where $B\subset\bb$ and $A\subset\reals^d\setminus\bb$.
Let $K_{q}$, $L_q$ be as defined above.

\begin{definition}
\begin{equation} \gamma_{q,d}= -x\cdot\nabla K_{q}(x)\big|_{|x|=1}.  \end{equation}
\end{definition}

\begin{definition} \label{Fdefn}
Let $E\subset\reals^d$ be a bounded Lebesgue measurable set.
The functions $a,b,F:S^{d-1}\to[0,\infty)$ are defined by
\begin{align}
a(\alpha) &= \int_0^\infty \one_A(r\alpha)r^{d-1}\,dr
\\b(\alpha) &= \int_0^\infty \one_B(r\alpha)r^{d-1}\,dr
\\ F(\alpha)&=b(\alpha)-a(\alpha)
\label{eq:Fdefn} \end{align}
for $\alpha\in S^{d-1}$. 
\end{definition}

These satisfy
\begin{equation*}
\int_{S^{d-1}}a\,d\sigma = |A|,
\qquad \int_{S^{d-1}}b\,d\sigma = |B|,
\qquad \int_{S^{d-1}}F\,d\sigma = 0.
\end{equation*}
Consequently
\begin{equation} \label{eq:CS}
\int_{S^{d-1}} (a^2+b^2)\,d\sigma \ge \sigma(S^{d-1})^{-1} \big( |A|^2+|B|^2)
=\frac{\defb^2}{2\sigma(S^{d-1})}
\end{equation}
by Cauchy-Schwarz since $|A|=|B|=\tfrac12 \defb$.
Moreover,
\begin{equation}
\int_{S^{d-1}} F^2\,d\sigma 
\le \int_{S^{d-1}} (a^2+b^2)\,d\sigma
\end{equation}
since $F^2 = (a^2+b^2-2ab)\le a^2+b^2$.

\begin{lemma} \label{lemma:nearness}
Let $d\ge 1$. For each compact subset $\Lambda\subset(q_d,\infty)$
there exists $\varrho>0$ such that the following holds for all $q\in\Lambda$.
Let $\lambda\in[1,\infty)$.
Suppose that $\lambda |E\symdif\bb|\le 1$ and that the set $E$ satisfies
\begin{equation}\label{eq:nearness} E\symdif\bb \subset 
\{x: \big|\,1-|x|\,\big| \le \lambda \defb\}.\end{equation}
Then
\begin{equation} \langle L_{q}*f,f\rangle
= \iint_{S^{d-1}\times S^{d-1}} F(\alpha)F(\beta) L_{q}(\alpha-\beta) \,d\sigma(\alpha)\,d\sigma(\beta) 
+O_\lambda(\defb^{2+\varrho}).  \end{equation}
Likewise
\begin{equation} \langle L_{q}*f,\tilde f\rangle
= \iint_{S^{d-1}\times S^{d-1}} F(\alpha)F(-\beta) L_{q}(\alpha-\beta) \,d\sigma(\alpha)\,d\sigma(\beta)
+O_\lambda(\defb^{2+\varrho}).  \end{equation}
Finally,
\begin{equation} \label{eq:K1ipidentity}
\langle K_{q},f\rangle
\le -\tfrac12 \gamma_{q,d} \int_{S^{d-1}} (a^2+b^2)\,d\sigma
+O_\lambda(\defb^{2+\varrho}).
\end{equation} 
The constants implicit in the expressions $O_\lambda(\defb^{2+\varrho})$ 
depend only on  $\lambda,\Lambda,d$.
\end{lemma}

The conclusion \eqref{eq:K1ipidentity} is an inequality, rather than an identity,
even up to the remainder $O(\defb^{2+\varrho})$.
It is important to note the form of \eqref{eq:K1ipidentity}, in which the right-hand side
is quadratic in $a,b$, even though the left-hand side is formally linear in $f$. 
Indeed, the left-hand side is roughly of second order in $|E\symdif\bb|$
when the set $E \symdif\bb$ satisfies \eqref{eq:nearness},
but is roughly of first order when $E\symdif\bb$ lies a fixed distance from this boundary.
Corollary~\ref{cor:corona} addresses the contribution of the latter situation,
while Lemma~\ref{lemma:nearness} deals with the former.

\begin{proof}
\begin{align*}
\langle L_{q}*f,\,f\rangle
&= 
\iint_{S^{d-1}\times S^{d-1}}
\int_0^\infty \int_0^\infty
f(r\alpha) f(\rho\beta) L_q(r\alpha-\rho\beta) 
r^{d-1}\,dr\,\rho^{d-1}\,d\rho
\,d\sigma(\alpha)\,d\sigma(\beta)
\\&= 
\iint_{S^{d-1}\times S^{d-1}}
\int_0^\infty \int_0^\infty
f(r\alpha) f(\rho\beta) L_q(\alpha-\beta) \,d\sigma(\alpha)\,d\sigma(\beta)
r^{d-1}\,dr\,\rho^{d-1}\,d\rho 
\\ & \qquad\qquad + O_\lambda(\defb^{2+\varrho})
\end{align*}
since $L_{q}$ is H\"older continuous 
and 
the hypothesis  \eqref{eq:nearness} ensures that
$|r-1|+|\rho-1|=O_\lambda(\defb)$ whenever $f(r\alpha)f(\rho\beta)\ne 0$.
Each factor of $f$ accounts for one factor $O_\lambda(|E\symdif\bb|)$, and the H\"older continuity of $L_{q}$ 
together with the support restriction result in another factor $O_\lambda(|E\symdif\bb|^\varrho)$.

By performing the integrals with respect to $r,\rho$ and invoking the definition
of $F$, one obtains
\begin{align*}
\langle L_{q}*f,\,f\rangle =
\iint_{S^{d-1}\times S^{d-1}} F(\alpha) F(\beta) L_{q}(\alpha-\beta) \,d\sigma(\alpha)\,d\sigma(\beta)
+O_\lambda (\defb^{2+\varrho}).
\end{align*}
The same reasoning applies to $\langle L_{q}*f,\tilde f\rangle$.

To prove \eqref{eq:K1ipidentity}, 
define $K_1^*(s)=K_{q}(x)$ where $|x|=s$. 
Since $\nabla K_{q}$ is H\"older continuous,
\[ K_1^*(r) = K_1^*(1)-\gamma_{q,d}(r-1) + O_\lambda (|r-1|^{1+\varrho})\]
where $\varrho>0$. Consequently
\begin{align*}
\langle K_{q},f\rangle
& = \int_{S^{d-1}} \int_{\reals^+} K_1^*(r)f(r\alpha)\,r^{d-1}\,dr\,d\sigma(\alpha)
\\& = K_1^*(1) \int_{S^{d-1}} F(\alpha)\,\,d\sigma(\alpha)
- \gamma_{q,d} \int_{S^{d-1}} \int_{\reals^+} f(r\alpha)\,(r-1)r^{d-1}\,dr\,d\sigma(\alpha)
\\&\qquad +O_\lambda (\defb^{1+\varrho}\norm{f}_1)
\end{align*}
since $E\symdif\bb\subset\{x: \big|\, |x|-1\,\big|\le\lambda|E\symdif\bb|\}$.
Since $\int_{S^{d-1}}F(\alpha)\,d\sigma(\alpha) = \int_{\reals^d} f(x)\,dx=0$
and $\norm{f}_1=|E\symdif\bb|$,
\begin{align*}
\langle K_{q},f\rangle
& = -\gamma_{q,d} \int_{S^{d-1}} \int_{\reals^+} f(r\alpha)\,r^{d-1}(r-1)\,dr\,d\sigma(\alpha) +O_\lambda(\defb^{2+\varrho})
\\&=
-\gamma_{q,d} \int_{S^{d-1}} \int_{\reals^+} (\one_A-\one_B)(r\alpha)\,r^{d-1}(r-1)\,dr\,d\sigma(\alpha)
+O_\lambda (\defb^{2+\varrho}).
\end{align*}

Consider the contribution of $\one_A$ to the last integral. 
The set $A\subset\reals^d$ is by hypothesis contained in $\{x: 1\le |x|\le \lambda \defb \}$,
so the factor $r-1$ is nonnegative whenever
$a(\alpha)=\int_{\reals^+} \one_A(r\alpha) r^{d-1}\,dr$ is nonzero.
Define $\tilde a(\alpha)$ by the relation
\[\int_1^{1+\tilde a(\alpha)} r^{d-1}\,dr = a(\alpha).\] 
Then \[\tilde a(\alpha) = a(\alpha)+O(a(\alpha)^2) = a(\alpha)+O_\lambda(\defb^2)\]
since the hypotheses $A\subset\{x: |x|\le 1+\lambda\defb\}$
and $\lambda\defb\le 1$ imply
that \[a(\alpha) \le  \int_1^{1+\lambda\defb} r^{d-1}\,dr = O_\lambda(\defb).\]

Fix $\alpha$ momentarily.
Among all sets $A\subset\reals^d\setminus\bb$ 
that satisfy $\int_1^\infty \one_A(r\alpha)r^{d-1}\,dr \equiv a(\alpha)$, plainly the integral
$\int_1^\infty \one_A(r\alpha) r^{d-1}(r-1)\,dr$ is minimized when 
$\{r: r\alpha\in A\}$ is equal to the interval $[1,1+\tilde a(\alpha)]$.
Therefore
\begin{align*}
\int_{S^{d-1}} \int_{\reals^+} \one_A(r\alpha)\,r^{d-1}(r-1)\,dr
\,d\sigma(\alpha)
&\ge
\int_{S^{d-1}} \int_{1}^{1+\tilde a(\alpha)}  \,r^{d-1}(r-1)\,dr
\,d\sigma(\alpha)
\\  & =
\int_{S^{d-1}} \big((d+1)^{-1}r^{d+1}-d^{-1}r^d\big)\big|_1^{1+\tilde a(\alpha)} 
\,d\sigma(\alpha)
\\  & = \int_{S^{d-1}} 
\Big(\tfrac12 \tilde a(\alpha)^2 + O_\lambda (\tilde a(\alpha))^3\Big) \,d\sigma(\alpha)
\\  & = \int_{S^{d-1}} \tfrac12 a(\alpha)^2 \,d\sigma(\alpha) + O_\lambda(\defb^3).
\end{align*}
The same analysis applies to
$\int_{S^{d-1}} \int_{\reals^+} \one_B(r\alpha)\,r^{d-1}(1-r)\,dr$ 
with appropriate reversals of signs and inequalities, establishing \eqref{eq:K1ipidentity}.
\end{proof}

\begin{definition}
The quadratic form $\scriptq_{q,d}$ acts on a pair of real-valued functions $\varphi,\psi\in L^2(S^{d-1})$ by
\begin{align} \label{Qdefn} \scriptq_{q,d}(\varphi,\psi)
= \iint_{S^{d-1}\times S^{d-1}} \varphi(\alpha)\psi(\beta) L_{q}(\alpha-\beta)\,d\sigma(\alpha)\,d\sigma(\beta).
\end{align} \end{definition}
This form is invariant with respect to the diagonal action of the rotation group $O(d)$
on $(\varphi,\psi)$, hence is potentially amenable to analysis in terms of spherical harmonic decomposition.

We have shown
\begin{lemma} \label{lemma:reducedtoF}
Let $d\ge 1$, and let $\Lambda\subset(q_d,\infty)$ be a compact set.
There exists $\varrho>0$ such that for any sufficiently large $\lambda\in\reals^+$,
any $q\in\Lambda$,
and any measurable set $E\subset\reals^d$ satisfying $|E|=|\bb|$ and $\lambda \defb\le 1$, 
if $E\symdif\bb \subset\{x: \big|\,|x|-1\,\big| \le \lambda \defb\}$ then 
$f=\one_E-\one_\bb$ satisfies
\begin{equation} \begin{aligned}
q\langle K_1,f\rangle &+ \tfrac14 q^2 \langle K_2*f,f\rangle + \tfrac14 q(q-2) \langle K_2*f,\tilde f\rangle
\\ &\le -\tfrac{1}2 q \gamma_{q,d} \int_{S^{d-1}} (a^2+b^2)\,d\sigma
+ \tfrac14 q^2 \scriptq_{q,d}(F,F) + \tfrac14 q(q-2) \scriptq_{q,d}(F,\tilde F)
\\ &\qquad\qquad +O_\lambda(\defb^{2+\varrho})
\end{aligned}\end{equation}
where $F$ is as defined in \eqref{eq:Fdefn}.
\end{lemma}
The notation $O_\lambda$ indicates an implicit constant that depends only on $d,\lambda,\Lambda$.

The same analysis establishes a converse inequality, under an extra hypothesis,
which will be exploited below.
We say that two measurable functions that are finite almost everywhere have disjoint supports
if their product vanishes almost everywhere.
Let $F,a,b$ be associated to $E$ as in Definition~\ref{Fdefn}.

\begin{lemma} \label{lemma:conversereduction}
Let $d\ge 1$ and $q\in (q_d,\infty)$.
There exists $\varrho>0$ such that the following holds 
for any sufficiently large $\lambda\in\reals^+$.
Let $E\subset\reals^d$ be Lebesgue measurable and satisfy $|E|=|\bb|$, $\lambda \defb\le 1$, 
and $E\symdif\bb \subset\{x: \big|\,|x|-1\,\big| \le \lambda \defb\}$.
Suppose that the functions $a,b\in\lt(S^{d-1})$ associated to $E$ have disjoint supports. 
Then
\begin{multline} 
\norm{\widehat{\one_E}}_q^q
\ge 
\norm{\widehat{\one_\bb}}_q^q
-\tfrac{1}2 q \gamma_{q,d} \int_{S^{d-1}} F^2 \,d\sigma
\\
+ \tfrac14 q^2 \scriptq_{q,d}(F,F) + \tfrac14 q(q-2) \scriptq_{q,d}(F,\tilde F)
- O_\lambda(\defb^{2+\varrho}).
\end{multline}
\end{lemma}

\section{Eigenvalue analysis for even integer exponents} \label{section:Eigenvalue}

Throughout \S\ref{section:Eigenvalue} it is assumed that $q\ge 4$ is an even integer,
and for $A\subset\reals^d$, $\widehat{A}$ is used as shorthand for $\widehat{\one_A}$.
The analysis here closely follows the corresponding step for the Riesz-Sobolev inequality 
in \cite{christRSult}.

For $n\in\{0,1,2,\dots\}$ let $\pi_n$ be the orthogonal projection from
$\lt(S^{d-1},\sigma)$ onto the subspace of spherical harmonics of degree $n$.
By construction, $\pi_0(F)=0$.

\begin{lemma} \label{lemma:quadformbound}
Let $d\ge 2$, and let $q\ge 4$ be an even integer.
There exists $c_{q,d}>0$ such that
\begin{equation} \label{eq:quadformbound}
\tfrac14 q^2 \scriptq_{q,d}(F,F) + \tfrac14 q(q-2) \scriptq_{q,d}(F,\tilde F)
\le \big(\tfrac12 q\gamma_{q,d}-c_{q,d}\big) \norm{F}_{\lt(S^{d-1})}^2
\end{equation}
for all $F\in \lt(S^{d-1})$
satisfying $\pi_n(F)=0$ for all $n\le 2$.
\end{lemma}

Since $L_q$ is a radial function, $\scriptq_{q,d}(G,G')=0$
for any two spherical harmonics $G,G'$ of different degrees.
If $G$ is a spherical harmonic of degree $n$, then so is its reflection 
$\tilde G(x)=G(-x)$.
Therefore if $\pi_n(G)=0$ for all $n\le 2$ then
the left-hand side of \eqref{eq:quadformbound} is equal to
\begin{equation} 
\sum_{n=3}^\infty
\Big(\tfrac14 q^2 \scriptq_{q,d}(\pi_n F,\pi_n F) 
+ \tfrac14 q(q-2) \scriptq_{q,d}(\pi_n F,\pi_n \tilde F)\Big).
\end{equation}
Since $\norm{F}_2^2 = \sum_n \norm{\pi_n F}_2^2$,
it suffices to prove Lemma~\ref{lemma:quadformbound} for spherical harmonics of
degree $n$, for arbitrary $n\ge 3$, with $0<c_{q,d}$ uniform in $n$.

Uniformity with respect to $n$ follows automatically
from any $n$--dependent bound.
$\scriptq_{q,d}(G,G')$ equals the inner product of $G'$ with the result
of applying a compact selfadjoint operator to $G$,
so for spherical harmonics $G$ of degree $n$,
\begin{equation} 
\tfrac14 q^2 \scriptq_{q,d}(G,G) + \tfrac14 q(q-2) \scriptq_{q,d}(G,\tilde G)
\le\eps_n \norm{G}_{\lt(S^{d-1})}^2
\end{equation}
where $\lim_{n\to\infty} \eps_n=0$ for each fixed $d,q$.
So it suffices to prove that for each $n\ge 3$,
\eqref{eq:quadformbound} holds with some positive constant.
Moreover, from the finite-dimensionality of the space of spherical
harmonics of degree $n$ it follows easily that \eqref{eq:quadformbound}
holds for degree $n$ if for each nonzero spherical harmonic $G$ of degree $n$
\begin{equation} \label{eq:neednomore}
\tfrac14 q^2 \scriptq_{q,d}(G,G) + \tfrac14 q(q-2) \scriptq_{q,d}(G,\tilde G)
- \tfrac12 q\gamma_{q,d} \norm{G}_{\lt(S^{d-1})}^2 \ <\ 0.
\end{equation}

Our proof strategy is not to calculate $\gamma_{q,d}$
and $\scriptq_{q,d}(G,G),\scriptq_{q,d}(G,\tilde G)$
explicitly as functions of $n,q,d$ for $G\in\pi_n(\lt(S^{d-1}))$.
Instead, following \cite{christRSult},
we will proceed indirectly, associating to an arbitrary $G$ a one-parameter family
of sets $E(s)\subset\reals^d$ satisfying $|E(s)|\equiv |\bb|$, 
in such a way that \eqref{eq:neednomore}
is equivalent to a certain asymptotic inequality for $\norm{\widehat{E(s)}}_q^q
- \norm{\widehat{\bb}}_q^q$ as $s\to 0$.
For $q$ an even integer, this difference will be interpreted
as a difference involving multiple convolutions.
We will establish the desired asymptotic
inequality directly, without any explicit eigenvalue calculations.

Let $G$ be a spherical harmonic of degree $n$.
Define $\varphi:S^{d-1}\times (-\tfrac12,\tfrac12) \to\reals^+$ as follows.
Let $\theta\in S^{d-1}$ be arbitrary.
If $sG(\theta)\ge 0$ then $\varphi(\theta,s)\ge 0$,
and $\int_{1}^{1+\varphi(\theta,s)} t^{d-1}\,dt = sG(\theta)$.
If $sG(\theta) \le 0$ then $\varphi(\theta,s)\le 0$,
and $\int_{1+\varphi(\theta,s)}^{1} t^{d-1}\,dt = -sG(\theta)$.
Equivalently, for either possible sign, 
$(1+\varphi(\theta,s))^d-1 = dsG(\theta)$.
Thus
\begin{equation}\label{eq:varphij} 
\varphi(\theta,s) = sG(\theta)+O(s^2),
\end{equation}
and $(\theta,s)\mapsto \varphi(\theta,s)$
is a $C^\infty$ function specified by \eqref{eq:varphij}.

For $s\in\reals$ with $|s|$ small define sets $E(s)\subset\reals^d$ by
\begin{equation} E(s) = \{t\theta: \theta\in S^{d-1} \text{ and }
0\le t \le 1 + \varphi(\theta,s)\}.\end{equation}
Since $\int_{S^{d-1}} G\,d\sigma=0$,
$|E(s)|=|\bb|$ for all $s$ in a neighborhood of $0$.
The function $F_{s}$ associated to $E(s)$ satisfies $F_{s}\equiv sG$.
Consequently
\begin{equation} \label{eq:E(s)expansion}
\norm{\widehat{{E(s)}}}_q^q
= \norm{\widehat{{\bb}}}_q^q
+ \Big(
-\tfrac12 q\gamma_{q,d}\norm{G}_{\lt}^2
+ \tfrac14 q^2 \scriptq_{q,d}(G,G) 
+ \tfrac14 q(q-2) \scriptq_{q,d}(G,\tilde G)
\Big)\,s^2
+ o(s^2)
\end{equation}
as $s\to 0$, where
the $o(s^2)$ remainder term depends also on $G,d,q$.

\begin{lemma}\label{lemma:quadgain} Let $d\ge 2$. 
Let $q\ge 4$ be an even integer.
For each nonzero spherical harmonic $G\in\lt(S^{d-1})$ of any degree $n\ge 3$
there exists $c>0$ such that
\begin{equation}\label{eq:quadgain}
\norm{\widehat{E(s)}}_q^q \le \norm{\widehat{{\bb}}}_q^q-cs^2 +o(s^2) \end{equation}
as $s\to 0$.
\end{lemma}

Combining \eqref{eq:quadgain} and \eqref{eq:E(s)expansion} with Lemma~\ref{lemma:conversereduction}
and the above reductions establishes Lemma~\ref{lemma:quadformbound}.

We next prove Lemma~\ref{lemma:quadgain}.
The Steiner symmetrization $S^\dagger$ of a Lebesgue measurable set $S\subset\reals^d$
is defined as follows.
Identify $\reals^d$ with $\reals^{d-1}\times\reals^1$ in the usual way.
For $x'\in\reals^{d-1}$ define $S_{x'}=\{t\in\reals: (x',t)\in S\}$.
Then $S^\dagger$ is defined to be
\[S^\dagger=\{(x',t)\in\reals^{d-1}\times\reals: |t| < \tfrac12 |S_{x'}|\}.\]

Denote by $E(s)^\dagger$ the Steiner symmetrization of $E(s)$.
For any rotation $\scripto\in O(d)$ set
\[ E^\dagger_\scripto = \scripto((\scripto^{-1}(E))^\dagger).\]
Since $|E(s)|=|E(s)^\dagger_\scripto|=|\bb|$,
we have
\begin{equation}
\norm{\widehat{{E(s)^\dagger_\scripto}}}_q^q\le
\norm{\widehat{\bb}}_q^q
\end{equation}
for every even integer $q\ge 2$,
by the Riesz-Sobolev inequality and the identity
\[\norm{\widehat{\one_A}}_q^q = \norm{\one_A*\cdots*\one_A}_2^2\]
where there are $\tfrac{q}2$ factors in the convolution product.
Therefore in order to prove Lemma~\ref{lemma:quadgain},
it suffices to prove that for any $n\ge 3$ and any nonzero spherical harmonic $G\in\lt(S^{d-1})$
of degree $n$, there exist $\scripto\in O(d)$ and $c>0$ such that
\begin{equation}
\norm{\widehat{{E(s)_\scripto}}}_q^q
\le \norm{\widehat{{E(s)^\dagger_\scripto}}}_q^q
-cs^2+o(s^2) \ \text{ as $s\to 0$.}
\end{equation}
Here $c$ is permitted to depend on $G,\scripto,q,d$.

Regard $\reals^d$ as $\reals^{d-1}\times\reals^1$.
For $x'\in\reals^{d-1}$ define
\[I(x',s) = \{t\in\reals: (x',t)\in E(s)\}.\]
There exist a neighborhood $V$ of the origin in $\reals^{d-1}$
and $\eta>0$
such that for each $x'\in V$ and each $s\in[-\eta,\eta]$,
$I(x',s)\subset\reals^1$ is an interval. 
Let $c(x',s)$ be the center of the interval $I(x',s)$.

Decompose $G$ as $G = \Gje+\Gjo$ 
by expanding $G(x',x_d)$ as a linear combination of monomials in $x=(x',x_d)$ and defining
$\Gje(x',x_d)$ to be the contribution of all monomials having even degrees with respect to $x_d$,
and $\Gjo(x',x_d)$ to be the contribution of all monomials having odd degrees.
Define $P:\reals^{d-1}\to\reals$ to be the polynomial of degree $\le n-1$ defined by
\begin{equation}\label{eq:associatedP}
P(x') =  x_d^{-1}\Gjo(x',x_d)\ \text{ with } x_d = (1-|x'|^2)^{1/2}.
\end{equation}
It is shown in \cite{christRSult} that
\begin{equation} c(x',s) = sP(x')+O(s^2), \end{equation}
for $x'\in\reals^{d-1}$ in a small neighborhood of $0$ and for $|s|$ small.

Define
$\Sigma$ to be the set of all $\bx'\in(\reals^{d-1})^q$
satisfying $x'_1+\cdots+x'_m = x'_{m+1}+\cdots + x'_q$,
where $m=q/2$. Define
$\mu$ to be the natural $(q-1)\cdot(d-1)$--dimensional Lebesgue measure on $\Sigma$.
Define $\Psharp(G):(\reals^{d-1})^q\to\reals$ by 
\begin{equation}\label{Psharpdefn}
\Psharp(G)(\bx')
 = \sum_{i=1}^m P(x'_i) - \sum_{j=m+1}^q P(x'_j).
\end{equation}
Choose any norm $\norm{\cdot}$ on the vector space of all 
polynomials $P:\Sigma \to\reals$ of degrees $\le n-1$.

\begin{lemma} \label{lemma:forSteiner}
For any integer $n\ge 1$ and any real-valued spherical harmonic
$G\in\lt(S^{d-1})$ of degree $n$ there exists $c>0$ such that
\begin{equation} \norm{\widehat{E(s)}}_q^q \le \norm{\widehat{E(s)^\dagger}}_q^q
- cs^2 \norm{\Psharp(G)}^2 + o(s^2) \text{ as $s\to 0$.}  \end{equation}
\end{lemma}

\begin{proof}
Write
\begin{equation}
\norm{\widehat{E(s)}}_q^q
= \langle 
\one_{E(s)}
*\one_{E(s)}
*\cdots*
\one_{E(s)},\,
\one_{E(s)}
*\one_{E(s)}
*\cdots*
\one_{E(s)}\rangle
\end{equation}
with $m$ factors in each convolution product.
Thus
\begin{equation}
\norm{\widehat{E(s)}}_q^q
= \int_\Sigma \langle 
\one_{I(x'_1,s)}*\cdots*\one_{I(x'_m,s)},\,
\one_{I(x'_{m+1},s)}*\cdots*\one_{I(x'_q,s)}\rangle\,d\mu(\bx')
\end{equation}
where 
these convolutions are of functions defined on $\reals^1$,
each convolution product has $m$ factors,
and the inner product in the integrand is that of $\lt(\reals^1)$.

There is a corresponding expression for $\norm{\widehat{E(s)^\dagger}}_q^q$,
in which each $I(x'_j,s)$ is replaced by the interval $I^\star(x'_j,s)$ of the same length
centered at $0\in\reals^1$.
Moreover
\begin{multline}
\langle \one_{I(x'_1,s)}*\cdots*\one_{I(x'_m,s)},\,
\one_{I(x'_{m+1},s)}*\cdots*\one_{I(x'_q,s)}\rangle
\\ \le
\langle \one_{I^\star(x'_1,s)}*\cdots*\one_{I^\star(x'_m,s)},\,
\one_{I^\star(x'_{m+1},s)}*\cdots*\one_{I^\star(x'_q,s)}\rangle
\end{multline}
for every $\bx'\in\Sigma$, by the one-dimensional Riesz-Sobolev inequality. 

If $V$ is a sufficiently small neighborhood of $0\in\reals^{d-1}$
then for all $\bx'\in \Sigma\cap V^q$,
\begin{multline}
\langle \one_{I(x'_1,s)}*\cdots*\one_{I(x'_m,s)},\,
\one_{I(x'_{m+1},s)}*\cdots*\one_{I(x'_q,s)}\rangle
\\ \le
\langle \one_{I^\star(x'_1,s)}*\cdots*\one_{I^\star(x'_m,s)},\,
\one_{I^\star(x'_{m+1},s)}*\cdots*\one_{I^\star(x'_q,s)}\rangle
\\ -c \Big|\sum_{i=1}^m c(x'_i,s) - \sum_{j=m+1}^q c(x'_j,s)\Big|^2.
\end{multline}
Here $c$ is a universal positive constant.
The final term can be expanded as
\begin{multline}
\Big|\sum_{i=1}^m c(x'_i,s) - \sum_{j=m+1}^q c(x'_j,s)\Big|^2
 = s^2 \Big|\sum_{i=1}^m P(x'_i) - \sum_{j=m+1}^q P(x'_j)\Big|^2
+ O(|s|^3)
\\ = s^2 \Psharp(G)(\bx')^2 + O(|s|^3),
\end{multline}
completing the proof of Lemma~\ref{lemma:forSteiner}.
\end{proof}

If $\Psharp(G)$ does not vanish identically then we conclude that 
there exists $c_G>0$ such that
$\norm{\widehat{E(s)}}_q^q
\le \norm{\widehat{E(s)^\dagger}}_q^q - c_Gs^2$
whenever $|s|$ is sufficiently small, as desired. 
This reasoning can be applied to any rotation of $E$. Thus 
in order to prove Lemma~\ref{lemma:quadgain}, it suffices to establish
the following algebraic fact.

\begin{lemma} \label{lemma:algebra:atlast}
Let $d\ge 2$, and let $q\in\{4,6,8,\dots\}$.  
For any $n\ge 3$ and any spherical harmonic $G\in\lt(S^{d-1})$ of degree $n$ that does
not vanish identically, there exists $\scripto\in O(d)$ such that $\Psharp(\scripto(G))$
does not vanish identically.
\end{lemma}

It follows immediately that for each $n,d$, $\inf_{G} \max_{\scripto\in O(d)} \norm{\Psharp(\scripto(G))}$
is strictly positive, where $G$ ranges over all spherical harmonics of
degree $n$ satisfying $\norm{G}_{\lt(S^{d-1})}=1$.

\begin{proof}[Proof of Lemma~\ref{lemma:algebra:atlast}]
It is elementary that for any even integer $q\ge 4$, $\Psharp$ can vanish
identically in a nonempty open subset of $\Sigma$ 
only if $P$ is an affine function of $x'\in\reals^{d-1}$. 
Therefore it suffices to show that for any  spherical harmonic $G$ of degree
strictly greater than $2$ which does not vanish identically, 
there exists $\scripto\in O(d)$ such that the polynomial
$P_\scripto$ associated to $G\circ\scripto$ via \eqref{eq:associatedP} fails to be affine.
This is proved in \cite{christRSult}.
\end{proof}

\section{Balancing} 

The following discussion can also be found in \cite{christRSult}, in a slightly more
complicated context.
Recall that $\aff(d)$ denotes the group of all affine automorphisms of $\reals^d$.
For any $\phi\in\aff(d)$ and measurable set $E\subset\reals^d$,
$\Phi_q(\phi(E))=\Phi_q(E)$
for all $q\in[2,\infty)$. 
Therefore in analyzing $\norm{\widehat{\one_E}}_q$ for measurable sets $E$ 
that have small symmetric difference with $\bb$, 
rather than writing $\one_E=\one_\bb + \one_{E\symdif\bb}$ and expanding 
$\norm{\widehat{\one_E}}_q$ about $\one_\bb$,
we wish to exploit an expansion based on a representation
$\one_E = \one_\scripte+\one_{E\symdif\scripte}$
for an optimally chosen element $\scripte$ of the orbit of $\bb$ under $\aff(d)$.
Equivalently, we seek to replace $E$ by $\phi(E)$
for some $\phi\in\aff(d)$, chosen so that $\phi(E)$
is best approximated by $\bb$.
In this section we specify which approximation is to be considered to be best
--- see Lemma~\ref{lemma:balanced} --- and prove that $\phi$ exists.

\begin{definition}
A bounded Lebesgue measurable set $E\subset\reals^d$
satisfying $|E|=|\bb|$ is balanced if
the function $F\in\lt(S^{d-1})$ associated to $E$ via Definition~\ref{Fdefn} satisfies
\begin{equation} 
\pi_n(F)=0
\ \text{ for every $n\in\{0,1,2\}$}.
\end{equation}
\end{definition}
Here, as above, $\pi_n$ denotes the orthogonal projection of $\lt(S^{d-1})$
onto its subspace of spherical harmonics of degree $n$. 
For $n=0$, this equation is the redundant 
assertion that $\int (\one_E-\one_\bb)=0$, a restatement of the hypothesis $|E|=|\bb|$.
An equivalent formulation is that
$\int_{S^{d-1}}F(y)P(y)\,d\sigma(y)=0$ 
for every polynomial $P:\reals^d\to\reals$ of degree less than or equal to $2$.

Denote by $\scriptm_d$ the vector space of all $d\times d$ square matrices with real entries, 
and by $\scriptm_d\oplus\reals^d$ the set of all ordered pairs $(S,v)$ where
$S\in \scriptm_d$ and $v\in\reals^d$, with the natural vector space structure.
Identify elements of $\scriptm_d$ with linear endomorphisms of $\reals^d$ in the usual way.
Fix any norm $\norm{\cdot}_{\scriptm_d}$ on $\scriptm_d$.

Elements $\phi\in\aff(d)$ take the form $\phi(x) = T(x)+v$ where 
$(T,v)\in\scriptm_d\oplus\reals^d$ is uniquely determined by $\phi$,
and $T:\reals^d\to\reals^d$ is an invertible linear transformation.
Define $\norm{\phi}_{\aff(d)}= \norm{T}_{\scriptm_d} + \norm{v}_{\reals^d}$.
We abuse notation by writing $\det(\phi)$ for the determinant of the unique $T\in\scriptm_d$
thus associated to $\phi$, and likewise $\trace(\phi) = \trace(T)$.

In this section we prove:
\begin{lemma}\label{lemma:balanced}
Let $d\ge 1$. There exists $c>0$ with the following property.
For every $\lambda\ge 1$
and every Lebesgue measurable set $E\subset\reals^d$ satisfying $|E|=|\bb|$,
$\lambda\defb\le c$, and 
$E\symdif\bb\subset\{x: \big|\, |x|-1\,\big|\le \lambda\defb\}$,
there exists a measure-preserving transformation $\phi\in\aff(d)$ such that
\begin{gather*}
\text{$\phi(E)$ is balanced,} 
\\ \norm{\phi-I}_{\aff(d)} =O_\lambda(\defb),
\\ \phi(E)\symdif\bb\subset \{x: \big|\,1-|x|\,\big| \le C_\lambda \defb\}.
\label{eq:nearness2} 
\end{gather*}
\end{lemma}


The constant $C_\lambda$ depends on $\lambda,d$ but not on $E$.
It is the case of large $\lambda$ that is of interest.
A point worthy of notice is that the inequality in the final conclusion
takes the form $|\,1-|x|\,|\le C_\lambda|E\symdif\bb|$,
whereas the variant 
$|\,1-|x|\,|\le C_\lambda|\phi(E)\symdif\bb|$ will be required as a hypothesis
in the subsequent analysis.
Because the functional $\Phi_q$ is affine-invariant, 
at the outset we may replace $E$ by $\psi$ where $\psi\in\aff(d)$
is measure-preserving and $|\psi(E)\symdif\bb| \le 2 \inf_{\scripte\in\frakE}|E\symdif\scripte|$,
taking the infinum over all ellipsoids satisfying $|\scripte|=|E|$.
For this modified set $E$,
if $\phi\in\aff(d)$ satisfies the conclusions of Lemma~\ref{lemma:balanced}
then $|\phi(E)\symdif\bb| \ge \inf_\scripte|E\symdif\scripte| \ge \tfrac12 |E\symdif\bb|$
and therefore 
\[ |\,1-|x|\,|\le C_\lambda|E\symdif\bb| \le |\,1-|x|\,|\le 2C_\lambda|\phi(E)\symdif\bb|.\]

To prepare for the proof of Lemma~\ref{lemma:balanced},
denote by $W_2$ the real vector space of all polynomials $P:\reals^d\to\reals$
that are finite linear combinations of homogeneous harmonic polynomials of degrees $\le 2$.
Denote by $V_2$ the real vector space
of all restrictions to $S^{d-1}$ of real-valued polynomials of degrees $\le 2$.
The natural linear mapping from $W_2$ to $V_2$ induced by restriction from
$\reals^d$ to $S^{d-1}$ is a bijection.

Regard $V_2$ as a real inner product space, with the $L^2(S^{d-1},\sigma)$ inner product.
Denote by $\Pi$ the orthogonal projection of $L^2(S^{d-1})$ onto its subspace $V_2$.
Define $\scriptt: \scriptm_d\to V_2$ by
\begin{equation}
\scriptt(S)(\alpha) = \Pi(\alpha\cdot S(\alpha)),
\end{equation}
that is, the right-hand side equals the
restriction to $S^{d-1}$ of the quadratic polynomial $\reals^d\owns x\mapsto x\cdot S(x)$.

\begin{lemma} $\scriptt: \scriptm_d\to V_2$ is surjective.  \end{lemma}

\begin{proof}
The range of $\scriptt$ is
the collection of all functions $S^{d-1}\owns \alpha\mapsto S(\alpha)\cdot\alpha$,
as the function $S$ varies over all affine mappings
from $\reals^d$ to $\reals^d$. Because $S\mapsto\scriptt(S)$ is linear,
this range is a subspace of $V_2$.

Firstly, the constant function $\alpha\mapsto 1$ equals $\scriptt(S)$ when $S(x)\equiv x$,
since $S(\alpha)\cdot\alpha = \alpha\cdot\alpha\equiv 1$ for $\alpha\in S^{d-1}$.
Secondly, a linear monomial $\alpha=(\alpha_1,\dots,\alpha_d)\mapsto\alpha_k$ is expressed
by choosing $S(x)\equiv e_k$, the $k$--th unit coordinate vector.
Thirdly, to express a monomial $\alpha\mapsto\alpha_j\alpha_k$ in the form $S(\alpha)\cdot\alpha$,
define $S(x)=(S_1(x),\dots,S_d(x))$ by $S_i(x)\equiv 0$ for all $i\ne j$, and $S_j(x)=x_k$.
Then $\alpha_j\alpha_k = S(\alpha)\cdot\alpha$. 
Functions of these three types span $V_2$, so $\scriptt$ is indeed surjective.
\end{proof}

\begin{proof}[Proof of Lemma~\ref{lemma:balanced}]
If $c\le\tfrac12$ then $E$ contains the ball of radius $\tfrac12$ centered at $0$,
so if $\phi\in\aff(d)$ is sufficiently close to the identity then $\phi(E)$ contains
the ball of radius $\tfrac14$ centered at $0$.

Let $k\in\{0,1,2\}$.
Let $P:\reals^d\to\reals$ be a homogeneous harmonic polynomial of degree $k$.
Set $g(x)$ be a smooth function that agrees with $|x|^{-k}P(x)$
in $\{x: \big|\,|x|-1\,\big|\le\tfrac34\}$. 

For $\phi\in\aff(d)$ let $f_{\phi(E)}=\one_{\phi(E)}-\one_\bb$ and $F_{\phi(E)}$ 
be the functions associated to $\phi(E)$
in the same way that $f=\one_E-\one_\bb$ and $F$ are associated to $E$.
Then
\begin{align}
\int_{S^{d-1}} F_{\phi(E)}(y)P(y)\,d\sigma(y)
&= \int_{S^{d-1}}\int_0^\infty (\one_{\phi(E)}-\one_\bb)(ry)\,r^{d-1}\,dr\,P(y) d\sigma(y)
\notag
\\ &= \int_{\reals^d}  (\one_{\phi(E)}-\one_\bb)(x)|x|^{-k} P(x) \,dx
\notag
\\ &= \int_{\reals^d}  (\one_{E}\circ\phi^{-1}-\one_\bb)\, g
\notag
\\ &= \int_{\reals^d}  (f\circ\phi^{-1})\, g
+  \int_{\reals^d}  (\one_\bb\circ\phi^{-1}-\one_\bb)\, g.
\label{eq:cocycle?}
\end{align}
The second to last equation holds because both $\bb$ and $\phi(E)$ contain
the ball of radius $\tfrac14$ centered at $0$, and $g(x)\equiv |x|^{-k}P(x)$
for all $x$ in the complement of this ball.
All of these quantities depend linearly on $P$.

We seek the desired $\phi\in\aff(d)$ in the form \[\phi=I+S,\] 
where $\norm{S}_{\aff(d)}$ is small and $I$ is the identity matrix;
that is, $\phi(x)=x+S(x)$ where $S$ is an affine mapping.
The second term on the right-hand side of \eqref{eq:cocycle?} is independent of $E$.
Moreover,
\begin{align*}
\int_{\reals^d} (\one_\bb\circ\phi^{-1}) \,g
&= |\det(\phi)| \int_{\bb} g\circ\phi 
\\&= (1+\trace(S)) \int_{\bb} g\circ\phi
+ O_P(\norm{S}_{\aff(d)}^2)
\\&= (1+\trace(S)) \int_{\bb} \,g(x+S(x))\,dx 
+ O_P(\norm{S}_{\aff(d)}^2).
\end{align*}
Here and below, 
$O_P(\norm{S}_{\aff(d)}^2)$ denotes a quantity which depends linearly on $P$,
whose norm or absolute value, as appropriate, is bounded above by a constant multiple
of the norm of $P$ times the $\aff(d)$ norm squared of $S$.

Invoking the Taylor expansion of $g$ about $x$ gives
\begin{align*}
\int_{\reals^d} (\one_\bb\circ\phi^{-1}) \,g
&= (1+\trace(S)) \int_{\bb} g
+ \int_{\bb} \, \nabla g\cdot S
+ O_P(\norm{S}_{\aff(d)}^2)
\\&=
\int_{\bb} g + \int_{\bb} \big(g\trace(S) + \nabla g\cdot S\big)
+ O_P(\norm{S}_{\aff(d)}^2)
\\&= \int_{\bb} g + \int_{\bb} \diver(gS) + O_P(\norm{S}_{\aff(d)}^2)
\\& = \int_{\bb} g 
+ \int_{S^{d-1}} g(\alpha) S(\alpha) \cdot \alpha \,d\sigma(\alpha)
+ O_P(\norm{S}_{\aff(d)}^2)
\\& = \int_{\bb} g 
+ \int_{S^{d-1}} P(\alpha) S(\alpha) \cdot \alpha \,d\sigma(\alpha)
+ O_P(\norm{S}_{\aff(d)}^2).
\end{align*}
The second to last equality is justified by the divergence theorem,
and the last by the identity $g\equiv P$ on $S^{d-1}$.  Thus
\begin{equation*}
\int_{\reals^d}  (\one_\bb\circ\phi^{-1}-\one_\bb)\, g
= \int_{S^{d-1}} P(\alpha) S(\alpha) \cdot \alpha \,d\sigma(\alpha)
+ O_P(\norm{S}_{\aff(d)}^2).
\end{equation*}

Since $\int_{\reals^d} (f\circ\phi^{-1})\,g =\int_{\reals^d} (f\circ\phi^{-1})\, P(x)|x|^{-k}$,
by returning to \eqref{eq:cocycle?} we find that the equation 
$\int_{S^{d-1}} F_{\phi(E)}P\,d\sigma=0$ for a still unknown $S\in\aff(d)$ takes the form
\begin{equation} \int_{S^{d-1}} P(\alpha) S(\alpha) \cdot \alpha \,d\sigma(\alpha)
= - \int_{\reals^d} (f\circ\phi^{-1})\,P(x)|x|^{-k}\,dx
+ O_P(\norm{S}_{\aff(d)}^2).  \end{equation}
All three terms in this equation depend linearly on $P$, so we may regard
this as an equation in $V_2^*$,
or equivalently as an equation in $V_2$ since this is a Hilbert space.


Write this equation as
\begin{equation} \scriptt(S)=\scriptn_f(S) + \scriptr(S)\end{equation}
where $\scriptt$ is defined above, 
$\scriptn_f(S)$ is the mapping $P\mapsto  -\int_{\reals^d} (f\circ\phi^{-1})\,P(x)|x|^{-k}\,dx$,
and $\scriptr$ represents the term $O_P(\norm{S}_{\aff(d)}^2)$. 
Both $\scriptn_f$ and $\scriptr$ are twice continuously differentiable.
Moreover,
\begin{equation} \norm{\scriptn_f(S)}_{V_2^*}\le C\defb\end{equation}
simply because $|f|\le \one_{\defb}$ and $f$ is supported where $\tfrac12\le |x|\le \tfrac32$.
Since $\scriptt:\scriptm_d\to V_2$ is surjective, the Implicit Function Theorem
guarantees that the equation
$\scriptt(S) =\scriptn_f(S) +\scriptr(S)$,
for an unknown $S\in\scriptm_d$,
admits a solution satisfying $\|S\|_{\scriptm_d} \le C_\scriptn|E\symdif \bb|$.
\end{proof}

\begin{lemma} \label{lemma:vanishing}
Let $E\subset\reals^d$ be balanced with respect to $\bb$. Then
\begin{equation}
\iint_{S^{d-1}\times S^{d-1}} F(\alpha)F(\beta)|\alpha-\beta|^{2k}\,d\sigma(\alpha)\,d\sigma(\beta)
=0 \ \ \text{for $k\in\{0,1,2\}$.}
\end{equation}
\end{lemma}

\begin{proof}
Expand $|\alpha-\beta|^{2k}$ as a linear combination of monomials in $(\alpha,\beta)$,
and expand the double integral accordingly.
Each monomial gives rise to a double integral that factors as a product of two single integrals of the form 
\[\int_{S^{d-1}} FP\,d\sigma \, \cdot\, \int_{S^{d-1}} FQ\,d\sigma\]
where $P,Q$ are polynomials, the sum of whose degrees equals $4$.
Therefore at least one of $P,Q$ must have degree less than or equal to $2$.
The corresponding integral vanishes, by the balancing hypothesis, which asserts
that the function $F$ associated to $E$ satisfies $\int_{S^{d-1}} FP\,d\sigma=0$
for all polynomials $P$ of degrees less than or equal to $2$.
\end{proof}

\section{Stability with respect to perturbation of exponents}\label{section:neareven}

Let $d\ge 1$.
Let $B\subset\reals^d$ be the closed ball centered at $0$
satisfying $|B|=1$. Let $r_0=r_0(d)$ be its radius.

\begin{definition}
For each $d\ge 1$, 
$G_d$ denotes the set of all exponents $q\in (q_d,\infty)$
with the following properties.
\newline
(i) There exists $c>0$
such that for all sets $E\subset\reals^d$ satisfying $|E|=1$,
\begin{equation} \label{eq:desid}
\norm{\widehat{\one_E}}_q \le \norm{\widehat{\one_{E^\star}}}_q 
-c\inf_{\scripte} |E\symdif \scripte|^2\end{equation}
where the infinum is taken over all ellipsoids $\scripte$
satisfying $|\scripte|=|E|=1$.
\newline
(ii) 
$x\cdot\nabla K_q(x)|_{|x|=r_0}<0$.
\newline
(iii) For each $\eta\in (0,r_0)$, 
\[\min_{|x|\le r_0-\eta} K_q(x) > \max_{|x|\ge r_0+\eta} K_q(x).\]
\end{definition}

The goal of this section is to establish the following stability result.
\begin{proposition} \label{prop:Gdopen}
For each $d\ge 1$, $G_d$ is an open subset of $\reals^+$.
\end{proposition} 

We have already shown that $G_d$ contains each even integer $q\ge 4$.
Therefore Theorem~\ref{thm:neareven} is an immediate consequence of Proposition~\ref{prop:Gdopen}.

\begin{proof}[Proof of Proposition~\ref{prop:Gdopen}]
Let $\barq\in G_d$.
Let $B\subset\reals^d$ be the closed ball centered at $0$
with Lebesgue measure equal to $1$; equivalently,
$B$ has radius equal to $r_0$.

According to Lemmas~\ref{lemma:thekernels1} and \ref{lemma:thekernels2}, 
properties (ii) and (iii) continue to hold for all exponents $q$ 
sufficiently close to $\barq$. Thus it remains only to treat property (i).

Let $\delta_0>0$. 
By condition (i) of the definition of $G_d$,
there exists $c_0>0$ such that $\norm{\widehat{\one_E}}_{\barq}
\le \norm{\widehat{\one_B}}_{\barq}-c_0$
for all Lebesgue measurable sets $E\subset\reals^d$
satisfying $|E|=1$ and $\inf_{|\scripte|=1}|E\symdif \scripte| \ge\delta_0$.
The functions
$q\mapsto \norm{\widehat{\one_E}}_q$
are continuous in a neighborhood of $\barq$,
uniformly over all such sets $E$.
Therefore there exists $\eta>0$ such that 
$\norm{\widehat{\one_E}}_{q} \le \norm{\widehat{\one_B}}_{q}-\tfrac12 c_0$
whenever $|E|=1$,
$\inf_{|\scripte|=1}|E\symdif \scripte| \ge\delta_0$,
and $|q-\barq|\le\eta$.
Thus it suffices to prove that if $\barq\in G_d$
then there exist $\delta_0,\eta,c>0$
such that the desired inequality \eqref{eq:desid} holds for all sets satisfying $|E|=1$ and
$\inf_{|\scripte|=1}|E\symdif \scripte| \le\delta_0$,
for all exponents satisfying $|q-\barq|\le\eta$.

Next, let $q$ be close to $\barq$, let $|E|=1$,
and assume that $\delta = \inf_\scripte|E\symdif\scripte|$
satisfies $\delta\le\delta_0$.
By making an affine change of variables we may reduce matters to the case
in which $|E\symdif B|\le 2\delta$.
Let $\lambda$ be a large constant. Construct a set $E^\dagger$
with the properties indicated in \S\ref{section:localization}.
In particular, $|E^\dagger|=1$, $E^\dagger\symdif B\subset E\symdif B$, and 
\[\{x\in E\symdif B: \big|\,|x|-r_0\,\big|>\lambda\delta\}\subset E\symdif E^\dagger.\]
Then
\[ \norm{\widehat{\one_E}}_q \le \norm{\widehat{\one_{E^\dagger}}}_q
+ c\langle K_q, f\rangle + O(\delta^{2+\rho} \]
where $c,\rho>0$ and $f = \one_E-\one_{E^\dagger}$.
The assumptions (ii) and (iii) of the definition of $G_d$
remain uniformly valid for all exponents $q$ sufficiently close to $\barq$.
As in the above analysis of the case in which $q$ was an even integer, it follows that
\[ \norm{\widehat{\one_E}}_q \le \norm{\widehat{\one_{E^\dagger}}}_q
- c\lambda\delta|E\symdif E^\dagger| + O(\delta^{2+\rho}).  \]
The same analysis shows that
\[ \norm{\widehat{\one_{E^\dagger}}}_q\le \norm{\widehat{\one_B}}_q
+O(\delta^2) \]
uniformly in $\lambda$.
If $\lambda$ is chosen to be sufficiently large then combining these two inequalities
gives
\[ \norm{\widehat{\one_E}}_q \le \norm{\widehat{\one_B}}_q -c\delta^2 \]
unless $|E\symdif E^\dagger|\le \tfrac1{10}|E\symdif B|$, and gives
\[ \norm{\widehat{\one_E}}_q \le \norm{\widehat{\one_{E^\dagger}}}_q + O(\delta^{2+\rho}) \]
otherwise.  In this latter case,
$\inf_\scripte |E^\dagger\symdif\scripte| \ge c\delta$.
Therefore we have reduced matters to establishing the inequality
\eqref{eq:desid} under the supplementary assumption that 
\[ E\symdif B \subset\{x: \big|\,|x|-r_0\,\big|\le \lambda\delta\}\]
where $\delta = \inf_\scripte|E\symdif \scripte|$ satisfies $\delta\le\delta_0$ 
and $\lambda$ is a finite constant independent of $\delta$.
A simple consequence of these assumptions is that $|E\symdif B|\le C\delta$.

By Lemma~\ref{lemma:balanced}, it suffices to analyze only the case in which $E$ is balanced,
which we now assume.
Then
\[ \norm{\widehat{\one_E}}_q^q = \norm{\widehat{\one_\bb}}_q^q
+q\langle K_{q},f\rangle +\tfrac14 q^2 \langle f*L_{q},f\rangle
+\tfrac14 q(q-2)  \langle f*L_{q},\tilde f\rangle +O(|E\symdif\bb|^{2+\rho}).  \]
Each term in this expansion is suitably close to the corresponding term for $\barq$.  Indeed,
\[ \langle f*L_q,f\rangle = \langle f*L_{\barq},f\rangle + O(\norm{L_{\barq}-L_q}_\infty\norm{f}_1^2)
= \langle f*L_{\barq},f\rangle + o_{q-\barq}(1)\cdot|E\symdif B|^2)\] 
where $o_{q-\barq}(1)$ indicates a quantity that is majorized by a quantity
that tends to zero as $|q-\barq|\to 0$, uniformly in $E$ under the indicated
assumptions on $E$, provided that $\delta_0$ is sufficiently small.
Likewise,
\[ \langle f*L_q,\tilde f\rangle
= \langle f*L_{\barq},\tilde f\rangle
+ o_{q-\barq}(1)\cdot|E\symdif B|^2).\]
In the same way,
\begin{align*} \langle K_q,f\rangle 
&= \int (K_q(x)-a_q)f(x)\,dx
\\& = \int (K_{\barq}(x)-a_{\barq})f(x)\,dx
+ O(\delta\norm{\nabla K_q-\nabla K_{\barq}}_{L^\infty}\norm{f}_1)
\\& =o_{q-\barq}(|E\symdif B|^2).
\end{align*}
Adding together these bounds gives
\begin{align*}
\norm{\widehat{\one_E}}_q- \norm{\widehat{\one_B}}_q
&\le
\norm{\widehat{\one_E}}_{\barq}- \norm{\widehat{\one_B}}_{\barq}
+ o_{q-\barq}(1)\delta^2
\\ &\le -c\delta^2 + o_{q-\barq}(1)\delta^2,
\end{align*}
completing the proof.
\end{proof}

\section{Large exponents}
In this section we examine asymptotic behavior as $q\to\infty$ while the dimension
$d$ remains fixed, proving Theorem~\ref{thm:local}. 
We will prove that if $q$ is sufficiently large and $\delta$ is sufficiently small,
if $E\subset\reals^d$ satisfies $|E|=|\bb|$,
and if $\defe\le\delta$ then
$\norm{\widehat{\one_E}}_q^q\le \norm{\widehat{\one_\bb}}_q^q
- c(q,d)|\defe|^2$ 
with $c(q,d)$ of order of magnitude $\omega_d^q q^{-(d+2)/2}$ as $q\to\infty$.

The notation $O(\psi(q)))$ indicates a quantity whose
absolute value is bounded above by $C\psi(q)$ where $C$ may depend on $d$
and on other parameters but is independent of $q$ provided only that $q$ is sufficiently large. 

\subsection{The functions $K_q$ and $L_q$ for large $q$}

$L_{q}$ is by definition the inverse Fourier transform of $|\widehat{\one_{\bb}}|^{q-2}$.
The function $|\widehat{\one_\bb}|$
has maximum value $\omega_d=|\bb|$, and achieves this value only
at the origin. In studying $L_{q}$ for large $q$,
it is natural to examine the normalized function $\omega_d^{2-q}L_{q}$; 
likewise $\omega_d^{1-q}K_q$.  Recall that $K_{q}$ is real-valued,
and that $K_q,L_q$ are three times continuously differentiable
globally bounded functions for all sufficiently large $q$,
which tend to zero as $|x|\to\infty$.

\begin{lemma} \label{lemma:largeq}
Let $d\ge 1$. 
The following hold for all sufficiently large $q<\infty$.
Firstly, $L_q(x)$ is a three times continuously differentiable bounded function, 
which tends to zero as $|x|\to\infty$.
Secondly, there exist $\kappa_j=\kappa_j(q,d)$ 
such that as functions of $x$ in any compact subset of $\reals^d$,
\begin{equation}
L_{q}(x) =   \kappa_0 + \kappa_1  |x|^2 + \kappa_2 |x|^4 
+ O(\omega_d^{q}q^{-(d+6)/2}).
\end{equation}
This holds in the $C^3$ norm as a function of $x$ on any fixed compact subset of $\reals^d$.
\end{lemma}
The proof gives explicit asymptotic values for $\kappa_j$, and expansions 
modulo lower order terms, but the contributions of these three leading terms 
in the expansion for $L_q$ will vanish in the analysis below, so that only the 
order of magnitude of the remainder term will subsequently be relevant.

\begin{proof}
Express
\begin{equation} \widehat{\one_\bb}(\xi) = \int_{|x|\le 1} e^{-2\pi i x\cdot\xi}\,dx
= \omega_{d-1} \int_{-1}^1 e^{-2\pi is|\xi|} (1-s^2)^{(d-1)/2}\,ds.  \end{equation}
In particular, $\widehat{\one_\bb}(0)=\omega_d$.
It is immediate that
\begin{equation}
\nabla(\widehat{\one_\bb})(0)=0 \ \text{ and }\ 
|\widehat{\one_{\bb}}(\xi)|<\omega_d \ \text{for all $\xi\ne 0$.}
\end{equation}

Expanding the exponential factor $e^{-2\pi is|\xi|}$ in the  last integral in Maclaurin series, 
terms of odd degree with respect to $s$ contribute zero. Therefore 
for $\xi$ in any bounded set,
\begin{align*}
\widehat{\one_\bb}(\xi)
& =  \omega_{d-1} \int_{-1}^1 \big(1 + \tfrac12 (-2\pi i|\xi|s)^2 + O(|\xi|^4)\big)
(1-s^2)^{(d-1)/2}\,ds 
\\& = \omega_d  
- 2\pi^2\omega_{d-1} |\xi|^2 \int_{-1}^1 s^2(1-s^2)^{(d-1)/2}\,ds
+ O(|\xi|^4)
\\& = \omega_d (1- \pi\rho_d  |\xi|^2)
+ O(|\xi|^4)
\end{align*}
where $\rho_d$ is defined by
\begin{equation} \label{rhodefn}
\rho_d  = 2\pi \omega_{d-1}\omega_d^{-1} \int_{-1}^1 s^2(1-s^2)^{(d-1)/2}\,ds.  \end{equation}
Provided that $|\xi|$ is sufficiently small, this can be rewritten as
\begin{equation}
\omega_d^{-1} \widehat{\one_\bb}(\xi) 
= \exp\big(-\pi \rho_d |\xi|^2 + O(|\xi|^4)\big).
\end{equation}
Therefore for $|\xi|$ in some neighborhood of $0$ that depends only on $d$,
\begin{equation} \label{eq:CLT} \omega_d^{2-q}|\widehat{\one_\bb}(\xi)|^{q-2} 
 = e^{-\pi \rho_d(q-2) |\xi|^2 + O(q|\xi|^4)}.  \end{equation}
Moreover for $\xi$ in any fixed bounded region,
\begin{equation} \big|\omega_d^{-1}\widehat{\one_\bb}(\xi)\big|
\le e^{-c|\xi|^2}\end{equation}
where $c>0$ depends only on the dimension $d$,
while $|\widehat{\one_\bb}(\xi)|= O(|\xi|^{-(d+1)/2})$ as $|\xi|\to\infty$.

Therefore for arbitrary $k\ge 0$,
\[ \int_{|\xi|\ge q^{-1/4}} |\xi|^k\cdot |\omega_d^{-1} \widehat{\one_\bb}(\xi)|^{q-2}\,d\xi
= O(q^{-N})\ \text{ for all $N$}\]
as $q\to\infty$ while $k$ remains fixed.
Therefore uniformly for all $x\in\reals^d$, for any $N<\infty$,
\begin{align} \omega_d^{2-q} L_{q}(x)
& =  \int_{\reals^d} e^{2\pi i x\cdot\xi} 
\omega_d^{2-q} |\widehat{\one_\bb}(\xi)|^{q-2}\,d\xi
\notag
\\ & =  \int_{\reals^d} e^{2\pi i x\cdot\xi} 
h(\xi)\,d\xi
\label{forLq}
\end{align}
where
$h = h_{q,d}$ is a radially symmetric function that satisfies
\begin{equation} 
\int_{\reals^d}
|\xi|^k\,h(\xi)\,d\xi
 = O(q^{-(d+k)/2}) 
\end{equation}
for any nonnegative real number $k$.

If $k$ is a nonnegative integer then by virtue of the radial symmetry of $h$,
\begin{equation}\label{radialityconsequence}
\int_{|\xi|\le q^{-1/4}} (x\cdot\xi)^k h(\xi)\,d\xi
= c_{k,q} |x|^k \end{equation}
where $c_{k,q}=0$ for all odd $k$
and $c_{k,q} = O_k(q^{-(d+2+k)/2})$.
Expand \[e^{2\pi i x\cdot\xi} = \sum_{k=0}^5 (2\pi i x\cdot\xi)^k/k!  + O(|x\cdot\xi|^6)\]
and invoke \eqref{radialityconsequence} to obtain
\begin{equation*}
\int_{|\xi|\le q^{-1/4}} e^{2\pi i x\cdot\xi} h(\xi)\,d\xi
= c_0(q) + c_1(q)|x|^2  + c_2(q)|x|^4 + O(q^{-(d+6)/2})
\end{equation*}
where $c_j(q)$ are independent of $x$ 
while the remainder term $O(q^{-(d+6)/2})$ 
satisfies this upper bound uniformly for all $x$ in an arbitrary compact set.
\end{proof}

Recall the notation $\gamma_{q,d} = -x\cdot\nabla K_{q}(x)\big|_{|x|=1}$.

\begin{lemma}\label{lemma:largeqgamma}
Let $d\ge 1$.  There exists a real number $\kappa>0$ that depends only on the dimension $d$ 
such that as $q\to\infty$,
\begin{equation} \gamma_{q,d} = \kappa q^{-(d+2)/2}\omega_d^{q}  + O(q^{-(d+4)/2}\omega_d^{q}). 
\end{equation}

For any bounded set $S\subset\reals^d$ there exists $r<\infty$
such that for all $q\ge r$ and all $x,y\in S$,
\begin{equation} K_{q}(x)>K_{q}(y) \text{ whenever $|x|<1<|y|$.} \end{equation}
\end{lemma}

\begin{proof}
Let $S\subset\reals^d$ be a bounded set, and consider any point $x\in S$.
In the expression
\begin{equation}
x\cdot\nabla K_{q}(x)
= \int e^{ 2\pi ix\cdot\xi} (2\pi i x\cdot\xi) \widehat{\one_\bb}(\xi)|\widehat{\one_\bb}(\xi)|^{q-2}\,d\xi, 
\end{equation}
expand $e^{2\pi i x\cdot\xi}=\sum_{k=0}^2 (2\pi i x\cdot\xi)^k/k! + O(|x\cdot\xi|^3)$
and argue as in the proof of Lemma~\ref{lemma:largeq} to obtain
\begin{align*} x\cdot\nabla K_{q}(x)
&= -4\pi^2  \int (x\cdot\xi)^2 \widehat{\one_\bb}(\xi)|\widehat{\one_\bb}(\xi)|^{q-2}\,d\xi
+ O(|x|^4 q^{-(d+4)/2})
\\& = -4\pi^2 |x|^2  \int |\xi|^2 \widehat{\one_\bb}(\xi)|\widehat{\one_\bb}(\xi)|^{q-2}\,d\xi
+ O(|x|^4 q^{-(d+4)/2}).  \end{align*}
Comparing with the proof of Lemma~\ref{lemma:largeq}, there is one extra factor of $\widehat{\one_\bb}$
in the integral here. This makes no essential difference in the inequalities.

The last line is obtained by expanding $(x\cdot\xi)^2 = \sum_{j,k} x_jx_k\xi_j\xi_k$ and noting that for any $j\ne k$,
\[ \int_{\reals^d} \xi_j\xi_k \widehat{\one_\bb}(\xi)|\widehat{\one_\bb}(\xi)|^{q-2}\,d\xi=0\]
since $\widehat{\one_\bb}$ is radial.
By using \eqref{eq:CLT} one obtains
\begin{equation} \int_{\reals^d} |\xi|^2 \widehat{\one_\bb}(\xi)|\widehat{\one_\bb}(\xi)|^{q-2}\,d\xi
= \kappa \omega_d^q q^{-(d+2)/2} +O(\omega_d^q q^{-(d+4)/2}) \end{equation}
where $\kappa=\kappa(d)>0$.
\end{proof}

We require some global control over $K_q$.
\begin{lemma} \label{lemma:Kglobalcontrol}
Let $d\ge 1$. For all sufficiently large exponents $q<\infty$,
$K_q$ satisfies
\begin{gather} 
\label{eq:Knecessarystronger2}
\min_{|x|\le 1-\delta} K_{q}(x) > \max_{|x|\ge 1+\delta} K_{q}(x)
\ \text{for all $\delta>0$} 
\\
K(x)-K(y) \ge c_q\big|\,|x|-|y|\,\big|
\ \text{whenever $|x|\le 1\le |y|\le 2$}
\end{gather}
where $c_q>0$.
\end{lemma}
These are the hypotheses \eqref{eq:Knecessarystronger} and \eqref{eq:strongnecessary} 
of Lemma~\ref{lemma:nearboundary}.

\begin{proof}
It suffices to establish the first conclusion, for the second follows from the first
together with Lemma~\ref{lemma:largeqgamma}.

As in the proof of Lemma~\ref{lemma:largeq},
$\omega_d^{1-q}K_q(x)$ is the inverse Fourier transform of the function
$e^{-\pi\rho_d(q-1)|\xi|^2 + O(q|\xi|^4)}\one_{|\xi|\le q^{-1/4}}$
plus $O(q^{-N})$ for all $N$.
The main term equals
\[ (q-1)^{-d/2} \int_{|\xi|\le q^{1/4}} e^{-\pi\rho_d |\xi|^2+O(q^{-1}|\xi|^4)} e^{2\pi i y\cdot\xi}\,d\xi\]
plus $O(q^{-N})$ for all $N$, where $y = (q-1)^{-1/2}x$.
Now the two functions
$e^{-\pi\rho_d |\xi|^2+O(q^{-1}|\xi|^4)} \one_{|\xi|\le q^{1/4}}$
and
$e^{-\pi\rho_d |\xi|^2}$
differ by $O(q^{-1})$ in $L^1$ norm. Therefore uniformly for all $x\in\reals^d$
and all sufficiently large $q$,
\[ \omega_d^{1-q} K_q(x) = (q-1)^{-d/2}\int e^{-\pi\rho_d|\xi|^2}e^{2\pi iy\cdot\xi}\,d\xi
+ O(q^{-(d+2)/2})\]
where $y = (q-1)^{-1/2}x$.
The leading term is a positive Gaussian function of $y$, and has order of magnitude $q^{-d/2}$
for $y$ in any compact set. 
\eqref{eq:Knecessarystronger2} follows from the monotonicity of Gaussians together with these asymptotics.
\end{proof}

\subsection{The coronal case}
By the coronal case, we mean that in which $E$ is close to the unit ball
in the strong sense that $E\symdif\bb\subset\{x: \big|\,|x|-1\,\big|\le\lambda|E\symdif\bb|\}$
for a suitable constant $\lambda$, which is at our disposal and is to depend only
on the dimension $d$ and on the exponent $q$.

\begin{lemma} \label{lemma:coronalcase}
Let $d\ge 1$ and let $q$ be sufficiently large.
There exists $\lambda_0=\lambda_0(q,d)<\infty$ with the following property.
Let $\lambda\ge \lambda_0$.
Suppose that $E\subset\reals^d$ satisfies
$|E|=|\bb|$, $\defb\le 2|E|\cdot\defe$, $\lambda \defb\le 1$, and
\begin{equation} E\subset\{x: \big| |x|-1\big|\le\lambda\defb\}.\end{equation}
Then 
\begin{equation} \label{tobeproved}
\norm{\widehat{\one_{E}}}_q^q \le \norm{\widehat{\one_{\bb}}}_q^q
-c\defe^2 + O_\lambda(\defe^{2+\varrho}) \end{equation}
where $c,\varrho$ are positive constants that depend on $d,q$ but are independent of $\lambda$.
\end{lemma}

\begin{proof}
If $\lambda_0$ is sufficiently small then by Lemma~\ref{lemma:balanced}, 
there exists a measure-preserving affine automorphism $\phi\in\aff(d)$
such that $\phi(E)$ is balanced with respect to $\bb$,
and $\phi(E)\subset\{x: |x|\le 1+O_\lambda(1)\defe\}$. 
By its definition,
$\defe= {\dist(\phi(E),\frakE)}$ 
satisfies $\defe \le |E|\cdot |\phi(E)\symdif\bb|$.
Therefore
$\phi(E)\subset\{x: |x|\le 1+O_\lambda(1)|\phi(E)\symdif\bb|\}$. 
Replace $E$ by $\phi(E)$ for the remainder of this proof,
recalling that $\norm{\widehat{\one_{\phi(E)}}}_q = \norm{\widehat{\one_E}}_q$.

Let $f=\one_{E}-\one_\bb$ and let
$a,b,F$ be the functions associated to $E,f$ as in Definition~\ref{Fdefn}.
We seek upper bounds for 
$|\langle L_{q}*f,f\rangle|$ and $|\langle L_{q}*f,\tilde f\rangle|$,
and a negative upper bound for $\langle K_{q},f\rangle$
for all sufficiently large exponents $q$,
strong enough to guarantee that in magnitude,
$\langle K_q,f\rangle$ dominates the other two terms. 
Lemma~\ref{lemma:Kglobalcontrol} guarantees that the hypotheses of 
Lemma~\ref{lemma:nearboundary} are satisfied for all sufficiently large $q$.
According to the latter lemma,
\begin{align*} \langle K_{q},f \rangle
&\le -\tfrac12 \gamma(q,d) \int (a^2+b^2)\,d\sigma
+O_\lambda(|E\symdif \bb|^{2+\varrho})
\\& \le -c q^{-(d+2)/2}\omega_d^q \int (a^2+b^2)\,d\sigma
+O\big(\omega_d^q q^{-(d+4)/2}\norm{F}_{L^2}^2 \big) +O_\lambda(|E\symdif \bb|^{2+\varrho})
\end{align*}
where $c>0$ is independent of $q$ so long as $q$ is sufficiently large.
Since $\norm{F}_{\lt}^2 =\int (a-b)^2\,d\sigma \le \int (a^2+b^2)\,d\sigma$
(recall that $a,b$ are nonnegative by their definitions),
by replacing $c$ by $c/2$ we obtain
\begin{equation*} \langle K_{q},f \rangle
\le -c q^{-(d+2)/2}\omega_d^q \int (a^2+b^2)\,d\sigma
+O_\lambda(|E\symdif \bb|^{2+\varrho}).\end{equation*}

We have shown in Lemma~\ref{lemma:nearness} that
\begin{align*}
\langle L_{q}*f,f\rangle
&  = \iint_{S^{d-1}\times S^{d-1}} 
F(\alpha)F(\beta) L_{q}(\alpha-\beta)
\,d\sigma(\alpha)\,d\sigma(\beta)
+ O_\lambda(|E\symdif \bb|^{2+\varrho}).
\end{align*}
According to Lemma~\ref{lemma:largeq}, 
$L_q(\alpha-\beta) = 
\kappa_0 + \kappa_1|\alpha-\beta|^2 
+\kappa_2|\alpha-\beta|^4 + O(\omega_d^q q^{-(d+6)/2})$
where $\kappa_k$ depend on $q,d$ but are independent of $\alpha, \beta$. 
The contributions of $\kappa_k|\alpha-\beta|^{2k}$ to the double
integral vanish for $k=0,1,2$, by Lemma~\ref{lemma:vanishing}.
Therefore
\begin{equation}
\langle L_q*f,f\rangle = O\big(\omega_d^q q^{-(d+6)/2}\norm{F}_{L^2}^2 \big). 
\end{equation}
By the same reasoning,
$\langle L_q*f,\tilde f\rangle = O\big(\omega_d^q q^{-(d+6)/2}\norm{F}_{L^2}^2 \big)$. 

Recall that $F = a-b$.
In all, we have shown that
\begin{align*}
q\langle K_q,f\rangle
& + \tfrac14 q^2 \langle L_q*f,f\rangle + \tfrac14 q(q-2)\langle L_q*f,\tilde f\rangle
\\ & \le 
-cq\cdot q^{-(d+2)/2}\omega_d^q \int (a^2+b^2)\,d\sigma
+ O(q^2 q^{-(d+6)/2}\omega_d^q \norm{F}_{L^2}^2)
+O_\lambda(\defb^{2+\varrho})
\\ & = 
-c\cdot q^{-d/2}\omega_d^q \int (a^2+b^2)\,d\sigma
+ O(q^{-(d+2)/2}\omega_d^q \norm{F}_{L^2}^2)
+O_\lambda(\defb^{2+\varrho})
\\ & \le 
-c'\cdot q^{-d/2}\omega_d^q \int (a^2+b^2)\,d\sigma
+O_\lambda(\defb^{2+\varrho})
\end{align*}
where $c',\varrho>0$ for all sufficiently large exponents $q$. 
The crucial point is that for large $q$, the factor $\omega_d^q q^{-d/2}$ in the leading term 
dominates the factor $\omega_d^q q^{-(d+2)/2}$ in the remainder term 
$O(q^{-(d+2)/2}\omega_d^q \norm{F}_{L^2}^2)$.

Since $\norm{F}_{L^2}=O(|E\symdif\bb|)$
and $\int (a^2+b^2)\,d\sigma \ge \tfrac12 \sigma(S^{d-1})^{-1} |E\symdif\bb|^2$,
this yields
\begin{align*}
\norm{\widehat{\one_{E}}}_q^q
&\le \norm{\widehat{\one_\bb}}_q^q
-c \omega_d^q q^{-d/2} |E\symdif\bb|^2
+O_\lambda(|E\symdif \bb|^{2+\varrho})
\\&\le \norm{\widehat{\one_\bb}}_q^q
-c \omega_d^q q^{-d/2} \defe^2
+O_\lambda(\defe^{2+\varrho})
\end{align*}
since $|E\symdif\bb|\ge |E|\cdot\defe$.
\end{proof}

\subsection{Reduction to the coronal case}
The next step is to show how the case of sets satisfying
the weaker condition $|E\symdif\bb|\ll |E|$ can be reduced to the coronal case,
completing the proof of Theorem~\ref{thm:local}.
Recall the notation $E_\eta$ introduced in \eqref{eq:nearnotation}
and employed in Corollary~\ref{cor:corona};
$E_\eta=\{x\in E: \big|\,|x|-1\,\big|>\eta\}$.

\begin{proof}[Proof of Theorem~\ref{thm:local}]
Let $\lambda$ be a large constant to be chosen below.
Suppose that $|E|=|\bb|$, and that $\bb$ is a quasi-optimal approximation
to $E$ among all ellipses of measure $|E|$ in the sense that
\[|E\symdif \bb|\le 2\inf_{|\scripte|=|E|}|E\symdif\scripte|= 2|E|\cdot \defe\ll 1.\] 

Express $E$ as $E=(\bb\cup A)\setminus B$ where $A\subset\bb$, $B\cap\bb=\emptyset$,
and $|A|=|B|$.
Set $\eta = \lambda\defb$ and consider $E_\eta,A_\eta,B_\eta$. 
Define sets $B^\dagger\subset B\setminus B_\eta$ and $A^\dagger\subset A\setminus A_\eta$ as follows.
If $|B_\eta|\ge |A_\eta|$ choose a measurable set $A'_\eta\subset A_\eta$
satisfying $|A'_\eta|=|B_\eta|$ and define 
$B^\dagger=B\setminus B_\eta$ and $A^\dagger = A\setminus A'_\eta$.
If $|A_\eta|\ge |B_\eta|$ choose a measurable set $B'_\eta\subset B_\eta$
satisfying $|B'_\eta|=|A_\eta|$, and define
$B^\dagger=B\setminus B'_\eta$ and
$A^\dagger = A\setminus A_\eta$.
Define 
\begin{align} 
E^\dagger &= (\bb \cup A^\dagger)\setminus B^\dagger 
\label{eq:Edagger} 
\\ f^\dagger &= \one_{A^\dagger}-\one_{B^\dagger}  
\end{align}
and $f_\eta = f-f^\dagger = \one_{A\setminus A^\dagger} - \one_{B\setminus B^\dagger}$.
Then $|E^\dagger|=|\bb|$ and
\begin{equation} |E\symdif\bb| = |E\symdif E^\dagger|+|E^\dagger\symdif\bb|.  \end{equation}
Let $\tilde f^\dagger(x)=f^\dagger(-x)$.

By Corollary~\ref{cor:corona}, for all sufficiently large $q$,
\begin{equation} \label{eq:daggerized}
\norm{\widehat{\one_E}}_q^q
 \le  \norm{\widehat{\one_{E^\dagger}}}_q^q
-c\lambda|E\symdif\bb|\cdot|E\symdif E^\dagger| + O_\lambda(|E\symdif\bb|^{2+\varrho}).
\end{equation}
Since $|E^\dagger|=|E|$, 
\begin{align*} \inf_{|\scripte|=|E^\dagger|} |E^\dagger\symdif\scripte|
= \inf_{|\scripte|=|E|} |E^\dagger\symdif\scripte|
\le |E^\dagger\symdif\bb| \le  |E\symdif \bb|.  \end{align*}
Moreover, a simplification of the analysis above establishes the simple bound
\begin{equation} \norm{\widehat{\one_{E^\dagger}}}_q^q
\le \norm{\widehat{\one_\bb}}_q^q + O(|E^\dagger\symdif\bb|^2) \end{equation}
since $\langle K_q,f^\dagger\rangle$ is nonpositive for $q$ sufficiently large.
Together these give
\begin{align*} \norm{\widehat{\one_E}}_q^q
 \le  \norm{\widehat{\one_{\bb}}}_q^q -c\lambda|E\symdif\bb|\cdot|E\symdif E^\dagger| 
+ O(|E\symdif\bb|^2).  \end{align*}
From this we deduce the desired conclusion
$ \norm{\widehat{\one_E}}_q^q \le  \norm{\widehat{\one_{\bb}}}_q^q -c|E\symdif\bb|^2$ 
unless
\begin{equation}\label{daggerclose} |E\symdif E^\dagger| \le C \lambda^{-1} |E\symdif\bb|.  \end{equation}
Here the constant $C$ depends only on $q,d$.

For the remainder of the proof we may assume that \eqref{daggerclose} holds.
If $\lambda$ is sufficiently large, then \eqref{daggerclose} together with the relation
$|E\symdif\bb| = |E\symdif E^\dagger| + |E^\dagger\symdif\bb|$ imply that 
\[ |E^\dagger\symdif\bb| \ge \tfrac12 |E\symdif\bb|.\]
This condition ensures that
\begin{equation} E^\dagger\symdif\bb \subset E\symdif\bb
\subset\{x: \big|\,|x|-1\,\big| \le\lambda|E\symdif\bb|\}
\subset\{x: \big|\,|x|-1\,\big| \le 2\lambda |E^\dagger \symdif\bb|.\} \end{equation}
Moreover, for any ellipsoid $\scripte$ satisfying $|\scripte|=|E^\dagger|=|E|$,
\[ |E^\dagger\symdif\scripte| \ge |E\symdif\scripte| - |E\symdif E^\dagger|
\ge \tfrac12 |E\symdif\bb| -|E\symdif E^\dagger|
\ge (\tfrac 12-C\lambda^{-1})|E\symdif\bb|
\ge \tfrac14|E\symdif\bb|\]
provided that $\lambda$ is sufficiently large.
Thus $\dist(E^\dagger,\frakE)\ge \tfrac14 \defe$.

By invoking first \eqref{eq:daggerized}, then Lemma~\ref{lemma:coronalcase}, we obtain
\begin{align*}
\norm{\widehat{\one_E}}_q^q
 &\le  \norm{\widehat{\one_{E^\dagger}}}_q^q
-c\lambda|E\symdif\bb|\cdot|E\symdif E^\dagger| + O_\lambda(|E\symdif\bb|^{2+\varrho})
\\ &
\le \norm{\widehat{\one_{\bb}}}_q^q
-\zeta\dist(E^\dagger,\frakE)^2 + O_\lambda(\dist(E^\dagger,\frakE)^{2+\varrho})
\\& \qquad \qquad \qquad
-c\lambda|E\symdif\bb|\cdot|E\symdif E^\dagger| + O_\lambda(|E\symdif\bb|^{2+\varrho})
\\ &
\le \norm{\widehat{\one_{\bb}}}_q^q
-\tfrac14 \zeta \defe^2
+ O_\lambda(\defe^{2+\varrho})
\end{align*}
where $\zeta=\zeta(d,q)$ is a positive constant.
\end{proof}

\section{$q=4$ analysis for dimension $d=2$}

Here we prove Theorem~\ref{thm:q4d2}. Let $d=2$ and $q=4$.  Define
\begin{align} K=K_{4}&= \one_\bb*\one_\bb*\one_\bb \\ L =L_{4}&=\one_\bb*\one_\bb.  \end{align}
These are radially symmetric functions on $\reals^2$.
The same reasoning as employed above for the case of large exponents $q$ reduces 
Theorem~\ref{thm:q4d2} to the following special case.

\begin{proposition} \label{prop:q4}
Let $E\subset\reals^2$ be a Lebesgue measurable set satisfying $|E|=|\bb|$.
Suppose that $E$ is balanced and lies close to $\bb$ in the sense \eqref{eq:nearness}.
Then
\begin{equation}
\norm{\widehat{\one_E}}_{L^4}^4 \le 
\norm{\widehat{\one_\bb}}_{L^4}^4  
-\tfrac{8}{5\pi} |E\symdif\bb|^2
+O(|E\symdif\bb|^{2+\varrho})
\end{equation}
where $\varrho>0$.
\end{proposition}

As in the analysis of the case of large exponents $q$ developed above,
the proof of Proposition~\ref{prop:q4} reduces to the estimation of
\begin{multline} -\tfrac{1}2 q \gamma_{4,2} \int_{S^{1}} (a^2+b^2)\,d\sigma
+ \tfrac14 q^2 \scriptq(F,F) + \tfrac14 q(q-2) \scriptq(F,\tilde F)
\\= -2 \gamma_{4,2} \int_{S^{1}} (a^2+b^2)\,d\sigma
+ 4 \scriptq(F,F) + 2 \scriptq(F,\tilde F) \end{multline}
where $\scriptq$ is the quadratic form defined in \eqref{Qdefn}
and $\gamma_{4,2} = -x\cdot\nabla K\big|_{|x|=1}$.

\subsection{Fourier series analysis}

Identify $S^1$ with $[0,2\pi]$ via 
$S^1\owns \alpha = (\cos(\theta),\sin(\theta))$.
The quadratic form $\scriptq=\scriptq_{4,2}$ defined in \eqref{Qdefn} becomes
\begin{equation}
\scriptq(F,G) 
=\iint_{S^1\times S^1} F(\alpha)G(\beta) L(\alpha-\beta)\,d\sigma(\alpha)\,d\sigma(\beta)
= \langle F*\scriptl,G\rangle
\end{equation}
where $*$ denotes convolution on $S^1$ and
\begin{equation} \scriptl(\theta) 
= L(x) \text{ where } |x|=\sqrt{2-2\cos(\theta)}). \end{equation} 

We utilize the Fourier transform, mapping $L^2(S^1)$ to $\ell^2(\integers)$,
normalized by 
\[\widehat{h}(n) = (2\pi)^{-1} \int_0^{2\pi} h(\theta)e^{-in\theta}\,d\theta.\]
Thus $\int_0^{2\pi}g\,\overline h = 2\pi \sum_{n=-\infty}^\infty \widehat{h}(n)\overline{\widehat{g}(n)}$,
and $\widehat{g*h} = 2\pi\, \widehat{g}\,\widehat{h}$.
Thus
\begin{equation} \label{Qparseval}
\scriptq(F,G) 
=  2\pi\sum_{n\in\integers} \widehat{F*\scriptl}(n)\,\overline{\widehat{G}(n)} 
 =  4\pi^2\sum_{n\in\integers} \widehat{F}(n)\,\overline{\widehat{G}(n)}\, \widehat{\scriptl}(n).
\end{equation}

For any function $F\in L^2(S^1)$, 
\begin{equation}
4\scriptq(F,F) + 2\scriptq(F,\tilde F) = 4\pi^2 \sum_n (4+2(-1)^n)\widehat{\scriptl}(n) |\widehat{F}(n)|^2.
\end{equation}

Indeed, specializing to $G=F$ in \eqref{Qparseval} gives
$\scriptq(F,F) =  4\pi^2\sum_{n\in\integers} |\widehat{F}(n)|^2\, \widehat{\scriptl}(n)$.
The reflection $\tilde F(\theta) = F(-\theta)$ 
of any real-valued function $F\in L^2(S^1)$
has Fourier coefficients 
$\widehat{\tilde F}(n)=e^{-in\pi}\widehat{F}(n)=(-1)^n \widehat{F}(n)$ for all $n$.
Setting $G=\tilde F$ instead gives
$\scriptq(F,\tilde F)
= 4\pi^2 \sum_n (-1)^n |{\widehat{F}}(n)|^2 \widehat{\scriptl}(n)$.

The next two lemmas will be proved in \S\ref{section:calculations}.
\begin{lemma}\label{lemma:scriptlfouriercoeffs}
For  any nonzero $n\in\integers$,
\begin{equation}
\widehat{\scriptl}(n)
= \begin{cases} 2n^{-2}\pi^{-1} \ \qquad &\text{if $n$ is odd}
\\ 2(n^2-1)^{-1}\pi^{-1} \ &\text{if $n$ is even.}
\end{cases}
\end{equation}
\end{lemma}

\begin{lemma} \label{lemma:gamma} $\gamma_{4,2}=4$. \end{lemma}

By Lemma~\ref{lemma:scriptlfouriercoeffs}, the numbers
$(4+2(-1)^n)\widehat{\scriptl}(n)$  are real and positive, and satisfy
\begin{equation}
(4+2(-1)^n)\widehat{\scriptl}(n) \ 
\begin{cases}
= 4\pi^{-1} \qquad &\text{for } n\in\{\pm1,\pm2\}
\\ \le \tfrac{4}{5} \pi^{-1} < 4\pi^{-1} &\text{for } |n|\ge 3.
\end{cases}
\end{equation}


\begin{lemma}
Let $d=2$. Suppose that $E\subset\reals^2$ is well approximated by $\bb$
in the sense \eqref{eq:nearness}, and that $E$ is balanced.  Then
\begin{equation}
4\langle K,f\rangle + 4 \langle L*f,f\rangle + 2\langle L*f,\tilde f\rangle
\le -\tfrac{8}{5} |E\symdif\bb|^2 +O(|E\symdif\bb|^{2+\varrho}).
\end{equation}
\end{lemma}

\begin{proof}
The hypothesis that $E$ is balanced is equivalent to
$\widehat{F}(n)=0$ for all $n\in\{0,\,\pm 1,\, \pm 2\}$. Therefore using the relation $F=a-b$,
\begin{align*}
4 \langle K,f\rangle & + 4\langle L*f,f\rangle + 2\langle L*f,\tilde f\rangle
\\&=
4\cdot (-\tfrac12\gamma_{4,2}) \int_{S^1} (a^2+b^2)(\alpha)
\,d\sigma(\alpha)
+  4\pi^2 \sum_{n\in \integers} |\widehat{F}(n)|^2 (4+2(-1)^n) \widehat{\scriptl}(n)
+O(\defb)^{2+\varrho}
\\&=
-8 \int_{S^1} (a^2+b^2)(\alpha) \,d\sigma(\alpha)
+  4\pi^2 \sum_{|n|\ge 3} |\widehat{F}(n)|^2 (4+2(-1)^n) \widehat{\scriptl}(n)
+O(\defb)^{2+\varrho}
\\ & \le
-8 \int_{S^1} (a^2+b^2)(\alpha) \,d\sigma(\alpha)
+  4\pi^2 \sup_{|n|\ge 3} \big((4+2(-1)^n) \widehat{\scriptl}(n)\big)
\cdot \sum_{|n|\ge 3} |\widehat{F}(n)|^2 
+O(\defb)^{2+\varrho}
\\& =  
-8 \int_{S^1} (a^2+b^2)(\alpha) \,d\sigma(\alpha)
+  4\pi^2\cdot\tfrac{4}{5}\pi^{-1} \sum_{n\in\integers} |\widehat{F}(n)|^2 
+O(\defb)^{2+\varrho}
\\&=  
-8 \int_{S^1} (a^2+b^2)(\alpha) \,d\sigma(\alpha)
+ \tfrac85 \int_{S^1} (a-b)^2\,d\sigma(\alpha)
+O(\defb)^{2+\varrho}
\\& \le   
-8 \int_{S^1} (a^2+b^2)(\alpha) \,d\sigma(\alpha)
+ \tfrac85 \int_{S^1} (a^2+b^2)\,d\sigma(\alpha)
+O(\defb)^{2+\varrho}
\\& =   
-\tfrac{32}5 \int_{S^1} (a^2+b^2)(\alpha) \,d\sigma(\alpha)
+O(\defb)^{2+\varrho}
\\& =   -\tfrac{8}{5} \pi^{-1} |E\symdif\bb|^2 +O(|E\symdif\bb|^{2+\varrho}
\end{align*}
for a certain constant $\varrho>0$.
We have used the trivial inequality $(a-b)^2= a^2+b^2-2ab \le a^2+b^2$, which holds
since $a,b$ are by their definitions nonnegative real-valued functions.
The final line is obtained from \eqref{eq:CS}:
\[\int_{S^1} (a^2+b^2)\,d\sigma \ge \tfrac12 \sigma(S^1)^{-1}|E\symdif\bb|^2 = (4\pi)^{-1} |E\symdif\bb|^2.\]
\end{proof}

One may observe that if $|E|=|\bb|$ but $E$ is not assumed to be balanced, then
$\sup_{|n|\ge 3} (4+2(-1)^n) \widehat{\scriptl}(n)$ must be replaced in this calculation by
$\sup_{n\ne 0} (4+2(-1)^n) \widehat{\scriptl}(n)$. 
With this change, the terms involving $\int_{S^1} (a^2+b^2)\,d\sigma$ in the above calculation cancel 
exactly, and consequently no useful conclusion is reached.

\subsection{Two calculations}\label{section:calculations}

\begin{proof}[Proof of Lemma~\ref{lemma:gamma}]
For $d=2$,
\begin{align*} 
(\onebb*\onebb)(x) = 4\int_{|x|/2}^1(1-s^2)^{1/2}\,ds
\end{align*}
so
\begin{align*} 
\frac{\partial}{\partial x_1}(\onebb*\onebb)(x_1,x_2)
= -4 (1-|x|^2/4)^{1/2} \cdot \tfrac12 \frac{x_1}{|x|}
= -2x_1|x|^{-1} (1-|x|^2/4)^{1/2} 
\end{align*}
for $|x|<2$, and this function vanishes for $|x|>2$.
Therefore
\begin{align*} 
\frac{\partial}{\partial x_1}(\onebb*\onebb*\onebb)(1,0)
&=\Big(\frac{\partial}{\partial x_1}(\onebb*\onebb)*\onebb\Big) (1,0)
\\&= -2\int_{|y-(1,0)|\le 1} y_1|y|^{-1}(1-|y|^2/4)^{1/2}\,dy.
\end{align*}
In polar coordinates $y=(r\cos(\theta),r\sin(\theta))$,
the domain of integration is
\[ \{y: (y_1-1)^2+y_2^2\le 1\} = \{(r,\theta): 0\le r\le 2\cos(\theta)
\text{ and } |\theta|\le \pi/2\}\]
and the integral becomes
\begin{align*}
-2 \int_{-\pi/2}^{\pi/2} \int_0^{2\cos(\theta)} & \cos(\theta)(1-\tfrac14 r^2)^{1/2}\,r\,dr\,d\theta
 =
2 \int_{-\pi/2}^{\pi/2}  \cos(\theta)\cdot \tfrac43(1-\tfrac14 r^2)^{3/2}\big|_0^{2\cos(\theta)} \,d\theta
\\& =
\tfrac83  \int_{-\pi/2}^{\pi/2}  \cos(\theta)
\big((1-\cos^2(\theta))^{3/2}-1\big) \,d\theta
 =
\tfrac83  \int_{-\pi/2}^{\pi/2}  \cos(\theta)
\big(|\sin(\theta)|^3-1\big) \,d\theta
\\& =
\tfrac{16}3  \int_{0}^{\pi/2}  \cos(\theta)
\big(\sin^3(\theta)-1\big) \,d\theta
 = \tfrac{16}3  
\big( \tfrac14\sin^4(\theta) -\sin(\theta)\big)\big|_0^{\pi/2}
\\& = \tfrac{16}3  \cdot(\tfrac14-1)
 = -4.
\end{align*}
\end{proof}

\begin{proof}[Proof of Lemma~\ref{lemma:scriptlfouriercoeffs}]
The Fourier coefficients of $\scriptl$ are
\begin{multline}
\widehat{\scriptl}(n) = 
(2\pi)^{-1} 
\int_{0}^{2\pi} L^*\big(\sqrt{2-2\cos(\theta)} \big) )\,e^{-in\theta}\,d\theta
\\
= (2\pi)^{-1} 
\int_{0}^{2\pi} L^*\big(\sqrt{2-2\cos(\theta)} \big) ) \cos(n\theta)\,d\theta
\end{multline}
since the $2\pi$--periodic function $\theta\mapsto L^*(\sqrt{2-2\cos(\theta)})$ is even. 
Assume that $n\ne 0$.  Integrate by parts to obtain
\begin{align*}
2\pi\widehat{\scriptl}(n)
&=-n^{-1} \int_{-\pi}^{\pi} \sin(n\theta)\, \frac{d}{d\theta} 
\big(L^*(\sqrt{2-2\cos(\theta)}\big)\,d\theta
\\&= - n^{-1} \int_{-\pi}^{\pi} \sin(n\theta) \, \big(\frac{d}{d\theta} L^*\big)(\sqrt{2-2\cos(\theta)})
\cdot \sin(\theta) \cdot (2-2\cos(\theta))^{-1/2} \,d\theta
\\&= - n^{-1} \int_{-\pi}^{\pi} \sin(n\theta)\, 
\big( -\omega_1 \big(1-\tfrac14 (2-2\cos(\theta))^{2/2} \big)^{1/2} \big)
\cdot \sin(\theta) \cdot (2-2\cos(\theta))^{-1/2} \,d\theta
\\&=  n^{-1} \int_{-\pi}^{\pi} \sin(n\theta)\, 
(1+ \cos(\theta))^{1/2} \cdot \sin(\theta) \cdot (1-\cos(\theta))^{-1/2} \,d\theta
\\&=  n^{-1} \int_{-\pi}^{\pi} \sin(n\theta)\, 
(1+ \cos(\theta)) \cdot \sin(\theta) \cdot (1-\cos^2(\theta))^{-1/2} \,d\theta
\\&= 2n^{-1} \int_0^{\pi} \sin(n\theta)\, (1+ \cos(\theta))
\cdot \sin(\theta) \cdot (1-\cos^2(\theta))^{-1/2} \,d\theta
\\&= 2n^{-1} \int_0^{\pi} \sin(n\theta)\, 
(1+ \cos(\theta)) \,d\theta.
\end{align*}

The integral is an odd function of $n$ and splits naturally as a sum of two terms.
For the first of these,
\begin{equation} \int_0^\pi \sin(n\theta)\,d\theta
= \begin{cases} 0\ &\text{if $n$ is even}
\\ 2n^{-1} &\text{if $n$ is odd}. \end{cases} \end{equation}
The second, for $n>0$, is
\begin{align*}
\int_0^\pi \sin(n\theta)\cos(\theta)\,d\theta
&= (4i)^{-1} \int_0^\pi \big(e^{in\theta}-e^{-in\theta}\big)\big( e^{i\theta}+e^{-i\theta}\big)\,d\theta
\\&
= \begin{cases}
0\ \qquad &\text{if $n$ is odd}
\\ 2n(n^2-1)^{-1}\ &\text{if $n$ is even.}
\end{cases}
\end{align*}
\end{proof}

\section{On dimension $d=1$} \label{section:1D}

Now $S^{d-1}=S^0=\{\pm 1\}$, and $\sigma(S^0)=2$.
If $E$ is balanced then
$F=b-a$ on $S^0$ is orthogonal in $L^2(S^0)$ to $x$ and to constant functions. 
Thus if $E$ is balanced then
\begin{equation} F\equiv 0.\end{equation}
Equivalently, $b\equiv a$. 

The exponent
$q_d= 4-\frac2{d+1}$ equals $3$ for $d=1$, so the range $(q_d,\infty)$ in which
results for general dimensions were established specializes to $(3,\infty)$. 

Together with results shown for general dimensions above, for balanced sets $E\subset\reals^1$ this
establishes the following simplified comparison of $\norm{\widehat{\one_E}}_q^q$
to $\norm{\widehat{\one_\bb}}_q^q$, in which no quadratic terms in $f=\one_E-\one_\bb$ appear.
\begin{lemma} \label{lemma:noquadraticterms}
Let $d=1$. For each $q\in(3,\infty)$ there exists $\varrho>0$
such that for any bounded, balanced set  $E\subset\reals^1$ that satisfies $|E|=|\bb|$,
\begin{equation} \norm{\widehat{\one_E}}_q^q = \norm{\widehat{\one_\bb}}_q^q
+q\langle K_{q},\one_E-\one_{\bb} \rangle +O(|E\symdif\bb|^{2+\varrho}).  \end{equation}
\end{lemma}

Another simplification is the availability of a simple expression 
for $\widehat{\one_\bb}$.  Since
\begin{equation}
\widehat{\one_\bb}(\xi) = \int_{-1}^1 e^{-2\pi ix\xi}\,dx
= (-2\pi i)^{-1}\xi^{-1} \big( e^{-2\pi i\xi}-e^{2\pi i \xi}\big) = \frac{\sin(2\pi \xi)}{\pi\xi},
\end{equation}
one has
\begin{align*}
K_{q}(x) 
& = \pi^{-(q-1)} \int_{-\infty}^\infty
e^{2\pi ix\xi} \frac{\sin(2\pi\xi)}{\xi} \big| \xi^{-1}\sin(2\pi\xi)\big|^{q-2}\,d\xi
\\& = \pi^{-(q-1)} \int_{-\infty}^\infty
\cos(2\pi x\xi) \xi^{-1} \sin(2\pi\xi) \big| \xi^{-1}\sin(2\pi\xi)\big|^{q-2}\,d\xi
\end{align*}
and
\begin{align*} \gamma_{1,q} = - \frac{d K_{q}}{dx}\Big|_{x=1} =  2\pi^{2-q} \int_{-\infty}^\infty
|\xi|^{2-q} \, |\sin(2\pi\xi)|^q  \,d\xi >0 \end{align*}
for all $q\in(2,\infty)$.
Moreover, if $q\ge 4$ is an even integer
then $K_{q}$, the convolution product of $q-1$ factors of $\one_\bb$,
is an even function that is strictly decreasing on the interval $[0,q-1]$,
and vanishes identically for $|x|>q-1$.  In particular, 
\eqref{eq:Knecessarystronger} holds
for any even integer $q\ge 4$.

Since $F\equiv 0$,
$\langle L_{q}*f,f\rangle = O(\defb)^{2+\varrho}$
and likewise
$\langle L_{q}*f,\tilde f\rangle = O(\defb)^{2+\varrho}$
where $\varrho>0$.
Therefore for any $q\in\{4,6,8,\dots\}$,
\begin{align*}
\norm{\one_E}_q^q
&= \norm{\one_\bb}_q^q
-\tfrac12 q \gamma_{1,q} \int_{S^0} (a^2+b^2)(\alpha)\,d\sigma(\alpha)
+ O(\defb)^{2+\varrho}
\\ &\le \norm{\one_\bb}_q^q -\tfrac18 q \gamma_{1,q}  \defb^2 + O(\defb)^{2+\varrho}.
\end{align*}
Theorem~\ref{thm:1D} follows from Lemma~\ref{lemma:noquadraticterms} 
by results and reasoning shown in earlier sections.  \qed

Proposition~\ref{prop:locald1q>3}
is a consequence of Lemma~\ref{lemma:noquadraticterms} and these formulas. \qed

\section{Precompactness of extremizing sequences}\label{section:cpctproof}

We turn to the proof of the compactness principle, Theorem~\ref{thm:compactness}.  
The method used is that of \cite{christHY}.

\subsection{Notation}
By a measurable partition of a Lebesgue measurable
set $E$ we mean a pair of Lebesgue measurable subsets $A,B\subset E$
satisfying $A\cup B=E$ and $A\cap B=\emptyset$ up to Lebesgue null sets.
We say that two Lebesgue measurable functions $g,h$ are disjointly supported if their
product vanishes almost everywhere.

$\norm{y}_{\reals/\integers}$  will denote the distance from $y\in\reals$ to $\integers$.
For $d>1$, if $x=(x_1,\dots,x_d)$,
$\norm{x}_{\reals^d/\integers^d} = \max_{1\le j\le d} \norm{x_j}_{\reals/\integers}$.
There is a triangle inequality $\norm{x+y}_{\reals^d/\integers^d} 
\le \norm{x}_{\reals^d/\integers^d} +  \norm{y}_{\reals^d/\integers^d}$.

$O_\delta(1)$ denotes a quantity whose absolute value is less than or equal to a finite constant
that depends only on $\delta$ and on the dimension $d$, or sometimes on 
the exponent $q$ under discussion, as well. Similarly, $o_\delta(1)$ denotes a quantity
that tends to zero as $\delta\to 0$, at a rate that may depend on $d$ or on both $d$ and $q$. 
All bounds are implicitly asserted to be uniform for $q$ in any compact subset of $(2,\infty)$.

$\Gl(d)$ denotes the group of all invertible linear transformations $T:\reals^d\to\reals^d$.  $\aff(d)$ 
denotes the group of all affine automorphisms of $\reals^d$, that is, all maps $x\mapsto \scriptt(x)=T(x)+v$
where $v\in\reals^d$ and $T\in \Gl(d)$.
$\det(T)$ denotes the determinant of $T$, and $|J(\scriptt)|$ denotes the Jacobian determinant of $\scriptt$,
which is equal to $|\det(T)|$. 

$|S|$ denotes the Lebesgue measure of a subset $S$ of a Euclidean space of arbitrary dimension. 
Most frequently $S$ will be a subset of $\reals^d$,
but other spaces, such as $\reals^{d^2}$ and $\reals^{d(d-1)}$, will also arise.
$\#(S)$ denotes the cardinality of a finite set $S$, and equals $\infty$ if $S$ is an infinite set.

$\unitQ^d$ denotes the unit cube in $\reals^d$, the set of all $x=(x_1,\dots,x_d)$ satisfying 
$0\le x_j\le 1$ for every $1\le j\le d$.  $\torus$ is the quotient group $\torus = \reals/\integers$.

\begin{definition}
A discrete multiprogression ${\mathbf P}$ in $\reals^d$ of rank $r$ is a function
\[{\mathbf P}: \prod_{i=1}^r\set{0,1,\dots,N_i-1} \to \reals^d\] of the form
\[ {\mathbf P}(n_1,\dots,n_r)=\big\{a + \sum_{i=1}^r n_i v_i: 0\le n_i<N_i\big\}, \]
for some $a\in\reals^d$, some $v_j\in\reals^d$, 
and some positive integers $N_1,\dots,N_r$. The {\it size} of ${\mathbf P}$ is 
$\sigma({\mathbf P})=\prod_{i=1}^r N_i$.
The multiprogression ${\mathbf P}$ is said to be {\it proper} if this mapping is injective.

A continuum multiprogression $P$ in $\reals^d$ of rank $r$ is a function 
\[P: \prod_{i=1}^r \set{0,1,\dots,N_i-1}\times\unitQ^d \to \reals^d\] of the form
\[(n_1,\dots,n_d;y)\mapsto a+\sum_i n_i v_i + sy\] where $a,v_i\in\reals^d$ and $s\in\reals^+$. 
The {\it size} of $P$ is defined to be \[\sigma(P)=s^d \prod_i N_i.\]
$P$ is said to be {\it proper} if this mapping is injective.
\end{definition}

We will loosely identify a multiprogression with its range, 
and will thus refer to multiprogressions as if they were sets rather than functions. 
If $P$ is proper then the Lebesgue measure of its range equals its size.
A more invariant definition would replace $\unitQ^d$ by an arbitrary {\it convex} set 
of positive and finite Lebesgue measure.
The above definition is equivalent for our purposes, according to a theorem of John, and is more convenient.


\subsection{Generalities concerning norm inequalities}

\begin{lemma}
For each $d\ge 1$ and $q\in(2,\infty)$
there exist $\delta_0,c,C'\in\reals^+$ with the following property.
Let $E\subset\reals^d$ be a Lebesgue measurable set satisfying $0<|E|<\infty$.
Let $\delta\in(0,\delta_0]$.
Suppose that $\norm{\widehat{\one_E}}_q	\ge (1-\delta) \bestAqd |E|^{1/q'}$.
For any measurable subset $A\subset E$ satisfying $|A|\ge C' \delta |E|$, 
\begin{equation} \norm{\widehat{\one_A}}_q \ge c\delta^{1/q}\bestAqd |A|^{1/q'}. \end{equation}
\end{lemma}

\begin{proof}
Let $p=q'$.
We may assume without loss of generality that $|E|=1$, by the scaling symmetry of the inequality.
Set $B=E\setminus A$.

There exists $C<\infty$ such that for arbitrary functions $G,H\in L^q$,
\begin{equation}
\norm{G+H}_q^q \le \norm{G}_q^q + C \norm{G}_q^{q-1}\norm{H}_q + C\norm{H}_q^q.
\end{equation}
Consequently
\begin{align*}
\norm{\widehat{\one_A} + \widehat{\one_B}}_q^q 
&\le \norm{\widehat{\one_B}}_q^q + C \norm{\widehat{\one_B}}_q^{q-1}\norm{\widehat{\one_A}}_q + C\norm{\widehat{\one_A}}_q^q
\\ & \le \bestAqd^q |B|^{q/p} + C\bestAqd^{q-1}|B|^{(q-1)/p}\norm{\widehat{\one_A}}_q + C\norm{\widehat{\one_A}}_q^q.
\end{align*}

Set $y = |A|^{1/p}\in(0,1]$ and $x = \frac{\norm{\widehat{\one_A}}_q}{\bestAqd \cdot|A|^{1/p}}\in[0,1]$. 
Then $|B|^{1/p} = (1-y^p)^{1/p}\le 1-c_py^p$  for a certain constant $c_p>0$.  

Since $|E|=1$, the hypothesis $\norm{\widehat{\one_E}}_q \ge (1-\delta) \bestAqd |E|^{1/p}$
can be rewritten as 
$\norm{\widehat{\one_A}+\widehat{\one_B}}_q \ge (1-\delta)\bestAqd$.
Together with the upper bound derived above, this gives
\begin{align*} (1-\delta)^q &\le  (1-y^p)^{q/p} + C(1-y^p)^{(q-1)/p} xy + Cx^qy^q
\\ & \le 1-c_py^p + Cxy \end{align*}
since $(1-y^p)\le 1$ and $x^qy^q\le xy$.  Therefore
$ x \ge cy^{p-1}-C\delta y^{-1} \ge c\delta^{(p-1)/p} $ 
provided that $|A|\ge C'\delta$ for a sufficiently large constant $C'$.
Inserting the definition of $x$, this inequality becomes
\[ \norm{\widehat{\one_A}}_q \ge c\delta^{(p-1)/p}\bestAqd |A|^{1/p}
= c\delta^{1/q}\bestAqd |A|^{1/q'}.  \]
\end{proof}

We also require the corresponding statement concerning the dual inequality. 
This is the inequality
\begin{equation}
\norm{\widehat{f}}_{L^{q,\infty}} \le \bestAqd\norm{f}_{L^{q'}}
\end{equation}
where $\bestAqd$ is the same constant as above.
Recall that we use
for the Lorentz space $L^{q,\infty}$ the norm 
\begin{equation*}
\norm{g}_{L^{q,\infty}} = \norm{g}_{q,\infty} = \sup_E \big|\int_E g\big|\cdot |E|^{-1/q'},
\end{equation*}
with the supremum taken over all Lebesgue measurable sets satisfying $0<|E|<\infty$.
Since $L^q$ is embedded continuously in $L^{q,\infty}$,
the Fourier transform is bounded from $L^{q'}$ to $L^{q,\infty}$ for each $q\in(2,\infty)$.
The norm that we have chosen for $L^{q,\infty}$ is the one naturally associated with 
the duality between $L^{q,\infty}$ and the Lorentz space $L^{q,1}$ \cite{steinweiss},
so that the optimal constant in the inequality 
$\norm{\widehat{f}}_{L^{q,\infty}(\reals^d)}\le C\norm{f}_{L^{q'}(\reals^d)}$
is $C=\bestAqd$.

\begin{lemma}
For each $d\ge 1$ and $q\in(2,\infty)$
there exist $\delta_0,c,C'\in\reals^+$ with the following property.
Let $f=g+h$ where $f,g,h\in L^{q'}(\reals^d)$ and $g,h$ are disjointly supported.
Let $\delta\in(0,\delta_0]$.
Suppose that $\norm{\widehat{f}}_{q,\infty}	\ge (1-\delta) \bestAqd \norm{f}_{q'}$.
If $\norm{h}_{q'} \ge C'\delta\norm{f}_{q'}$ then
\begin{equation}
\norm{\widehat{h}}_{q,\infty} \ge c\delta^{q'/q}\bestAqd \norm{h}_{q'}.
\end{equation}
\end{lemma}

\begin{proof}
We may assume without loss of generality that $\norm{f}_{q'}=1$.
Set $p=q'$.
Let $A,B$ be disjoint measurable sets such that $g=\one_A g$ and $h=\one_B h$ almost everywhere.
By definition of the $L^{q,\infty}$ norm,
there exists a Lebesgue measurable set $E\subset\reals^d$ with $0<|E|<\infty$ satisfying 
$|\int_E \widehat{f}| \ge (1-2\delta)\bestAqd |E|^{1/p}$.
Then
\begin{align*}
|\int_E \widehat{h}| 
&\ge 
|\int_E \widehat{f}| -  |\int_E \widehat{g}| 
\\& \ge 
(1-2\delta)\bestAqd |E|^{1/p} -  \bestAqd |E|^{1/p}\norm{g}_{p}
\\& = \bestAqd |E|^{1/p}
\big( (1-2\delta)-(1-\norm{h}_{p}^{p})^{1/p}\big)
\\& = \bestAqd |E|^{1/p}
\big( p^{-1}\norm{h}_{p}^{p} -2\delta + O(\norm{h}_{p}^{2p}) \big).
\end{align*}
If $\norm{h}_{p}^{p} \ge 4p\delta$, and if $\delta$ is sufficiently small, we conclude that
\begin{equation}
|\int_E \widehat{h}| 
\ge c \bestAqd |E|^{1/p}\norm{h}_{p}^{p}
\ge c \bestAqd |E|^{1/p}\delta^{p/q}\norm{h}_{p}.
\end{equation}
Since $\norm{\widehat{h}}_{q,\infty} \ge |E|^{-1/p}|\int_E \widehat{h}|$ by definition of this norm,
this completes the proof.
\end{proof}

\begin{lemma}
For each $d\ge 1$ and $q\in(2,\infty)$ there exist $\delta_0,C'\in\reals^+$ with the following property.
Let $E\subset\reals^d$ be a Lebesgue measurable set satisfying $0<|E|<\infty$.
Let $\delta\in(0,\delta_0]$.
Suppose that $\norm{\widehat{\one_E}}_q	\ge (1-\delta) \bestAqd |E|^{1/q'}$.
Suppose that $E = A\cup B$ where $A\cap B=\emptyset$, and $A,B$ are Lebesgue measurable.
If $\min(|A|,|B|)\ge C' \delta |E|$ then
\begin{equation}
\norm{\widehat{\one_A}\widehat{\one_B}}_{q/2} \ge \delta \bestAqd^2 |E|^{2/q'}.
\end{equation}
\end{lemma}

\begin{proof}
Set $p=q'$.
\begin{align*}
\norm{\widehat{\one_E}}_q^2
& = \norm{\widehat{\one_A}+\widehat{\one_B}}_q^2
\\& = \norm{\big|\widehat{\one_A}+\widehat{\one_B}\big|^2}_{q/2}
\\& \le \norm{\widehat{\one_A}}_q^2
+ \norm{\widehat{\one_B}}_q^2
+2\norm{\widehat{\one_A}\widehat{\one_B}}_{q/2}
\\& \le
\bestAqd^2|A|^{2/p}
+ \bestAqd^2|B|^{2/p}
+2\norm{\widehat{\one_A}\widehat{\one_B}}_{q/2}.
\end{align*}
Now setting $\rho = 2p^{-1}-1>0$, 
\begin{align*}
|A|^{2/p}+|B|^{2/p} 
&\le \max(|A|,|B|)^\rho (|A|+|B|)
\\& = \max(|A|,|B|)^\rho |E|
\\& = \max(|A|/|E|,|B|/|E|)^\rho |E|^{2/p}
\\& \le  (1-C'\delta)^\rho |E|^{2/p}
\\& \le  (1-C'c\delta) |E|^{2/p}.
\end{align*}
Therefore
\begin{align*}
\norm{\widehat{\one_A}\widehat{\one_B}}_{q/2}
\ge \tfrac12 \bestAqd^2\Big( (1-\delta)^2  - (1-C'c\delta) \Big) |E|^{2/p}.
\end{align*}
If $C'$ is sufficiently large then
$(1-\delta)^2  - (1-C'c\delta) \ge 2\delta$ for all sufficiently small $\delta>0$.
\end{proof}

\begin{lemma}
For each $d\ge 1$ and $q\in(2,\infty)$
there exist $\delta_0,c,C_0\in\reals^+$ with the following property.
Let $f=g+h$ where $f,g,h\in L^{q'}(\reals^d)$ and $g,h$ are disjointly supported.
Let $\delta\in(0,\delta_0]$.
Suppose that $\norm{\widehat{f}}_{q,\infty} \ge (1-\delta) \bestAqd \norm{f}_{q'}$.
If $\min(\norm{g}_{q'},\norm{h}_{q'}) \ge C_0\delta^{1/q'}\norm{f}_{q'}$ then
\begin{equation}
\norm{\,|\widehat{g}\,\,\widehat{h}|^{1/2}}_{q,\infty} \ge c\delta^{3/2}\bestAqd \norm{f}_{q'}.
\end{equation}
\end{lemma}

\begin{proof}
Choose some Lebesgue measurable set $E\subset\reals^d$ with $0<|E|<\infty$ satisfying 
$|\int_E \widehat{f}| \ge (1-2\delta)\bestAqd \norm{f}_{q'} |E|^{1/q'}$.
Define 
\begin{gather*}
E_g=\{\xi\in E: |\widehat{g}(\xi)|\ge \delta^{-1}|\widehat{h}(\xi)|\},
\\ E_h=\{\xi\in E: |\widehat{h}(\xi)|\ge \delta^{-1}|\widehat{g}(\xi)|\},
\\ E_0=E\setminus (E_g\cup E_h).
\end{gather*}
Then
\begin{align*}
|\int_E \widehat{f}| 
&\le (1+\delta)\bestAqd|\int_{E_g} \widehat{g}|
+ (1+\delta)\bestAqd|\int_{E_h} \widehat{h}|
+ \delta^{-1/2} \int_{E_0} |\widehat{g}\widehat{h}|^{1/2}
\\& \le
(1+\delta) \bestAqd|E_g|^{1/q'}\norm{g}_{q'}
+ (1+\delta) \bestAqd |E_h|^{1/q'}\norm{h}_{q'}
+ \delta^{-1/2} \int_{E} |\widehat{g}\widehat{h}|^{1/2}
\end{align*}
so that
\begin{align*}
\int_{E} |\widehat{g}\widehat{h}|^{1/2}
\ge 
\delta^{1/2}
\bestAqd\Big((1-2\delta)\norm{f}_{q'} |E|^{1/q'}
- (1+\delta) |E_g|^{1/q'}\norm{g}_{q'}
- (1+\delta) |E_h|^{1/q'}\norm{h}_{q'}\Big).
\end{align*}
By H\"older's inequality,
\begin{align*}
|E_g|^{1/q'}\norm{g}_{q'} + |E_h|^{1/q'}\norm{h}_{q'}
&\le 
(|E_g|+|E_h|)^{1/q'} \big( \norm{g}_{q'}^q+\norm{h}_{q'}^q\big)^{1/q}
\\ & \le 
|E|^{1/q'}\norm{f}_{q'}^{q'/q} \max\big(\norm{g}_{q'},\norm{h}_{q'}\big)^{(q-q')/q}
\\ & \le 
|E|^{1/q'}\norm{f}_{q} (1-C'\delta)
\end{align*}
where $C'= cC_0^{q'}$,
since
\begin{align*}
\max\big(\norm{g}_{q'},\norm{h}_{q'}\big)^{q'}
= \norm{f}_{q'}^{q'} - \min(\norm{g}_{q'}^{q'},\norm{h}_{q'}^{q'})
\le \norm{f}_{q'}^{q'} - C_0^{q'} \delta \norm{f}_{q'}^{q'}
\end{align*}
by the hypothesis that
$\min(\norm{g}_{q'}^{q'},\norm{h}_{q'}^{q'}) \ge C_0^{q'}\delta\norm{f}_{q'}^{q'}$.
Therefore
\begin{align*}
\int_{E} |\widehat{g}\widehat{h}|^{1/2}
& \ge 
\delta^{1/2}
\bestAqd
\norm{f}_{q'} |E|^{1/q'}
\Big( (1-2\delta) - (1+\delta)(1-C'\delta) \Big)
\\& \ge 
\delta^{1/2}
\bestAqd
\norm{f}_{q'} |E|^{1/q'}
(C'\delta-2\delta-\delta)
\\& \ge  \delta^{3/2}
\bestAqd \norm{f}_{q'} |E|^{1/q'}
\end{align*}
provided that $C_0$ is sufficiently large to ensure that $C'\ge 4$.
\end{proof}

\subsection{Compatibility of multiprogressions}

The next result is Lemma~5.5 of \cite{christHY}, to which we refer for a proof.
\begin{lemma}
\label{lemma:compatibleprogressions}
Let $d\ge 1$ and let $\Lambda\subset(2,\infty)$ be compact.
Let $\lambda>0$  and $R<\infty$. 
There exists $C<\infty$, depending only on $R,d,\Lambda,\lambda$, with the following property
for all $q\in\Lambda$.
Let $P,Q\subset\reals^d$ be nonempty proper continuum multiprogressions of ranks $\le R$.
Let $\varphi,\psi$ be functions supported on $P,Q$ respectively that satisfy
$\norm{\varphi}_\infty|P|^{1/q'}\le 1$ and $\norm{\psi}_\infty|Q|^{1/q'}\le 1$.  If
\begin{equation} \norm{\,|\widehat{\varphi}\,\widehat{\psi}|^{1/2}}_{q} \ge \lambda \end{equation}
then \begin{equation} |P+Q| \le C\min(|P|,|Q|).\end{equation}
\end{lemma}

If $\norm{\,|\widehat{\varphi}\,\widehat{\psi}|^{1/2}}_{q,\infty} \ge \lambda$
then the hypothesis of Lemma~\ref{lemma:compatibleprogressions} is satisfied,
since a constant multiple of the $L^q$ norm majorizes the $L^{q,\infty}$ norm.
Both situations arise below in the analysis.

\subsection{Quasi-extremizers}

\begin{lemma} \label{lemma:quasi1}
Let $d\ge 1$, let $\Lambda\subset(2,\infty)$ be compact, and let $\eta\in(0,1]$.
There exist $C_\eta,c_\eta\in\reals^+$ with the following property
for all $q\in\Lambda$.
Suppose that $E\subset\reals^d$ is a Lebesgue measurable set with $|E|\in\reals^+$
satisfying $\norm{\widehat{\one_E}}_{q} \ge \eta |E|^{1/q'}$.
Then there exists a proper continuum multiprogression $P$ satisfying
\begin{gather*}
|P\cap E|\ge c_\eta|E|,
\\ |P|\le C_\eta|E|,
\\ \rank(P) \le C_\eta.
\end{gather*}
\end{lemma}

\begin{lemma} \label{lemma:quasi2}
Let $d\ge 1$, let $\Lambda\subset(2,\infty)$ be compact, and let $\eta\in(0,1]$.
There exist $C_\eta,c_\eta\in\reals^+$ with the following property for all $q\in\Lambda$.
Suppose that $0\ne f\in L^{q'}(\reals^d)$ satisfies $\norm{\widehat{f}}_{q,\infty} \ge \eta \norm{f}_{q'}$.
Then there exist a proper continuum multiprogression $P$ and a 
disjointly supported Lebesgue measurable decomposition $f=g+h$ such that 
\begin{gather*}
\text{$g$ is supported on $P$,}
\\ \norm{g}_{q'}\ge c_\eta \norm{f}_{q'},
\\ \norm{g}_\infty |P|^{1/q'}\le C_\eta \norm{f}_{q'},
\\ \rank(P) \le C_\eta.
\end{gather*}
\end{lemma}

\begin{proof}[Proofs of Lemmas]
These two lemmas are simple consequences of Proposition~6.4 of \cite{christHY}, 
in which a function $f\in L^{q'}$ is assumed to satisfy $\norm{\widehat{f}}_q\ge\eta\norm{f}_{q'}$,
and the conclusions are the same as those of Lemma~\ref{lemma:quasi2}.
Lemma~\ref{lemma:quasi2} follows directly from that result; its hypotheses are stronger 
than required since $\norm{\widehat{f}}_{q,\infty} \le C_q\norm{\widehat{f}}_q$. 

Lemma~\ref{lemma:quasi1} is obtained by specializing Proposition~6.4 of \cite{christHY} 
to the case $f=\one_E$.
In the resulting disjointly supported decomposition $f=g+h$, 
the summands $g,h$ are necessarily indicator functions
of disjoint sets. Thus the conclusions of the Proposition are those of Lemma~\ref{lemma:quasi1}.
\end{proof}

\subsection{Structure of near-extremizers}

\begin{lemma}[Multiprogression structure of near-extremizers] \label{lemma:NE1}
Let $d\ge 1$, and let $\Lambda\subset(2,\infty)$ be compact. 
For any $\eps>0$ there exist $\delta>0$ and $C_\eps<\infty$ with the following property 
for all $q\in\Lambda$.
Let $E\subset\reals^d$ be a Lebesgue measurable set with $|E|\in\reals^+$ that satisfies
$\norm{\widehat{\one_E}}_{q} \ge (1-\delta)\bestAqd |E|^{1/q'}$.
There exist a measurable partition $E=A\cup B$
and a continuum multiprogression $P$ satisfying
\begin{gather*}
|B| \le \eps |A|
\\ A\subset P
\\ |P|\le C_\eps |E|
\\ \rank(P)\le C_\eps.
\end{gather*}
\end{lemma}

\begin{lemma}[Multiprogression structure of near-extremizers] \label{lemma:NE2}
Let $d\ge 1$, and let $\Lambda\subset(2,\infty)$ be compact. 
For any $\eps>0$ there exist $\delta>0$ and $C_\eps<\infty$ with the following property
for all $q\in\Lambda$.
Let $0\ne f\in L^{q'}(\reals^d)$ satisfy $\norm{\widehat{f}}_{q,\infty} \ge (1-\delta)\bestAqd\norm{f}_{q'}$.
There exist a disjointly supported Lebesgue measurable
decomposition $f=g+h$ and a continuum multiprogression $P$ satisfying
\begin{gather*}
\norm{h}_{q'} <\eps\norm{f}_{q'}
\\ \norm{g}_\infty |P|^{1/q'}\le C_\eps\norm{f}_{q'}
\\ g \text{ is supported on } P
\\ \rank(P)\le C_\eps.
\end{gather*}
\end{lemma}

\begin{proof}[Proofs]
These two lemmas are analogues of Lemma~7.3 of \cite{christHY}.
They are consequences of the lemmas of the preceding three subsections, 
by the same reasoning used in \cite{christHY} to deduce that Lemma~7.3 from corresponding ingredients.
Therefore the details are omitted. 
\end{proof}

\subsection{Discrete and hybrid groups}

In \S\ref{subsect:lifting} we will gain information about near-extremizers
by lifting sets and operators from $\reals^d$ to $\integers^d\times\reals^d$.
In the present subsection we establish basic facts about 
near-extremizers of the corresponding Fourier transform inequality 
in this product setting. 

We use the same notation $\widehat{\cdot}$ to denote the Fourier transform 
for each of the groups $G=\integers^\kappa\times\reals^d$,
mapping functions on $G$ to functions on the dual group $\widehat{G} = \torus^\kappa\times\reals^d$
by
\[ \widehat{f}(\theta,\xi) = \int_{\reals^d} \sum_{n\in\integers^\kappa} 
e^{-2\pi i x\cdot\xi} e^{-2\pi i n\cdot \theta} f(n,x)\,dx\]
where $\torus = \reals/\integers$. Integration on $\integers^\kappa\times\reals^d$
will always be with respect to the product of Lebesgue measure with counting
measure; likewise integration on $\torus^\kappa\times\reals^d$
is with respect to the natural Lebesgue measure.

\begin{lemma} \label{lemma:productA}
Let $d,\kappa\ge 1$, and $q\in(2,\infty)$. 
The optimal constant ${\mathbf A}(q,d,\kappa)$ in the inequality
\begin{equation} \norm{\widehat{\one_E}}_{q} \le {\mathbf A}(q,d,\kappa) |E|^{1/q'}\end{equation}
for $\integers^\kappa\times\reals^d$ satisfies
\begin{equation} {\mathbf A}(q,d,\kappa) = \bestAqd.\end{equation}

The optimal constant ${\mathbf A'}(q,d,\kappa)$ in the inequality
\begin{equation} \norm{\widehat{f}}_{q,\infty} \le {\mathbf A'}(q,d,\kappa) \norm{f}_{q'}\end{equation}
for $\integers^\kappa\times\reals^d$ is likewise equal to $\bestAqd$.
\end{lemma}

The proof requires some additional notation.
Let $q\in[2,\infty]$, and let $p\in[1,2]$ be the conjugate exponent.
The product Fourier transform can be expressed as a composition
$\scriptf\circ\tilde\scriptf$ of commuting operators, with 
\begin{align}
\scriptf g(\theta,\xi) &= \sum_{n\in\integers^\kappa} g(n,\xi) e^{-2\pi i n\cdot\theta}
\\ \tilde\scriptf f(n,\xi) &= \int_{\reals^d} f(n ,x) e^{-2\pi i x\cdot\xi}\,dx .\end{align}
Here $f,g:\integers^\kappa\times\reals^d\to\complex$. 
Thus $\scriptf$ is a partial Fourier transform in the first variable, while $\tilde\scriptf$
is a partial Fourier transform in the second.

There are corresponding partial Fourier transforms which act in the reverse order;
one maps a function of $(n,x)$ to a function of $(\theta,x)$ while another
maps a function of $(\theta,x)$ to a function of $(\theta,\xi)$.
We will always work in $\integers^\kappa\times\reals^d$
or $\torus^\kappa\times\reals^d$, with the coordinate for $\integers^\kappa$
or $\torus^\kappa$ regarded as the first coordinate and that for
$\reals^d$ as the second coordinate.
We will denote these corresponding partial Fourier transforms 
by $\scriptf$ (partial Fourier transform with respect to
the first variable) and $\tilde\scriptf$ (partial Fourier transform with
respect to the second variable) respectively,
despite the ambiguity in notation. These operators commute in the sense that
\begin{equation}\label{FTscommute} \scriptf\circ\tilde\scriptf = \tilde\scriptf\circ\scriptf,\end{equation}
although the operators on the left-hand side of this equation are not
exactly identical to the corresponding ones on the right with the same symbols. 
It will be clear from context which operators are meant in the discussion below
if one bears in mind that $\scriptf$ acts with respect to the first coordinate,
and $\tilde\scriptf$ with respect to the second.

The norms for
$L^p_n L^q_\xi(\integers_n^\kappa\times\reals_\xi^d)$ 
and
$L^q_\xi L^p_n(\integers_n^\kappa\times\reals_\xi^d)$ 
for a function $g:\integers^\kappa\times\reals^d\to\complex$ are defined respectively by 
\begin{align*}
\norm{g}_{L^p_n L^q_\xi}
&= \big(\sum_n \big(\int |g(n,\xi)|^q\,d\xi\big)^{p/q} \big)^{1/q}
\\
\norm{g}_{L^q_\xi L^p_n}
&= \big(\int_\xi \big(\sum_n |g(n,\xi)|^p\big)^{q/p} \,d\xi\big)^{1/p}.
\end{align*}
There are corresponding norms for 
$L^s_\theta L^t_x(\torus_\theta^\kappa\times\reals^d_x)$
and for
$L^t_x L^s_\theta (\torus_\theta^\kappa\times\reals^d_x)$.

Since $q\ge p$,
$L^p_x L^q_\theta(\torus_\theta^\kappa\times\reals_x^d)$ is contained in 
$L^q_\theta L^p_x(\torus_\theta^\kappa\times\reals_x^d)$. 
Moreover, this inclusion is a contraction;
\[ \norm{g}_{L^q_\theta L^p_x(\torus^\kappa_\theta \times\reals^d_x)}
\le \norm{g}_{L^p_x L^q_\theta(\reals^d_x\times \torus^\kappa_\theta)}\]
by Minkowski's integral inequality. 
Furthermore,
$\tilde\scriptf$ is a contraction from $L^q_\xi L^p_n(\integers_n^\kappa\times \reals_\xi^d)$
to $L^q_\xi L^q_\theta(\torus_\theta^\kappa\times \reals_\xi^d)$.

For $n\in\integers^\kappa$ define $E_n\subset\reals^d$ by
\begin{equation} E_n=\{x\in\reals^d: (n,x)\in E\}.  \end{equation}
For each $n\in\integers^\kappa$,
$\norm{\tilde\scriptf\one_{E_n}}_{L^q(\reals^d)} \le \bestAqd |E_n|^{1/p}$.
Therefore
\begin{equation*}
\norm{\tilde\scriptf\one_E}_{L^p_n L^q_\xi} \le \bestAqd
\big(\sum_{n\in\integers^\kappa} |E_n|^{p/p}\big)^{1/p} = |E|^{1/p}.
\end{equation*}

\begin{proof}[Proof of Lemma~\ref{lemma:productA}]
First consider the optimal constant ${\mathbf A}(q,d,\kappa)$. One has for any 
measurable $E\subset\reals^d$ with finite Lebesgue measure 
\begin{equation*}
\norm{\widehat{\one_E}}_{L^q(\torus^\kappa\times\reals^d)} 
= \norm{\scriptf{\tilde \scriptf \one_E }}_{L^q_\xi L^q_\theta}
\le   \norm{\tilde \scriptf \one_E}_{L^q_\xi L^p_n }
\le  \norm{\tilde\scriptf \one_E}_{L^p_n L^q_\xi}
\le  \bestAqd |E|^{1/p}.
\end{equation*}
Thus ${\mathbf A}(q,d,\kappa)\le \bestAqd$.
Consideration of product sets $E=A\times\{0\}$
immediately yields the reverse inequality. 

Although these partial Fourier transforms commute in the sense \eqref{FTscommute}, 
this reasoning breaks down if
they are applied in reversed order, because the partial Fourier transform
of $\one_E$ with respect to the $\integers^\kappa$ coordinate, a function
of $(\theta,x)\in\torus^\kappa\times\reals^d$,  is no longer
a scalar multiple of the indicator function of a subset of $\reals^d$
for fixed values of $\theta$.
For the same reason, in the analysis of ${\mathbf A'}(q,d,\kappa)$
these will have to be composed in the reverse order. 

Next consider ${\mathbf A'}(q,d,\kappa)$.
For any $f\in L^p(\integers^\kappa\times\reals^d)$
and any Lebesgue measurable set $E\subset\torus^\kappa\times\reals^d$
with $|E|\in\reals^+$, writing $E_\theta=\{\xi\in\reals^d: (\theta,\xi)\in E\}$,
\begin{align*}
\big|\int_E \widehat{f}\,\,\big|
&= \Big| \iint \tilde\scriptf\scriptf(f)(\theta,\xi)\one_E(\theta,\xi)\,d\xi\,d\theta \Big|
\\& \le \int  \norm{\tilde\scriptf\scriptf(f)(\theta,\cdot)}_{L^{q,\infty}_\theta} |E_\theta|^{1/p}\,d\theta
\\& \le \int \bestAqd \norm{\scriptf(f)(\theta,\cdot)}_{L^p_x} |E_\theta|^{1/p}\,d\theta,
\end{align*}
since $\bestAqd$ is the optimal constant in the
dual inequality $\norm{\widehat{h}}_{L^{q,\infty}(\reals^d)}\le C \norm{h}_{L^p(\reals^d)}$.
Therefore
\begin{align*}
\big|\int_E \widehat{f}\,\,\big|
& \le
\bestAqd |E|^{1/p} \big( \int \norm{\scriptf(f)(\theta,\cdot)}_{L^p_x}^q\,d\theta\big)^{1/q}
\\&  =
\bestAqd |E|^{1/p} \norm{\scriptf(f)}_{L^q_\theta L^p_x}
\\& \le  
\bestAqd |E|^{1/p} \norm{\scriptf(f)}_{L^p_x L^q_\theta}
\\& \le  
\bestAqd |E|^{1/p} \norm{f}_{L^p_x L^p_n}
\\&  =  
\bestAqd |E|^{1/p} \norm{f}_{L^p}.
\end{align*}
Thus ${\mathbf A'}(q,d,\kappa) \le \bestAqd$.
Once again, the reverse inequality follows by consideration of functions
$f(n,x)$ that are supported on a single value of $n$.
\end{proof}

For a function $f:\integers^\kappa\times\reals^d\to\complex$
and $m\in\integers^\kappa$ define $f_m:\reals^d\to\complex$ by 
\begin{equation} f_m(x)=f(m,x).  \end{equation}
Continue to write $E_m=\{x\in\reals^d: (m,x)\in E\}$ for $E\subset\integers^\kappa\times\reals^d$.

The main results of this subsection are the two parallel Propositions
\ref{prop:singlem1} and \ref{prop:singlem2}. They will be proved at the end
of this subsection after discussion of requisite lemmas.

\begin{proposition} \label{prop:singlem1}
Let $d,\kappa\ge 1$ and $q\in(2,\infty)$. 
Let $\delta>0$ be small.
Let $E\subset\integers^\kappa\times \reals^d$ be a Lebesgue measurable set with $|E|\in\reals^+$.
If $\norm{\widehat{\one_E}}_q\ge (1-\delta)\bestAqd |E|^{1/q'}$ then 
there exists $m\in\integers^\kappa$ such that
\begin{equation} |E_m|\ge (1-o_\delta(1))|E|.  \end{equation}
\end{proposition}

In fact the proof will give
$|E_m|\ge (1-C\delta^\rho) |E|$ for some constants $C,\rho\in\reals^+$, but
the $o_\delta(1)$ bound suffices in the sequel.

The proof will use the representation $\widehat{\cdot} = \scriptf\circ\tilde\scriptf$
together with specific properties of the factor $\scriptf$, but
the only properties of $\tilde\scriptf$ used are that it is linear, acts with respect 
to the $\reals^d$ variable alone, and satisfies a restricted type $(p,q)$ inequality 
with optimal constant $\bestAqd$.

\begin{proposition} \label{prop:singlem2}
Let $d,\kappa\ge 1$ and $q\in(2,\infty)$. 
Let $\delta>0$ be small.
Let $0\ne f\in L^{q'}(\integers^\kappa\times\reals^d)$.
If $\norm{\widehat{f}}_{q,\infty}\ge (1-\delta)\bestAqd \norm{f}_{q'}$ then 
there exists $m\in\integers^\kappa$ such that
\begin{equation} \norm{f_m}_{L^{q'}(\reals^d)}
\ge (1-o_\delta(1))\norm{f}_{L^{q'}(\integers^\kappa\times\reals^d)}.  \end{equation}
\end{proposition}

The following lemma will be used in the proof of Proposition~\ref{prop:singlem1}.
It is a direct consequence of the chain of inequalities in the analysis of $\bestA'(q,d,\kappa)$ 
in the proof of Lemma~\ref{lemma:productA}, so requires no further proof.

\begin{lemma} \label{lemma:mixednormlowerbound}
Let $d,\kappa\ge 1$ and $q\in(2,\infty)$. Set $p=q'$.
Let $\delta>0$ be small.
Let $E\subset\integers^\kappa\times \reals^d$ be a Lebesgue measurable set with $|E|\in\reals^+$.
If $\norm{\widehat{\one_E}}_q\ge (1-\delta)\bestAqd |E|^{1/p}$ then all of the following hold: 
\begin{align} 
\norm{\scriptf\tilde\scriptf{\one_E}}_{L^q_\xi L^q_\theta} &\ge (1-\delta) 
\norm{\tilde\scriptf(\one_E)}_{L^q_\xi L^p_n}
\label{eq:triad1}
\\ \norm{\tilde\scriptf(\one_E)}_{L^q_\xi L^p_n}
&\ge (1-\delta) \norm{\tilde\scriptf(\one_E)}_{L^p_n L^q_\xi}
\label{eq:triad2}
\\ \norm{\tilde\scriptf(\one_E)}_{L^p_n L^q_\xi} &\ge (1-\delta)\bestAqd |E|^{1/p}.
\label{eq:triad3}
\end{align}
\end{lemma}

To simplify notation, we continue to write $p=q'$ below.

\begin{lemma} \label{lemma:BST}
Let $E\subset\integers^\kappa\times\reals^d$ satisfy
$\norm{\widehat{\one_E}}_q \ge (1-\delta)\bestAqd |E|^{1/p}$ where $p=q'$.
There exists a disjointly supported decomposition
\[ \tilde\scriptf\one_E(n,\xi) = g(n,\xi)+h(n,\xi)\]
where 
\[ \norm{h}_{L^q_\xi L^p_n} 
\le o_\delta(1)|E|^{1/p}\]
and for each $\xi\in\reals^d$ there exists $n(\xi)\in\integers^\kappa$ such that 
\[ g(n,\xi)= 0  \text{ for all } n\ne n(\xi).\]
\end{lemma}

The conclusion is weaker than the one we seek eventually to establish
in Proposition~\ref{prop:singlem2}, in that $n(\xi)$ is permitted here to depend on $\xi$.

\begin{proof}[Proof of Lemma~\ref{lemma:BST}]
Let $\eta=\delta^{1/2}$.
For $\xi\in\reals^d$, define $\varphi_\xi:\integers^\kappa\to\complex$
by \[\varphi_\xi(n) = \tilde\scriptf\one_E(n,\xi).\]
This is well-defined for almost every $\xi$.

Define 
\begin{equation}
\scriptg=\set{\xi\in\reals^d:
\norm{\varphi_\xi}_{L^p(\integers^\kappa)}\ne 0
\text{ and } 
\norm{\widehat{\varphi_\xi}}_{L^q(\torus^\kappa)}
\ge (1-\eta)
\norm{\varphi_\xi}_{L^p(\integers^\kappa)} }.
\end{equation}
The Fourier transform here is that for the group $\integers^\kappa$.
With this definition,
\begin{align*}
\norm{\scriptf\tilde\scriptf(\one_E)}_{L^q_\xi L^q_\theta}^q
& =
\int_{\reals^d\setminus\scriptg} 
\norm{\widehat{\varphi_\xi}}_{L^q(\torus^\kappa)}^q\,d\xi
+ \int_{\scriptg} 
\norm{\widehat{\varphi_\xi}}_{L^q(\torus^\kappa)}^q\,d\xi
\\& 
\le
\int_{\reals^d\setminus\scriptg} 
(1-\eta)^{-1} \norm{\varphi_\xi}_p^q \,d\xi
+ \int_{\scriptg} 
\norm{\varphi_\xi}_p^q \,d\xi
\\&
\le \int_{\reals^d} 
\norm{\varphi_\xi}_p^q \,d\xi
-c\eta 
\int_{\reals^d\setminus\scriptg} 
\norm{\varphi_\xi}_p^q \,d\xi.
\end{align*}
Therefore by \eqref{eq:triad1},
\begin{multline}
(1-\delta)^q \int_{\reals^d} \norm{\varphi_\xi}_p^q \,d\xi
= (1-\delta)^q \norm{\tilde\scriptf(\one_E)}_{L^q_\xi L^p_n}^q
\\
\le 
\norm{\scriptf\tilde\scriptf(\one_E)}_{L^q_\xi L^q_\theta}^q
\le \int_{\reals^d} \norm{\varphi_\xi}_p^q \,d\xi
-c\eta \int_{\reals^d\setminus\scriptg} \norm{\varphi_\xi}_p^q \,d\xi.
\end{multline}
Consequently
\begin{equation}\label{eq:offG} \int_{\reals^d\setminus\scriptg} \norm{\varphi_\xi}_p^q \,d\xi 
\le C\eta^{-1} \delta \int_{\reals^d} \norm{\varphi_\xi}_p^q \,d\xi
\le  C\delta^{1/2} \norm{\tilde\scriptf(\one_E)}_{L^q_\xi L^p_n}^q \end{equation}
by the choice $\eta=\delta^{1/2}$.

According to \cite{christmarcos},
if $\xi\in\scriptg$,  that is, if
$\norm{\widehat{\varphi_\xi}}_{L^q(\torus^\kappa)} \ge (1-\eta)$,
then there exists $n(\xi)\in\integers^\kappa$ such that
\[\norm{\varphi_\xi}_{L^p(\integers^\kappa\setminus\{n(\xi)\})} \le o_\eta(1) 
\norm{\varphi_\xi}_{L^p(\integers^\kappa)}. \]
In this case define 
\begin{equation*}
g(n,\xi) = 
\begin{cases} &\varphi_\xi(n)=\tilde\scriptf \one_E(n,\xi)
\ \text{for the single point $n=n(\xi)$,}
\\
&0\  \text{for all other $n\in\integers^\kappa$.}
\end{cases}\end{equation*}
For $\xi\in\scriptg$ define $h(n,\xi)=\tilde\scriptf\one_E(n,\xi)$ for all $n\ne n(\xi)$.

On the other hand, for $\xi\notin\scriptg$ define 
$h(n,\xi) = \varphi_\xi(n)=\tilde\scriptf \one_E(n,\xi)$
for all $n\in\integers^\kappa$,
and $g(n,\xi)\equiv 0$.
Thus 
\[ \tilde\scriptf\one_E(n,\xi) = g(n,\xi)+h(n,\xi)\]
for almost  all $(n,\xi)\in\integers^\kappa\times\reals^d$,
and this is a disjointly supported decomposition of $\tilde\scriptf\one_E$.

The function $g$ satisfies the stated conclusion of the lemma by its definition.
On the other hand, the summand $h$ has small $L^q_\xi L^p_n$ norm. Indeed,
\begin{align*}
\int_{\xi\in\scriptg} \big(\sum_{n\in\integers^\kappa}  |h(n,\xi)|^p\big)^{q/p}\,d\xi
&\le \int_{\xi\in\scriptg} o_\eta(1) \big(\sum_{n\in\integers^\kappa}  |\varphi_\xi(n)|^p\big)^{q/p}\,d\xi
\\& = o_\eta(1) \norm{\tilde\scriptf\one_E}_{L^q_\xi L^p_n}^q
\end{align*}
while
\begin{equation*}
\int_{\xi\notin\scriptg} \big(\sum_{n\in\integers^\kappa}  |h(n,\xi)|^p\big)^{q/p}\,d\xi
= \int_{\xi\notin\scriptg} \norm{\varphi_\xi}_{L^p_n}^{q}\,d\xi
\le C\delta^{1/2} \norm{\tilde\scriptf(\one_E)}_{L^q_\xi L^p_n}^q
\end{equation*}
by \eqref{eq:offG}.
\end{proof}

The conclusion of Lemma~\ref{lemma:BST}, with $\one_E(n,\xi)$ supported
at a single $\xi$--dependent point $n(\xi)$ for nearly all $\xi$,
is manifestly weaker than the type of conclusion asserted by Proposition~\ref{prop:singlem1},
that $\one_E(n,\xi)$ is supported at a single value of $n$, independent of $\xi$,
up to a small norm remainder. The balance of the proof of Proposition~\ref{prop:singlem1}
is a simple argument to bridge this gap.

\begin{proof}[Proof of Proposition~\ref{prop:singlem1}]
Let $p=q'$.
Let $E\subset\integers^\kappa\times\reals^d$ satisfy
$\norm{\widehat{\one_E}}_q \ge (1-\delta)\bestAqd |E|^{1/p}$.
Decompose $\tilde\scriptf(\one_E)(n,\xi) = g(n,\xi)+h(n,\xi)$
as in Lemma~\ref{lemma:BST}.

The upper bound $\norm{h}_{L^q_\xi L^p_n}\le o_\delta(1)|E|^{1/p}$ 
and the lower bounds \eqref{eq:triad2}, \eqref{eq:triad3} together imply
\begin{equation} \label{eq:triad2b}
\norm{g}_{L^q_\xi L^p_n} \ge (1-o_\delta(1)) \norm{g}_{L^p_n L^q_\xi}.  \end{equation}
Now because $n\mapsto g(n,\xi)$ is supported on a set of cardinality at most one for each $\xi$,
\begin{equation}
\norm{g}_{L^q_\xi L^p_n}^q 
= \int (\sum_n |g(n,\xi)|^p)^{q/p}\,d\xi
= \int |g(n(\xi),\xi)|^q\,d\xi
= \norm{g}_{L^q_\xi L^q_n}^q 
= \norm{g}_{L^q_n L^q_\xi}^q, 
\end{equation}
so
\begin{equation} \label{eq:bst2}
\norm{g}_{L^q_n L^q_\xi} 
\ge (1-o_\delta(1)) \norm{g}_{L^p_n L^q_\xi}.  \end{equation}
Furthermore
\begin{equation} \label{eq:bst3}
\norm{g}_{L^q_n L^q_\xi}^q
\le \sup_{m\in\integers^\kappa} \norm{g(m,\cdot )}_{L^q_\xi}^{q-p}
\norm{g}_{L^p_n L^q_\xi}^p.
\end{equation}
Since $q>p$, \eqref{eq:bst2} and \eqref{eq:bst3} together give
\begin{equation}
\sup_{m\in\integers^\kappa} \norm{g(m,\cdot )}_{L^q_\xi}\ge (1-o_\delta(1))
\norm{g}_{L^p_n L^q_\xi} \ge (1-o_\delta(1)) \bestAqd |E|^{1/p}.
\end{equation}
Because $\tilde\scriptf(\one_E)=g+h$ and $g,h$ have disjoint supports, 
$\norm{\tilde\scriptf(\one_E)(m,\cdot)}_{L^q_\xi}
\ge \norm{g(m,\cdot)}_{L^q_\xi}$ for any $m$ and therefore there exists
$m$ satisfying
\begin{equation} \label{eq:bst4}
\norm{\tilde\scriptf(\one_E)(m,\cdot)}_{L^q_\xi}
\ge (1-o_\delta(1)) \bestAqd |E|^{1/p}.
\end{equation}

On the other hand,
by definition of $\bestAqd$ as the optimal constant in the inequality,
\begin{equation}
\norm{\tilde\scriptf(\one_E)(m,\cdot)}_{L^q_\xi}
\le \bestAqd|E_m|^{1/p}.
\end{equation}
Together with the lower bound \eqref{eq:bst4} for $|E_m|$, this implies that
\begin{equation} |E_m|\ge (1-o_\delta(1))|E|,  \end{equation}
concluding the proof of Proposition~\ref{prop:singlem1}.
\end{proof}

Proposition~\ref{prop:singlem2}
is proved by combining the reasoning in the proof of Proposition~\ref{prop:singlem1}
with the proof of the second conclusion of Lemma~\ref{lemma:productA}.
\qed

\subsection{Lifting to $\integers^d\times\reals^d$} \label{subsect:lifting}

Let $\scriptq_d = [-\tfrac12,\tfrac12)^d$ be the unit cube of sidelength
$1$ centered at $0\in\reals^d$.
Identify $\reals^d$ with $\integers^d+\scriptq_d$ 
via the map $\integers^d\times\scriptq_d \owns (n,y)\mapsto n+y\in\reals^d$.
Likewise identify $\torus^d = \reals^d/\integers^d$ with $\scriptq_d=[-\tfrac12,\tfrac12)^d$
in the natural way.
Express any $\xi\in\reals^d$ as $\xi=n(\xi)+\alpha(\xi)$
with $n(\xi)\in\integers^d$ and $\alpha(\xi)\in\scriptq_d$.

\begin{definition} \label{defn:lifting}
To any function $f:\reals^d\to\complex$ associate the function
$f^\dagger: \integers^d\times\reals^d\to\complex$
defined by
\begin{equation}
f^\dagger(n,x) = \begin{cases}
f(n+x) \ &\text{if } x\in\scriptq_d
\\
0 &\text{if $x\notin\scriptq_d$.}
\end{cases} \end{equation}
\end{definition}

We will also require a dual construction.
View $\reals^d$ as $\reals^d_\xi$, the Fourier dual of $\reals^d_x$. 

\begin{definition}
To any set $A\subset\reals^d$ we associate
$A^\ddagger\subset \torus^d\times\reals^d$, defined by 
\begin{equation}\label{eq:dualdagger} 
A^\ddagger=\{(\theta,\xi)\in\torus^d\times\reals^d: 
\theta+n(\xi) \in A\}. \end{equation}
\end{definition}
Thus $|A^\ddagger| =|A|$.

\begin{lemma}\label{lemma:lifting1}
Let $d\ge 1$ and $q\in(2,\infty)$.
Let $\delta,\eta>0$ be small.
Let $E\subset\reals^d$ be a Lebesgue measurable set with $|E|\in\reals^+$.
Suppose that 
\begin{equation} \distance(x,\integers^d)\le\eta\ \text{ for all } x\in E \end{equation}
and that
\begin{equation}
\norm{\widehat{\one_E}}_{L^q(\reals^d)} \ge (1-\delta)\bestAqd |E|^{1/q'}.
\end{equation}
Then 
\begin{equation}
\norm{\widehat{\one_{E^\dagger}}}_{L^q(\torus^d\times\reals^d)} 
\ge (1-\delta-o_\eta(1) )\bestAqd |E^\dagger|^{1/q'}.
\end{equation}
\end{lemma}
The notation $\widehat{\cdot}$ is used for the Fourier transform for any of the groups
$\reals^d$, $\integers^\kappa$, and $\integers^\kappa\times\reals^d$.

\begin{proof}
Corollary~9.2 of \cite{christHY} asserts a result for general functions $f:\integers^d\to\complex$,
whose specialization to indicator functions of sets is Lemma~\ref{lemma:lifting1}. 
\end{proof}

\begin{lemma} \label{lemma:lifting2}
Let $d\ge 1$ and $q\in(2,\infty)$.
Let $\delta,\eta>0$ be small.
Let $0\ne f\in L^{q'}(\reals^d)$.
Suppose that 
\begin{equation} f(x)\ne 0 \Rightarrow \distance(x,\integers^d)\le\eta \end{equation}
and that
\begin{equation}
\norm{\widehat{f}}_{q,\infty} \ge (1-\delta)\bestAqd \norm{f}_{q'}. 
\end{equation}
Then 
\begin{equation}
\norm{\widehat{f^\dagger}}_{q,\infty} \ge (1-\delta-o_\eta(1) )\bestAqd \norm{f^\dagger}_{q'}.
\end{equation}
\end{lemma}

In the statement of the lemma and in its proof,
$\widehat{f}$ and $\widehat{f^\dagger}$
denote the Fourier transform of $f$ for the group $\reals^d$,
and the Fourier transform of $f^\dagger$ for the group $\integers^d\times\reals^d$,
respectively.

\begin{proof} Let $p=q'$.
Choose $E\subset\reals^d$ with $|E|\in\reals^+$ satisfying
$|\int_E \widehat{f}| \ge (1-2\delta) \bestAqd \norm{f}_p|E|^{1/p}$.

Represent elements of $\reals^d$ as $\xi = n(\xi)+\alpha(\xi)$ 
where $n(\xi)\in\integers^d$ and $\alpha(\xi)\in\scriptq_d$.
It is shown in the proof of Lemma~9.1 of \cite{christHY}
that 
\begin{equation}
\norm{\widehat{f^\dagger}(\theta,\xi)-\widehat{f}(n(\xi)+\theta)}_{L^q(\torus^d\times\reals^d)}
\le o_\eta(1)\norm{f}_p.
\end{equation}
The quantity $o_\eta(1)$ depends also on $q,d$ but not on $f$.


Identify $\torus^d$ with $\scriptq_d$ in the natural way,
by associating to each element of $\scriptq_d\subset\reals^d$ its equivalence class in $\reals^d/\integers^d
= \torus^d$. Then
\begin{align*}
|\int_{E^\ddagger} \widehat{f^\dagger}|
&= \big| \int_{\reals^d}\int_{\torus^d} 
\one_{E^\ddagger}(\theta,\xi) \widehat{f^\dagger}(\theta,\xi)\,d\theta\,d\xi\big|
\\&= \big| \int_{\reals^d}\int_{\scriptq_d} 
\one_{E^\ddagger}(\theta,\xi) \widehat{f}(n(\xi)+\theta)\,d\theta\,d\xi\big|
+o_\eta(1)\norm{f}_p|E^\ddagger|^{1/p}
\\ &= \big| \int_{\reals^d}\int_{\scriptq_d} \one_{E}(n(\xi)+\theta) \widehat{f}(n(\xi)+ \theta)\,d\theta\,d\xi\big|
+o_\eta(1)\norm{f}_p|E^\ddagger|^{1/p}.
\end{align*}
Now writing $\xi = n(\xi)+\alpha$ with $\alpha\in\scriptq_d$ gives
\begin{align*}
\int_{\reals^d}\int_{\scriptq_d} \one_{E}(n(\xi)+\theta) \widehat{f}(n(\xi)+ \theta)\,d\theta\,d\xi
&=
\sum_{n\in\integers^d} \int_{\scriptq_d} 
\Big(\int_{\scriptq_d} \one_{E}(n+\theta) \widehat{f}(n+ \theta)\,d\theta\Big)\,d\alpha
\\& = \sum_{n\in\integers^d} \int_{\scriptq_d} 
\one_{E}(n+\theta) \widehat{f}(n+ \theta)\,d\theta
\\& = \int_{\reals^d} \one_E(\xi) \widehat{f}(\xi)\,d\xi  \end{align*}
since $|\scriptq_d|=1$.
Thus 
\begin{align*}
|\int_{E^\ddagger} \widehat{f^\dagger}|
= \big|\int_E \widehat{f}\,\,\big|
+o_\eta(1)\norm{f}_p|E|^{1/p}
\ge (1-2\delta) \bestAqd \norm{f}_p|E|^{1/p}
+o_\eta(1)\norm{f}_p|E|^{1/p},
\end{align*}
which is equal to
$(1-2\delta) \bestAqd \norm{f^\dagger}_p|E^\ddagger|^{1/p} +o_\eta(1)\norm{f^\dagger}_p|E^\ddagger|^{1/p}$.
\end{proof}

\subsection{Spatial localization}

\begin{proposition}
\label{prop:spatialloc1}
Let $d\ge 1$ and $q\in(2,\infty)$. 
For every $\eps>0$ there exists $\delta>0$ with the following property.
Let $E\subset\reals^d$ be a Lebesgue measurable set with $|E|\in\reals^+$.
If $\norm{\widehat{\one_E}}_{q} \ge (1-\delta)\bestAqd |E|^{1/q'}$
then there exists an ellipsoid $\scripte\subset\reals^d$ such that
\begin{align} 
|E\setminus \scripte| &\le \eps |E|
\\|\scripte| &\le C_\eps |E|.
\end{align}
\end{proposition}

\begin{proposition}
\label{prop:spatialloc2}
Let $d\ge 1$ and $q\in(2,\infty)$. 
For every $\eps>0$ there exists $\delta>0$ with the following property.
Let $0\ne f\in L^{q'}(\reals^d)$ satisfy $\norm{\widehat{f}}_{q,\infty} \ge (1-\delta)\bestAqd\norm{f}_{q'}$.
There exist an ellipsoid $\scripte\subset\reals^d$ and a decomposition $f = \varphi+\psi$ such that 
\begin{gather}
\norm{\psi}_{q'}<\eps\norm{f}_{q'}
\\ \varphi\equiv 0 \text{ on } \reals^d\setminus\scripte 
\\ \norm{\varphi}_\infty |\scripte|^{1/{q'}}\le C_\eps\norm{f}_{q'}.
\end{gather}
\end{proposition}

\begin{proof}[Proof of Proposition~\ref{prop:spatialloc1}]
Continue to write $p=q'$.
Let $\delta$ be small, and assume that
$\norm{\widehat{\one_E}}_{q} \ge (1-\delta)\bestAqd |E|^{1/p}$.
By Lemma~\ref{lemma:NE1}
there exists a continuum multiprogression $P$ of rank $O_\delta(1)$
such that $|P|\le O_\delta(1)|E|$ and $|E\setminus P|\le o_\delta(1)|E|$. 
By replacing $E$ by its dilate $|P|^{-1/d}E$ we may assume that $|P|=1$.

Lemma~5.2 of \cite{christHY} states that for any $r<\infty$ and
any continuum multiprogression
$P$ of rank at most $r$ and Lebesgue measure equal to $1$
there exists an invertible affine automorphism $\scriptt$
of $\reals^d$ such that $\distance(x,\integers^d)= o_\delta(1)$ for every $x\in P$,
and $|\scriptt(P)|$ is bounded below by a positive quantity that depends only on $d,q,\delta,r$;
the same lower bound holds for $|\scriptt(E\cap P)|$.
Choose such an affine automorphism $\scriptt$.

Change notation, denoting the set $\scriptt(E)$ again by $E$ and $\scriptt(P)$ by $P$.
Affine transformations of ellipsoids are ellipsoids, so 
it suffices to establish the stated conclusions for the new set $E$.  
No upper bound on the Jacobian determinant of $\scriptt$ is possible, so we obtain no upper
bound for $|\scriptt(E)|$.

The ratio $\norm{\widehat{\one_A}}_q/|A|^{1/p}$ is invariant under affine transformations of $A$. 
Consequently the modified set $E$ still satisfies
$\norm{\widehat{\one_E}}_{q} \ge (1-\delta)\bestAqd |E|^{1/p}$.
Since $|E\setminus P|\le o_\delta(1)|E|$,
$\norm{\widehat{\one_{E\setminus P}}}_q \le o_\delta(1)|E|^{1/p}$.
Therefore
\begin{align*} 
\norm{\widehat{\one_{E\cap P}}}_{q} 
&\ge \norm{\widehat{\one_{E}}}_{q} 
- \norm{\widehat{\one_{E\setminus P}}}_{q} 
\\& \ge (1-\delta(1))\bestAqd |E|^{1/p}-o_\delta(1)|E|^{1/p}
\\& \ge (1-o_\delta(1))\bestAqd |P\cap E|^{1/p}.
\end{align*}

Consider the lifted set 
$ (E\cap P)^\dagger \subset\integers^d\times\reals^d$ introduced in Definition~\ref{defn:lifting}, 
that is, the set of all $(n,x)$ such that $n+x\in E\cap P$ and $x\in\scriptq_d$.
By Lemma~\ref{lemma:lifting1}, 
\[\norm{\widehat{\one_{(E\cap P)^\dagger}}}_q \ge (1-o_\delta(1))|(E\cap P)^\dagger|^{1/p},\]
where $\widehat{\cdot}$ denotes the Fourier transform for $\integers^d\times\reals^d$.

By Proposition~\ref{prop:singlem1}
there exists $m\in\integers^d$ such that 
\begin{equation}
\big| (E\cap P)^\dagger\setminus(\{m\}\times\reals^d)\big| \le o_\delta(1)|(E\cap P)^\dagger|.
\end{equation}
In terms of the given set $E$, this means that
$\big| (E\cap P)\setminus (m+\scriptq_d)\big| \le o_\delta(1)|E|$ and hence
\begin{equation} \big| E\setminus (m+\scriptq_d)\big| 
\le \big| (E\cap P)\setminus (m+\scriptq_d)\big| + |E\setminus P|
\le o_\delta(1)|E|.  \end{equation}
The cube $m+\scriptq_d$ is contained in a ball of comparable measure,
which is equivalent by affine invariance to the stated conclusion. 
\end{proof}

The proof of Proposition~\ref{prop:spatialloc2}
is nearly identical to that of Proposition~\ref{prop:spatialloc1};
the necessary preliminary results are established above.
Therefore the details are omitted.  \qed

\subsection{Frequency localization}

\begin{proposition}
\label{prop:frequencylocalization}
Let $d\ge 1$ and $q\in(2,\infty)$. 
For every $\eps>0$ there exists $\delta>0$ with the following property. 
Let $E\subset\reals^d$ be a Lebesgue measurable set with $|E|\in\reals^+$
satisfying
$\norm{\widehat{\one_E}}_{q} \ge (1-\delta)\bestAqd |E|^{1/q'}$.
Then there exist an ellipsoid $\scripte'\subset\reals^d$ 
and a disjointly supported decomposition $\widehat{\one_E} = \Phi+\Psi$ such that 
\begin{gather}
\norm{\Psi}_{q}<\eps |E|^{1/q'}
\\ \Phi\equiv 0 \text{ on } \reals^d\setminus\scripte'
\\ \norm{\Phi}_\infty |\scripte'|^{1/q}\le C_\eps |E|^{1/q'}.
\end{gather}
\end{proposition}

\begin{proof}
Set $p=q'$.
Set $f = \widehat{\one_E}\cdot |\widehat{\one_E}|^{q-2}$.
Then $|f|\equiv |\widehat{\one_E}|^{q-1}$,
so 
$|f|^p = |\widehat{\one_E}|^{p(q-1)} = |\widehat{\one_E}|^q$, 
$f\in L^p$. This function $f$ satisfies
\begin{equation}  \norm{\widehat{f}}_{q,\infty} \ge (1-\delta)^q \bestAqd \norm{f}_p.  \end{equation}
Indeed,
\begin{equation} \label{eq:ndeed} 
\norm{f}_p = \norm{\widehat{\one_E}}_q^{q-1} \le \bestAqd^{q-1} |E|^{(q-1)/p}.
\end{equation}
Moreover, denoting by $g^\vee$ the inverse Fourier transform of a function $g$,
\begin{equation} \norm{\widehat{\one_E}}_q^q = \langle \widehat{\one_E},\,f\rangle
= \langle \one_E, f^\vee \rangle 
\le |E|^{1/p}\norm{f^\vee}_{q,\infty}
=  |E|^{1/p}\norm{\widehat{f}}_{q,\infty};
\end{equation}
the relation
$ |\langle \one_E, f^\vee \rangle| \le |E|^{1/p}\norm{f^\vee}_{q,\infty}$
is a tautology because of the definition chosen for the $L^{q,\infty}$ norm.
Consequently
\begin{equation} 
\norm{\widehat{f}}_{q,\infty}
\ge |E|^{-1/p} \norm{\widehat{E}}_q^q
\ge (1-\delta)^q \bestAqd^q |E|^{(q-1)/p}
\ge (1-\delta)^q \bestAqd \norm{f}_p,
\end{equation}
using \eqref{eq:ndeed} to obtain the final inequality.

Apply Proposition~\ref{prop:spatialloc2} to $f$ to obtain an ellipsoid $\scripte'$
and a disjointly supported decomposition $f = \varphi+\psi$
with the properties listed in that Proposition.
Define $\Phi,\Psi$ by the relations
\[\Phi = \varphi|\varphi|^{(2-q)/(q-1)} \ \text{ and }\ \Psi = \psi|\psi|^{(2-q)/(q-1)}\] 
to conclude the proof.
\end{proof}

Although the dual inequality $\norm{\widehat{f}}_{q,\infty}\le \bestAqd \norm{f}_{q'}$ 
is not our main object of study,
the various results developed above concerning it are required in this proof.

We have now associated two ellipsoids to any near-extremizing set $E$.
The set $E$ itself is nearly contained in $\scripte$, while 
$\widehat{\one_E}$ is nearly supported in $\scripte'$.
It is clear that by the uncertainty principle, broadly construed,
the product $|\scripte|\cdot |\tilde\scripte|$ is bounded below
by some positive constant, provided that $\eps$ is sufficiently small
in Propositions~\ref{prop:spatialloc1} and \ref{prop:frequencylocalization}.
See \cite{christHY}.
We need to show next that $|\scripte|\cdot |\tilde\scripte|$ is bounded above. 

\begin{definition}
The polar set $\scripte^*$ of an ellipsoid $\scripte\subset\reals^d$ centered at $0$ is 
\[ \scripte^*=\set{y: |\langle x,y\rangle|\le 1\ \text{ for every $x\in \scripte$}}\]
where $\langle \cdot,\cdot\rangle$ denotes the Euclidean inner product.
\end{definition}

\begin{definition} \label{defn:normalize}
Let $\Theta:\reals^+\to\reals^+$ be a continuous function satisfying $\lim_{t\to \infty} \Theta(t)=0$.
Let $r\in[1,\infty)$ be an exponent. Let $\eta>0$.
We say that a function $0\ne f\in L^r(\reals^d)$ is $\eta$--normalized with respect $\Theta,r$ if
for all $R\in[1,\infty)$,
\begin{align}
\label{eq:normalized1} 
\int_{|x|\ge R} |f|^r\,dx &\le \big( \Theta(R) + \eta\big) \norm{f}_r^r
\\ \label{eq:normalized2} 
\int_{|f(x)|\ge R} |f|^r\,dx &\le \big( \Theta(R) + \eta\big) \norm{f}_r^r.
\end{align}
\end{definition}
Inequality \eqref{eq:normalized1}
prevents the mass represented by $|f|^r$ from being too diffuse,
while inequality \eqref{eq:normalized2} prevents excessive concentration.

For any ellipsoid $\scripte\subset\reals^d$ there exists $T_\scripte\in\aff(d)$
satisfying $T_\scripte(\bb)=\scripte$. 
If $T,T'$ are any two such transformations then $T'=T\circ S$ for some element
$S$ of the orthogonal group.

\begin{definition} \label{defn:normalize2}
Let $\scripte\subset\reals^d$ be any ellipsoid. 
A function $0\ne f\in L^r(\reals^d)$ is said to be $\eta$--normalized with respect $\Theta,r,\scripte$ if
$f\circ T_\scripte$ is $\eta$--normalized with respect to $\Theta,r$.
A measurable set $E$ is $\eta$--normalized with respect to $\Theta,r,\scripte$
if $|E|\in\reals^+$ and the function $\one_E$ has this property.
\end{definition}

Although $T_\scripte$ is not uniquely defined, all choices lead to identical definitions because
of the rotation-invariance of Definition~\ref{defn:normalize}.

The result shown thus far can be reformulated in these terms:
	\begin{lemma}\label{lemma:reformulate}
Let $d\ge 1$ and $q\in(2,\infty)$.  
There exists a continuous function $\Theta:\reals^+\to\reals^+$ 
satisfying $\lim_{t\to \infty} \Theta(t)=0$ with the following property.
For every sufficiently small $\delta>0$ and any 
Lebesgue measurable set $E\subset\reals^d$ with $|E|\in\reals^+$
that satisfies $\norm{\widehat{\one_E}}_q \ge (1-\delta)  \bestAqd |E|^{1/p}$,
there exist ellipsoids $\scripte,\scripte'\subset\reals^d$ such that
$E$ is $o_\delta(1)$--normalized with respect to $\Theta,q',\scripte$
and $\widehat{\one_E}$ is $o_\delta(1)$--normalized with respect to $\Theta,q,\scripte'$.
\end{lemma}
The quantities denoted by $o_\delta(1)$ depend only on $\delta,d,q$ and on a choice of auxiliary
function $\Theta$, but not on $E$, and tend to zero as $\delta\to 0$ while $d,q,\Theta$ remain fixed.

\subsection{Compatibility of approximating ellipsoids}
The next step is to show that the ellipsoids $\scripte,\scripte'$ are dual
to one another, up to bounded factors and independent translations.
For $s\in\reals^+$ and $E\subset\reals^d$, we consider the dilated set 
$sE=\set{sy: y\in E}$.

\begin{lemma}
\label{lemma:tightness}
Let $d\ge 1$ and $q\in(2,\infty)$.  
Let $\rho>0$.
Let $\Theta:\reals^+\to\reals^+$ be a continuous function satisfying $\lim_{t\to \infty} \Theta(t)=0$.
There exist $\eta>0$ and $C<\infty$ with the following property.
Let $\scripte,\scripte'\subset\reals^d$ be ellipsoids centered at $0\in\reals^d$, 
and let $u,v\in\reals^d$.
Let $E\subset\reals^d$ be a Lebesgue measurable set, 
and suppose that $\norm{\widehat{\one_E}}_q\ge\rho |E|^{1/q'}$.
Suppose moreover that $E$ is $\eta$--normalized with respect to $\Theta,q',\scripte+u$,
and that $\widehat{\one_E}$ is $\eta$--normalized with respect to $\Theta,q,\scripte'+v$.
Then
\begin{equation} \scripte \subset C\scripte'^* \ \text{ and } 
\ \scripte'\subset C\scripte^*. \end{equation}
\end{lemma}

\begin{proof} A more general result, with $\one_E$ replaced by an arbitrary function in $L^{q'}(\reals^d)$,
is Lemma~11.2 of \cite{christHY}.  \end{proof}

\subsection{Precompactness}

\begin{lemma} \label{lemma:cannormalize}
Let $d\ge 1$ and $q\in(2,\infty)$.  Let $p=q'$.
There exists a continuous function $\Theta:\reals^+\to\reals^+$ satisfying $\lim_{t\to \infty} \Theta(t)=0$
with the following property.
For any sufficiently small $\delta>0$ and
any Lebesgue measurable set $E\subset\reals^d$ with $|E|\in\reals^+$
that satisfies $\norm{\widehat{\one_E}}_q \ge (1-\delta)  \bestAqd |E|^{1/q'}$
there exists $T\in\aff(d)$
such that $T(E)$ is $o_\delta(1)$--normalized with respect to $\Theta,q',\bb$, 
and $\widehat{\one_{T(E)}}$  is $o_\delta(1)$--normalized with respect to $\Theta,q,\bb$. 
\end{lemma}

The corresponding step in \cite{christHY} is immediate from the analogue of
Lemma~\ref{lemma:tightness}
by the affine-- and modulation--invariance of the ratio $\norm{\widehat{f}}_q/\norm{f}_p$,
but here no modulation-invariance is available. 

\begin{proof}
Let the ellipsoids $\scripte,\scripte'$ and auxiliary function
$\Theta$ satisfy the conclusions
of Lemmas~\ref{lemma:reformulate} and \ref{lemma:tightness}. 
Since the ratio $\norm{\widehat{f}}_q/\norm{f}_p$
is invariant under precomposition of $f$ with affine automorphisms of $\reals^d$,
we may assume without loss of generality that $\scripte=\bb$.
Then by Lemma~\ref{lemma:tightness}, we may take $\scripte'$
to also be a ball, with $|\scripte'|\asymp 1$.
If $\delta$ is sufficiently small then 
$|E\setminus C_0\bb|\ll |E|$ for a certain constant $C_0$
that depends only on $d,q$.

There exists a constant
$\rho_0>0$ which depends only on $d,q$ such that
$|\widehat{\one_E}(\xi)|\ge\rho_0$ whenever $|\xi|\le\rho_0$.
Indeed, split $E = E'\cup E''$ where $E'=E\cap s\bb$ where
$s\in\reals^+$ is chosen sufficiently large to ensure that
$|E''|\le \tfrac14|E|=\tfrac14$. Such a parameter $s$ may be taken
to be independent of $E$ so long as $\delta$ is sufficiently small
since $E$ is normalized.
Then $|\widehat{\one_{E'}}(0)|=|E'|\ge \tfrac34$, 
$\widehat{\one_{E'}}$ is Lipschitz continuous with Lipschitz constant $O(s)$,
and $\norm{\widehat{\one_{E''}}}_\infty \le |E''|\le \tfrac14$.


It is given that the function $\widehat{\one_E}$ is normalized with respect to $\scripte'$.
The center of $\scripte'$ must lie
within a bounded distance of the origin, or the lower bound
$|\widehat{\one_E}(\xi)|\ge \tfrac12$ for all $|\xi|\le Cs^{-1}$
would contradict the first normalization inequality \eqref{eq:normalized1}.
Therefore $\scripte'$ can be replaced by a ball centered at the origin,
whose radius is bounded above and below by positive
constants that depend only on $q,d$.
With respect to this modified $\scripte'$,
$\widehat{\one_E}$ is still normalized, with respect to a modified
function $\tilde\Theta$
which still satisfies $\lim_{t\to\infty}\tilde\Theta(t)=0$,
and depends only on the function $\Theta$ given by
Lemmas~\ref{lemma:reformulate} and \ref{lemma:tightness}, not on $E,\delta$.
\end{proof}

\begin{proposition}
Let $d\ge 1$ and $q\in(2,\infty)$.  
Let $(E_\nu)$ be a sequence of Lebesgue measurable subsets of $\reals^d$
with $|E_\nu|\in\reals^+$.
Suppose that $\lim_{\nu\to\infty} |E_\nu|^{-1/q'}\norm{\widehat{\one_{E_\nu}}}_q=\bestAqd$.
Then there exists a sequence of elements $T_\nu\in\aff(d)$ such that 
$|T_\nu(E_\nu)|=1$ for all $\nu$ and
the sequence of indicator functions $(\one_{T_\nu(E_\nu)})$ is precompact in $L^{q'}(\reals^d)$.
\end{proposition}

If a sequence $f_\nu$ of indicator functions of sets converges in $L^{q'}$ norm,
then the limiting function is necessarily the indicator function of a set,
so the conclusion is the type of convergence desired.
Nonetheless, we will directly prove convergence in a stronger norm.

\begin{proof}
According to Lemma~\ref{lemma:cannormalize}, 
there exist a positive continuous auxiliary function $\Theta$ that satisfies $\lim_{t\to\infty}\Theta(t)=0$,
and a sequence of affine automorphisms $(T_\nu)$ of  $\reals^d$
satisfying the conclusions of that lemma.
By composing $T_\nu$ with a dilation of $\reals^d$ we may arrange
that $|T_\nu(E_\nu)|=1$ for all $\nu$.
The other conclusions of Lemma~\ref{lemma:cannormalize} are unaffected.
Thus both $E_\nu$ and $\widehat{\one_{E_\nu}}$ satisfy uniform upper bounds with respect to $\Theta$ of the type
formulated in Definition~\ref{defn:normalize}, with parameters $\eta=\eta_\nu$
that satisfy $\lim_{\nu\to\infty} \eta_\nu=0$.

Write $E_\nu$ in place of $T_\nu(E_\nu)$ to simplify notation.
Since the ratio
$\norm{\widehat{\one_E}}_q/|E|^{1/p}$ is invariant under replacement of $E$ by any affine 
image of $E$, we still have an extremizing sequence;
$\norm{\widehat{\one_{E_\nu}}}_q/|E_\nu|^{1/p}\to\bestAqd$ as $\nu\to\infty$.

As in the proof of Lemma~12.2 of \cite{christHY},
it follows immediately\footnote{ 
This reasoning relies on the Hausdorff-Young inequality
and does not apply directly to $(\one_{E_\nu})$ in $L^p$ norm.
If $\widehat{f}=\widehat{g}+\widehat{h}$ where $\norm{\widehat{h}}_q$ is small,
one cannot conclude that $\norm{h}_p$ is small.}
that the sequence of functions $\widehat{\one_{E_\nu}}$ is precompact in $L^q(\reals^d)$. 
Therefore there exists a subsequence, which we continue to denote by $(E_\nu)$,
such that $(\widehat{E_\nu})$ converges to some limit $g\in L^q$ in $L^q$ norm.
 
We claim that this subsequence $(\one_{E_\nu})$ converges in $L^p$,
or equivalently, $|E_\mu\symdif E_\nu|\to 0$ as $\mu,\nu\to\infty$.
This is proved as was done for the corresponding situation in \cite{christHY}.
Define $f_\nu = \one_{E_\nu}$. Then
$\norm{\widehat{f_\nu}}_q\to\bestAqd$. Therefore the function
$g = \lim_{\nu\to\infty} \widehat{f_\nu}$ satisfies $\norm{g}_q=\bestAqd$.
Define
\[h_{\mu,\nu} =\tfrac12 (f_\nu+f_\mu) 
= \one_{E_\mu\cap E_\nu} + \tfrac12 \one_{E_\mu\symdif E_\nu}.\]
Then 
$\widehat{h_{\mu,\nu}} = \tfrac12\widehat{f_\nu}+\tfrac12\widehat{f_\mu}\to g$
in $L^q$ norm as $\mu,\nu\to\infty$. Therefore
$\norm{\widehat{h_{\mu,\nu}}}_q\to\bestAqd$,
so $\liminf_{\mu,\nu\to\infty} \norm{h_{\mu,\nu}}_{p,1} \ge 1$
since $\bestAqd$ is the optimal constant in the inequality.
Since $\norm{h_{\mu,\nu}}_p\le 1$,
$\lim_{\mu,\nu\to\infty} \norm{\tfrac12 f_\mu+\tfrac12 f_\nu}_p=1$.

The exponent $p$ lies in $(1,2)$.
$f_\nu$ lies in the unit ball of $L^p$ for all $\nu$.
From the uniform strict convexity of this unit ball it follows that 
$\lim_{\mu,\nu\to\infty} \norm{f_\mu-f_\nu}_p=0$.
That is, $|E_\mu\symdif E_\nu|\to 0$ as $\mu,\nu\to\infty$.
\end{proof}

\end{document}